\documentclass[11pt]{article}
\usepackage{amsmath,amsthm,amssymb,mathtools,bm,mathrsfs,microtype}
\usepackage{graphicx}
\usepackage{enumitem}
\usepackage{stmaryrd}
\usepackage{aop_like_compact} 
\usepackage[colorlinks,citecolor=blue,linkcolor=blue,urlcolor=blue]{hyperref}
\theoremstyle{plain}
\newtheorem{theorem}{Theorem}[section]
\newtheorem{lemma}[theorem]{Lemma}
\newtheorem{prop}[theorem]{Proposition}
\newtheorem{coro}[theorem]{Corollary}
\theoremstyle{definition}
\newtheorem{definition}[theorem]{Definition}

\newtheorem*{openproblem}{Open problem}
\numberwithin{equation}{section}

\DeclareMathOperator{\dive}{div}

\DeclareMathOperator{\curl}{curl}
\DeclareMathOperator{\lip}{Lip}
\DeclareMathOperator{\var}{var}

\DeclareMathOperator{\loc}{loc}
\DeclareMathOperator{\supp}{supp}

\DeclareMathOperator{\dist}{dist}

\DeclareMathOperator{\rem}{rem}

\DeclareMathOperator{\app}{app}
\DeclareMathOperator{\rec}{rec}
\DeclareMathOperator{\sq}{sq}
\DeclareMathOperator{\inv}{inv}

\newcommand{\ssum}{\mathop{\textstyle\sum}\nolimits}

\newcommand{\fa}{{\mathfrak a}}
\newcommand{\fq}{{\mathfrak q}}
\newcommand{\pp}{{\mathfrak p}}
\newcommand{\PPPPP}{{\mathfrak P}}
\newcommand{\fz}{{\mathfrak z}}

\newcommand{\fn}{{\mathfrak n}}

\newcommand{\FM}{{\mathfrak M}}

\newcommand{\FI}{{\mathfrak I}}
\newcommand{\XXX}{{\mathfrak X}}
\newcommand{\fu}{{\mathfrak u}}
\newcommand{\fv}{{\mathfrak v}}

\newcommand{\FT}{{\mathfrak T}}

\newcommand{\ff}{{\mathfrak f}}
\newcommand{\LLL}{{\mathfrak L}}

\newcommand{\FJ}{{\mathfrak J}}

\newcommand{\fw}{{\mathfrak w}}

\newcommand{\fc}{{\mathfrak c}}
\newcommand{\FB}{{\mathfrak B}}

\newcommand{\fg}{{\mathfrak g}}

\newcommand{\E}{{\mathbb E}}

\newcommand{\I}{{\mathbb I}}

\newcommand{\C}{{\mathbb C}}
\newcommand{\R}{{\mathbb R}}
\newcommand{\Z}{{\mathbb Z}}

\newcommand{\PP}{{\mathbb P}}

\newcommand{\PPPP}{{\mathcal P}}
\newcommand{\LL}{{\mathcal L}}

\newcommand{\JJ}{{\mathcal J}}
\newcommand{\EE}{{\mathcal E}}
\newcommand{\FF}{{\mathcal F}}
\newcommand{\QQQ}{{\mathcal Q}}
\newcommand{\ZZ}{{\mathcal Z}}
\newcommand{\V}{{\mathcal V}}
\newcommand{\OO}{\mathcal{O}}
\newcommand{\RR}{{\mathcal R}}
\newcommand{\CC}{{\mathcal C}}
\newcommand{\DDD}{{\mathcal D}}

\newcommand{\KK}{{\mathcal K}}
\newcommand{\MM}{{\mathcal M}}

\newcommand{\II}{{\mathcal I}}

\newcommand{\YYY}{{\mathcal Y}}

\newcommand{\SSSS}{{\mathcal S}}
\newcommand{\GG}{{\mathcal G}}
\newcommand{\BB}{{\mathcal B}}
\newcommand{\A}{{\mathcal A}}

\newcommand{\KKK}{{\mathscr K}}
\newcommand{\CCC}{{\mathscr C}}
\newcommand{\LLLL}{{\mathscr L}}
\newcommand{\MMM}{{\mathscr M}}

\newcommand{\ddd}{{\mathrm d}}
\newcommand{\PPP}{{\mathrm P}}
\newcommand{\QQQQQ}{{\mathrm Q}}

\newcommand{\tg}{{\tilde g}}

\newcommand{\tv}{{\tilde v}}
\newcommand{\tu}{{\tilde u}}

\newcommand{\hv}{{\hat v}}

\newcommand{\bfPhi}{{\mathbf \Phi}}
\newcommand{\bfPsi}{{\mathbf \Psi}}
\newcommand{\uu}{{ \boldsymbol{u} }}

\title{Polynomial mixing for infinite-energy solutions of the white-forced Navier–Stokes equation in an infinite pipe}

\begin{aug}
\author[A]{\fnms{Meng}~\snm{Zhao}\ead[label=e1]{mzhao1009@gmail.com}}
\address[A]{Department of Mathematics, CY Cergy Paris University, CNRS UMR 8088, 2 avenue Adolphe Chauvin, 95302 Cergy-Pontoise, France\printead[presep={,\ }]{e1}}
\end{aug}

\date{\today}

\begin{document}
\maketitle
\begin{abstract}
We study the long-time behavior of the infinite-energy solutions to the white-forced Navier--Stokes equation in an infinite pipe $\R\times (-1,1)$ with no-slip boundary conditions. The equation is globally well-posed in uniformly local spaces satisfying the divergence-free condition. Under suitable degeneracy and spatial decay assumptions on the noise, we prove the uniqueness of the invariant measure and establish polynomial mixing in the dual-Lipschitz metric. This provides an affirmative answer to Open Problem 1 posed by Kuksin and Shirikyan \cite{KS17} in the case of zero cross-sectional flux.


\end{abstract}
\tableofcontents
\section{Introduction}
We study the mixing properties of the infinite-energy solutions to the Navier--Stokes (NS, for short) equation perturbed by an additive white-in-time noise in an infinite pipe $\R\times (-1,1)$ with the no-slip boundary conditions
\begin{align}
    \label{INTRO1}
    \begin{cases}
        \partial_tu+ (u\cdot \nabla )u+\nabla \pp-\nu\Delta u=h+\eta,\qquad (x,y)\in D:=\R\times (-1,1),\\
        \dive u=0,\\
        u|_{t=0}=u_0,\qquad u|_{\partial D}=0,
    \end{cases}
\end{align}
where $u=(u_1,u_2)$ is the velocity field, $\pp$ is the pressure, and $\nu$ is the kinematic viscosity. Define the space 
\begin{align}
    \label{INTRO1-1}H:=\{u\in L^2(D)| \dive u=0,\ (u\cdot n)|_{\partial D}=0\}
\end{align}
endowed with the usual $L^2$-norm $\|\cdot\|$. The driving process is composed of two parts: $h=h(x)\in H$ is a deterministic forcing term and $\eta$ is a white-in-time noise in the form
\begin{align}\label{INTRO2}
    \eta(t,x):=\partial_t W(t,x),\qquad W(t,x):=\ssum_{j\ge 1}b_je_j(x)\beta_j(t),
\end{align}
where $\{e_j\}_{j\ge 1}$ is an orthonormal basis in $H$, $\{b_j\}_{j\ge 1}$ is a sequence of real numbers satisfying $\FB_0:=\ssum_{j\ge 1}b_j^2<\infty$, and $\{\beta_j\}_{j\ge 1}$ is a sequence of mutually independent standard Brownian motions defined on a filtered probability space $(\Omega,\FF,\FF_t,\PP)$ satisfying the usual conditions; see Definition 2.25 in \cite{KS91}.

\subsection{The main results}
\subsubsection{Global well-posedness in the uniformly local spaces}
For $1\le p< \infty$, we define the uniform-local Lebesgue space 
\begin{align}
    \label{INTRO3}
    L^p_b:=\left\{u\in L^p_{\loc}\bigg| \|u\|_{L^p_b}:=\sup_{\fz\in\R}\|u\|_{L^p(D_{\fz})}<\infty\right\},\qquad D_{\fz}:=(\fz,\fz+1)\times (-1,1),
\end{align}
and the uniform-local velocity class
\begin{align}\label{INTRO4}
    L^p_{b,\dive}:=\left\{u\in L^p_b\bigg| \dive u=0,\ (u\cdot n)|_{\partial D}=0,\ \int_{-1}^1u_1\ddd x_2=0\mbox{ for a.e. $x_1$}\right\}
\end{align}
with the norm inherited from $L^p_b$. For $p=2$, we write 
\begin{align}
    \label{INTRO4-1}H_b:=L^2_{b,\dive}.
\end{align} 
The condition on the cross-sectional flux $\int_{-1}^1u_1 \ddd x_2$ is prescribed, in order to ensure the uniqueness of the infinite-energy solutions; see Remark~5.3 in \cite{Zel07}. In this paper, we focus on the case of zero cross-sectional flux (zero-flux, for short). Introduce the phase space
\begin{align}\label{INTRO6}
    X:=(H_b,d_{\loc}),
\end{align}
where
\begin{align}\label{INTRO5}
    d_{\loc}(u,v):=\ssum_{k\ge 1}\frac{1\wedge\|u-v\|_{L^2(I_k)}}{2^k},\qquad I_k:=(-k,k)\times (-1,1).
\end{align}
Throughout, $X$ is equipped with its Borel $\sigma$-algebra
$\mathcal B(X)$. It should be noted that although $X$ is not a Polish space, it is a standard measurable space; see Definition A.2.1 in \cite{KS12}. Indeed,
\[X=\cup_{m\ge 1}\cap_{k\in\Z}\left\{u\in H_{\loc}| \|u\|_{L^2(D_k)}\le m\right\},\]
where the space $H_{\loc}$ is defined in \eqref{notation1}. Notice that each set on the right-hand side is a closed subset of the Polish space $(H_{\loc},d_{\loc})$, so $X$ is a Borel subset of $H_{\loc}$ and thus standard measurable. As a result, regular conditional probabilities are well-defined, which provides the measure-theoretic framework for studying $X$-valued Markov families.\footnote{See Section~1.3 in \cite{KS12} for the definition of a Markov family.}

The notion of infinite-energy solutions is stated as follows.
\begin{definition}
A progressively measurable process $u:\Omega\times [0,\infty)\times D\to \R^2$ is said to be an infinite-energy solution of \eqref{INTRO1}, if the following conditions are satisfied.
    \begin{enumerate}
        \item $u(\cdot)\in C([0,\infty);X)$ holds $\PP$-almost surely.
        \item For any $T>0$, it follows that
        \[\sup_{t\in[0,T]}\|u(t)\|^2_{L^2_b}+\sup_{\fz\in\R}\int_0^T\|\nabla u(t)\|^2_{L^2(D_{\fz})}\ddd t<\infty,\qquad 
        \PP\mbox{-a.s.}\]
        \item $u|_{\partial D}=0$ holds for almost every $(\omega,t)$.
        \item For any test function $\varphi\in C_0^\infty\cap H$, it follows that
        \begin{align*}
            \langle u(t),\varphi\rangle-\langle u_0,\varphi\rangle\!-\!\int_0^t\!\left(\langle u\otimes u,\nabla \varphi\rangle+\nu\langle u,\Delta \varphi\rangle+\langle h,\varphi\rangle\right)\ddd s\!=\!\ssum_{j\ge 1}\!\int_0^t b_j \langle e_j,\varphi\rangle \ddd\beta_j
        \end{align*}
        for any $t\ge 0$ $\PP$-almost surely.
    \end{enumerate}
    For the deterministic NS equation, infinite-energy solutions are defined analogously.
\end{definition}

Now we present the well-posedness result for equation \eqref{INTRO1} in the infinite-energy space $X$; see Appendix \ref{GWP} for its proof. 
\begin{theorem}\label{MT1}
    Suppose that $h\in H$, 
    \begin{align}
        \label{INTRO7}
        \FB_1:=\ssum_{j\ge 1} b_j^2\|e_j\|^2_{H^1}<\infty,
    \end{align}
    and the orthonormal basis $\{e_j\}_{j\ge 1}$ satisfies \eqref{INTRO9}. Then, for any $u_0\in X$, there is a unique global infinite-energy solution of \eqref{INTRO1}. Moreover, the following statements hold.
    \begin{enumerate}
        \item For any $\delta,T>0$, $\sup_{t\in[\delta,T]}\|u(t)\|_{\infty}<\infty$ holds $\PP$-almost surely.
        \item If $u_0\in X\cap L^\infty$, then the following decomposition holds
        \[u(t)=w(t)+u_{\rem}(t),\]
        where $w(t)$ is an infinite-energy solution of the homogeneous NS equation
        \begin{align}
            \label{INTRO8}
            \begin{cases}
            \partial_t w+(w\cdot \nabla)w+\nabla \fq-\nu\Delta w=0,\\
            \dive w=0,\\
            w|_{t=0}= u_0,\qquad w|_{\partial D}=0,
            \end{cases}
        \end{align}
        and $u_{\rem}(t)$ satisfies the path regularity
        \begin{align}
            \label{INTRO8-0}u_{\rem}(\cdot)\in C([0,\infty);H)\cap L^2_{\loc}(\R_+;H^1_0\cap H)
        \end{align}
        $\PP$-almost surely.
        \item The infinite-energy solutions define a Markov family $(u_t,\PP^X_u)$.
    \end{enumerate}
\end{theorem}
\subsubsection{Mixing for the finite-energy solutions}
Next, we present a mixing result for the finite-energy solutions, which is an intermediate step towards the infinite-energy case. To formulate the result, we introduce the so-called admissible weight functions.
\begin{definition}
 A smooth function $\fw(x):\R\to [1,+\infty)$ is called an admissible weight, if 
 \begin{align}
     \label{INTRO8-1}\lim_{x\to\infty}\fw(x)=\lim_{x\to-\infty}\fw(x)=+\infty.
 \end{align}
\end{definition}
Throughout the paper, we shall fix an admissible weight $\fw$ and assume that the orthonormal basis $\{e_j\}_{j\ge 1}$ in \eqref{INTRO2} satisfies
\begin{align}
    \label{INTRO9}e_j\in H^2\cap H^1_0\cap H,\qquad \fw e_j \in L^2.
    \end{align}
Such an orthonormal basis can be constructed through standard density and Gram--Schmidt arguments; see Section 1.1 in \cite{NZ24-1}. The mixing result of the finite-energy solutions is presented below.
\begin{theorem}\label{MT2} For any $\RR>0$, there exists a positive integer $N=N(\RR,\nu,h)$ such that if
    \begin{align}
        \label{INTRO10}
        \FB_\fw:=\ssum_{j\ge 1} b_j^2(\|\fw e_j\|^2+\|e_j\|^2_{H^1})\le \RR^2,
    \end{align}
    and
    \begin{align}
        \label{INTRO11}b_j\neq 0,\qquad \forall j=1,\ldots,N,
    \end{align}
    then the stochastic NS equation \eqref{INTRO1} has a unique invariant measure $\mu\in \PPPP(H)$, where $\PPPP(H)$ denotes the space of all probability measures on $H$ endowed with the weak topology. Moreover, the following statements hold.
    \begin{enumerate}
    \item The invariant measure $\mu$ satisfies
    \begin{align}
        \label{INTRO12}\int_{H}\|v\|^2\mu(\ddd v)<\infty, \qquad \mu(H\cap L^\infty)=1.
    \end{align} 
    \item For any $q\ge 1$, it follows that 
        \begin{align}
            \label{INTRO13}\left|\E (f(u(t)))-\int_{H}f(u)\mu(\ddd u)\right|\le C_q\|f\|_{\lip}\left(1+t\right)^{-q}\left(1+ \|u_0\|^2\right)
        \end{align}
        for any $t\ge0$, $u_0\in H$, and bounded Lipschitz-continuous function $f:H\to \mathbb{R}$, where $C_q$ depends on $q,\nu,h$ and $\{b_j\}_{j\ge 1}$.
    \item In particular, for any probability measure $\lambda\in\PPPP(H)$,
    \begin{align}
        \label{INTRO14}\lim_{t\to\infty}\|\DDD(u(t;\lambda))-\mu\|^*_{\lip(H)}=0,
    \end{align}
    where $\|\cdot\|_{\lip(H)}^*$ is the dual-Lipschitz metric defined in \eqref{notation2} and $\DDD(u(t;\lambda))$ denotes the distribution of the solution $u(t)$ issued from the initial data $u_0$ distributed as $\lambda$.
    \end{enumerate}
    \end{theorem}
    Several remarks are in order. First, no size restriction is imposed on the viscosity $\nu$. Second, conditions \eqref{INTRO10} and \eqref{INTRO11} are compatible. Indeed, fix $\RR>0$ and let $N_\RR$ be the integer given by Theorem~\ref{MT2}. Given any sequence $\{b_j\}_{j\ge1}$ satisfying \eqref{INTRO10}, replace each vanishing coefficient among $b_1,\ldots,b_N$ by a nonzero number. Since only finitely many coefficients are modified, the weighted sum in \eqref{INTRO10} remains finite. Multiplying the resulting sequence by a sufficiently small positive constant then ensures that \eqref{INTRO10} holds, while preserving the nondegeneracy condition \eqref{INTRO11}.
    
    The fact that the number $N$ in \eqref{INTRO11} depends on the parameter $\RR$, instead of each noise coefficient $b_j$, plays an essential role in the above adding and rescaling argument. Such parameter dependence is ensured by the admissible weight function $\fw$. Indeed, it is shown in Lemma \ref{lemmaMFE5} that the set of coefficients $\{b_j\}_{j\ge 1}$ satisfying \eqref{INTRO10} is a compact subset of the sequence space $\ell^2$. Consequently, we are able to choose a uniform integer $N$ on this set; see Section \ref{UFDS}. 
    
    For the non-degenerate case, i.e., $b_j\neq 0$ for all $j\ge 1$, the above parameter dependence is not required, and therefore the assumption can be simplified: Theorem \ref{MT2} holds for any orthonormal basis $\{e_j\}_{j\ge 1}$ in $H$ satisfying $e_j\in H^1_0\cap H$ for all $j\ge1$ and any noise coefficients $\{b_j\}_{j\ge 1}$ satisfying $\FB_1<\infty.$
    
    As we shall see in Sections \ref{FDS} and \ref{MFE}, the proof of Theorem \ref{MT2} does not rely essentially on the particular shape of the underlying domain $D$. In fact, for any unbounded Poincar\'e-type domain, one can define an admissible weight function $\fw: D\to [1,+\infty)$ by replacing the condition \eqref{INTRO8-1} with 
    \[\fw(x,y)\to\infty, \qquad \mbox{for any $(x,y)\in D$ satisfying $(x,y)\to\infty$}.\]
    This, together with the proof presented in Sections \ref{FDS} and \ref{MFE}, yields polynomial mixing for finite-energy solutions of \eqref{INTRO1} in arbitrary Poincar\'e-type domain.

    Finally, it is noteworthy that the framework established in the previous paper \cite{NZ24-1} by Nersesyan and the author does not apply to the current situation, due to the presence of a physical boundary. Indeed, the Foia\c{s}--Prodi estimate established in Proposition 4.3 therein relies on a linear growth of the parabolic energy of the vorticity with high probability. Under the current no-slip boundary conditions, the boundary terms prevent a direct application of these vorticity estimates; see \cite{KV14} for the difficulties in obtaining strong-norm moment bounds. Thus, Theorem \ref{MT2} is not only an intermediate step towards the infinite-energy case, but also constitutes a substantial improvement of \cite{NZ24-1}.
    
\subsubsection{Mixing for the infinite-energy solutions}
    Finally, we present the mixing result for the infinite-energy solutions. 
    \begin{theorem}\label{MT3}
   For any $\RR>0$, there exists a positive integer $N=N(\RR,\nu,h)$ such that if \eqref{INTRO10} and \eqref{INTRO11} are satisfied, then the following statements hold for the invariant measure $\mu$ constructed in Theorem \ref{MT2}.
    \begin{enumerate}
    \item For any $q\ge 1$ and $R>0$, it follows that
    \begin{align}
        \label{INTRO15}\sup_{\|u_0\|_{L^2_b}\le R}\left|\E (f(u(t)))-\int_{X}f(u)\mu(\ddd  u)\right|\le C_{q,R} \|f\|_{\lip}\left(1+t\right)^{-q}
    \end{align}
    for any $t\ge0$ and bounded Lipschitz-continuous function $f: X\to \mathbb{R}$, where $C_{q,R}$ depends on $q,R,\nu,h,\{b_j\}_{j\ge 1}$.
    \item The probability measure $\mu$ is the unique invariant measure of the stochastic NS equation \eqref{INTRO1} in $\PPPP(X)$, where $\PPPP(X)$ denotes the space of all probability measures on $X$ endowed with the weak topology.
    \item In particular, for any probability measure $\lambda\in\PPPP(X)$,
    \begin{align}
        \label{INTRO16}\lim_{t\to\infty}\|\DDD(u(t;\lambda))-\mu\|^*_{\lip(X)}=0,
    \end{align}
    where $\DDD(u(t;\lambda))$ denotes the distribution of the solution $u(t)$ issued from the initial data $u_0$ distributed as $\lambda$. 
    \end{enumerate}
\end{theorem}
We make several comments below. In the review paper \cite{KS17}, Kuksin and Shirikyan posed an open problem concerning the mixing properties of the infinite-energy pipe flow, which is reformulated below.
\begin{openproblem}
    Let us consider the stochastic Navier--Stokes equation \eqref{INTRO1}. This equation is well-posed in $L^\infty$ spaces on $D$, with no decaying conditions at infinity; see \cite{AM05,AZ14}. Moreover, it is dissipative and possesses a global attractor in the deterministic setting. It follows that, at least in the case of a bounded stochastic forcing, the random dynamics has an invariant measure. A challenging problem is to prove its uniqueness and mixing.
\end{openproblem}
 Notice that
\[L^\infty_{\dive}:=\left\{u\in L^\infty\bigg| \dive u=0,\ (u\cdot n)|_{\partial D}=0,\ \int_{-1}^1u_1\ddd x_2=0\mbox{ for a.e. $x_1$}\right\}\subset X,\]
so each probability measure on $L^\infty_{\dive}$ can be extended to $\PPPP(X)$. Consequently, Theorem \ref{MT3} provides an affirmative answer to the above open problem in the zero-flux case. The case of non-zero flux requires additional treatments, which shall be addressed in a forthcoming paper. 

The conditions \eqref{INTRO10} and \eqref{INTRO11} for the degeneracy of the noise are posed for the same reason as discussed in the finite-energy case. Similarly, for the non-degenerate noise, the assumption can be simplified: Theorem \ref{MT3} holds for any orthonormal basis $\{e_j\}_{j\ge 1}$ in $H$ such that $e_j\in H^1_0\cap H$ for all $j\ge1$ and any non-degenerate noise coefficients $\{b_j\}_{j\ge 1}$ satisfying $\FB_1<\infty.$

\subsection{Literature review and the strategy of the proof}\label{LRSP}
\subsubsection{Literature review}
Since the 1990s, substantial progress has been made in understanding the ergodic properties of stochastic dissipative PDEs. The 2D NS equation has played a central role because of its importance in fluid dynamics and mathematical turbulence. Foundational contributions to its ergodic theory include \cite{FM95, EMS01, KS02, HM06, HM08}; comprehensive accounts and further references can be found in the monograph \cite{KS12} and the surveys \cite{KS17, D13}. Much of the classical theory concerns bounded domains or spatially periodic settings, where compact Sobolev embeddings and the discrete spectrum of the Stokes operator provide effective tools for studying invariant measures and mixing properties.

The analysis of ergodic properties in unbounded domains is substantially more delicate, as the compactness and discrete spectral structure are no longer available. In this setting, even the existence of an invariant measure becomes a nontrivial issue, which has been addressed in \cite{BL06, BMO17}. Concerning the uniqueness and mixing, Nersesyan \cite{N22} established uniqueness of the invariant measure and exponential mixing for the NS equation in unbounded Poincar\'e domains driven by bounded random forcing. More recently, Nersesyan and the author \cite{NZ24-1} proved uniqueness and polynomial mixing for the damped NS equation on $\mathbb{R}^2$ driven by additive white-in-time noise. As pointed out in the previous section, their proof is based on a weighted version of the Foia\c{s}--Prodi estimate and thus requires a linear growth of the parabolic energy of the vorticity with high probability. Such a probabilistic growth estimate has not yet been established, due to the presence of physical boundaries, therefore leaving additional difficulties in the study of mixing properties for pipe flows.

The above existence and uniqueness results in unbounded domains concern finite-energy solutions. In an infinite pipe, this restriction excludes every nontrivial shear flow $(U(t,y),0)$, whose total kinetic energy is infinite. This motivates the study of infinite-energy solutions, which provide a natural framework for describing these physically important pipe flows. 

In the infinite-energy setting, the existence of an invariant measure was established by Eckmann--Hairer \cite{EH01} and Rougemont \cite{Rou02} for stochastic Ginzburg--Landau equations with localized noise and homogeneous noise, respectively. As for the uniqueness and mixing, to the best of our knowledge, a substantial body of work beyond the finite-energy setting concerns Burgers equations and scalar conservation laws on $\R$. For the inviscid Burgers equation on $\mathbb R$, invariant measures and convergence to equilibrium in suitable classes with prescribed mean velocity were established under spatially homogeneous Poisson noise in \cite{BCK14}. Analogous results for the viscous equation were obtained in \cite{BL19,DGR21} under spatially homogeneous random kick forcing and spatially homogeneous white-in-time Gaussian potential forcing, respectively. For more general viscous scalar conservation laws, existence and uniqueness of ergodic, spatially homogeneous invariant measures for each mean in a certain unbounded set were proved in \cite{DDGLR23}. These results exploit structural properties specific to scalar equations, and extending such results to infinite-energy NS flows presents significant challenges.

\subsubsection{The strategy of the proof}
\subsubsection*{Mixing for the finite-energy solutions}
The proof of Theorem \ref{MT2} is based on the asymptotic compactness property of the finite-energy solutions. This property was established by Rosa \cite{Ros98} for the deterministic 2D NS equation on unbounded Poincar\'e domains, in order to prove the existence of a global attractor. This approach was subsequently extended to the stochastic setting by Brzeźniak and his collaborators \cite{BL06,BCLLLR13}, who employed asymptotic compactness to establish the existence of invariant measures and random attractors in unbounded Poincar\'e domains. More recently, this property was introduced by Nersesyan \cite{N22} to the study of mixing for the NS equation with bounded noise in unbounded Poincaré domains. Beyond the NS setting, Xiang and his collaborators \cite{LWXZZ24,CXZ26,CXZZ26} used control-theoretic methods combined with the asymptotic compactness property to establish exponential mixing for randomly forced dispersive equations in bounded domains, including the nonlinear wave and Schrödinger equations.

Although the precise formulation of asymptotic compactness varies across the literature, these formulations describe a common notion of compactness in the long-time dynamics. In the present paper, it refers specifically to the property that the solutions approach a compact set with high probability as $t\to\infty$; see Proposition \ref{propMFE3}.

To understand the role of asymptotic compactness in the proof, we recall the coupling framework used to establish mixing for Markov families; see Chapter~3 in \cite{KS12}. Mixing of a Markov family $(u_t,\PP_u)$ can be established by constructing a coupling $(\uu_t,\PP_{\uu})$ of two copies, $(u_t,\PP_u)$ and $(u'_t,\PP_{u'})$, that satisfies suitable conditions; see Proposition~\ref{propMFE5}. A key step in verifying these conditions is to establish a Foia\c{s}--Prodi estimate, which quantifies the stabilization of the NS dynamics under a suitable finite-dimensional perturbation. In the unbounded-domain setting, however, the classical Foia\c{s}--Prodi argument is no longer directly applicable, due to the loss of compactness of Sobolev embeddings and the presence of continuous spectrum for the Laplacian. In \cite{NZ24-1}, we overcame this difficulty on $\R^2$ by introducing a suitable space-time weight function. As discussed above, this approach does not extend to the present setting, because of the difficulties in controlling the vorticity in the presence of a physical boundary.

Rather than refining the weighted approach used in $\R^2$, we establish a new Foia\c{s}--Prodi estimate that can be directly extended to other unbounded Poincar\'e domains. Let $\{u(t)\}_{t\ge0}$ be a finite-energy solution of \eqref{INTRO1}, and let $\{v(t)\}_{t\ge0}$ solve the perturbed stochastic NS equation \eqref{MFE2}, obtained by adding the finite-dimensional feedback term $a\PPP_N(u-v)$. We observe that though the high-mode part $(I-\PPP_N)u$ cannot be bounded by $\epsilon\|\nabla u\|$ with $\epsilon\ll1$ in unbounded domains, a localized version still holds; see \eqref{MFE5}. This, together with the standard energy estimate, yields the Foia\c{s}--Prodi estimate \eqref{MFE3} in the unbounded-domain setting. As a result, 
\[\|u(t)-v(t)\|\lesssim e^{-ct},\]
provided that one can find a finite set $K\subset C_0^\infty\cap H$ such that
\begin{align}
    \label{SP0}\sup_{t\ge0}\left(\int_0^t\dist_V(u(s),K)^2\ddd s-\epsilon t\right)\le \rho,\qquad V:=H^1_0\cap H
\end{align}
with high probability. Such an estimate is obtained by using the asymptotic compactness property of the stochastic NS equation. Indeed, let $\{\fu_m\}_{m\ge 1}\subset C_0^\infty\cap H$ be a dense sequence in $V$, and write $K_m:=\{0,\fu_1,\ldots,\fu_m\}$. It is proved that the sequence
\[\Phi_m(u):=\E_u\int_0^1 \dist_V(u_t,K_m)^2\ddd t\]
is uniformly bounded and equicontinuous on bounded subsets of $H$, and for any $u\in H$, $\lim_{m\to\infty}\Phi_m(u)=0$. On the other hand, the asymptotic compactness property ensures the existence of a compact set $\KK$ that attracts the solution $u(t)$ with high probability, as $t\to\infty$. Consequently, there is $\delta>0$ and $m\ge 1$ such that
\[\sup_{\dist_H(u,\KK)\le \delta }\E_u\int_0^1 \dist_V(u_t,K_m)^2\ddd t<\frac{\epsilon}{2}.\]
Notice that $u_{\fn_0}$ belongs to $\{u\in H| \dist_H(u,\KK)\le \delta \}$ with high probability for sufficiently large $\fn_0$, due to  Proposition~\ref{propMFE3}. Therefore, an application of the Markov property ensures the small growth of the averaged distance functional, that is,
\begin{align}
    \label{SP0-1}\sup_{\|u\|\le R}\E_u\int_{\fn_0}^{\fn_0+1} \dist_V(u_t,K_m)^2\ddd t<\epsilon.
\end{align}
The long-time growth estimate \eqref{SP0} is established through a standard argument, using the Markov property and martingale techniques. Roughly speaking, by applying \eqref{SP0-1} combined with the Markov property and the moment estimate \eqref{A6}, one can obtain
\begin{align}
    \label{SP0-2}\E_u\int_{0}^{M} \dist_V(u_t,K_m)^2\ddd t\lesssim \epsilon M+\V(u)
\end{align}
for any $M\ge \fn_0$ and $u\in H$, where $\V(u)$ is a Lyapunov function satisfying 
\[\E_u\V(u_M)\le e^{-cM}\V(u)+C.\]
For simplicity, we temporarily assume $\V\equiv 0$ in \eqref{SP0-2}, since the general case can be handled with the help of the above Lyapunov estimate; see the proof of Proposition \ref{propMFE2}. For any $j\ge 1$, 
\[\int_{0}^{jM} \dist_V(u_t,K_m)^2\ddd t=\ssum_{k\le j-1} d_k,\qquad d_k:=\int_{kM}^{(k+1)M} \dist_V(u_t,K_m)^2\ddd t,\]
where, by the Markov property and \eqref{SP0-2}, it follows that
\begin{align}
    \label{SP0-3}\E_u(d_k|\FF_{kM})\lesssim \epsilon M.
\end{align}
Consequently, we get
\[\int_{0}^{jM} \dist_V(u_t,K_m)^2\ddd t\le j\epsilon M+\ssum_{k\le j-1} (d_k-\E_u(d_k|\FF_{kM})).\]
Notice that the second term on the right-hand side is a martingale satisfying polynomial moment estimates ensured by the relation $\dist_V(u,K_m)\le \|\nabla u\|$ and Proposition \ref{propA1}. By the Chebyshev inequality, the long-time growth estimate \eqref{SP0} holds. This yields the finite-dimensional stabilization mechanism of the 2D NS equation on any unbounded Poincar\'e domain, which implies Theorem \ref{MT2} via the classical coupling framework.
 
\subsubsection*{Mixing for the infinite-energy solutions}
In the infinite-energy case, the key observation is the following decomposition
\[u(t)=w(t)+u_{\rem}(t),\]
where $w(t)$ is an infinite-energy solution of the homogeneous NS equation \eqref{INTRO8} issued from $u_0$, and the remainder $u_{\rem}$ satisfies the stochastic perturbed NS equation around the background flow $w$:
\begin{align}
    \label{SP1}
    \begin{cases}
    \partial_t u_{\rem}+(u_{\rem}\cdot\nabla )u_{\rem}+(w\cdot\nabla)u_{\rem}+(u_{\rem}\cdot\nabla )w+\nabla \pp_{\rem}-\nu\Delta u_{\rem}=h+\eta,\\
    \dive u_{\rem}=0,\\
    u_{\rem}|_{t=0}=0,\qquad u_{\rem}|_{\partial D}=0.
    \end{cases}
    \raisetag{25pt}
\end{align}

These two parts are treated separately. For the deterministic part $w(t)$, we aim to establish the following decay estimate
\begin{align}
    \label{SP2}\int_0^\infty \|w(s)\|^2_\infty \ddd s<\infty,\qquad \|w(t)\|^2_\infty\lesssim e^{-\alpha_1t},\qquad \forall t\ge1.
\end{align}
Although the decay and smoothing property of infinite-energy solutions have been addressed in \cite{AZ14}, we notice that their estimates are not sufficient for our purpose. For example, it was shown in Section 6 therein that
\[\|w(t)\|_{L^2_b}\lesssim e^{-ct}+1,\qquad \forall t\ge 0,\]
and 
\[\|w(t)\|_{W^{2,2}_b}\lesssim t^{-1},\qquad t\in(0,1].\]
By interpolation, these estimates imply $\|w(t)\|^2_\infty\lesssim e^{-c t}+1$ for $t\ge 1$ and $\|w(t)\|^2_\infty\lesssim t^{-1}$ for $t\in (0,1]$, which is insufficient to obtain \eqref{SP2}. To this end, we first prove that the Stokes operator generates an analytic semigroup in a weighted Lebesgue space. As a result, we obtain a linear estimate for the Stokes semigroup in the uniformly local velocity class $L^p_{b,\dive}$ for all $p\in(1,\infty)$; see Proposition \ref{propSTOKES}. Together with a fixed point argument and interpolation, this ensures the integrability of $\|w(t)\|_\infty^2$ near $t=0$. On the other hand, by a refined weighted energy estimate, we obtain the exponential decay of $\|w(t)\|_\infty$ to $0$, as $t\to\infty$. Consequently, the desired estimate \eqref{SP2} holds. 

As for the remainder $u_{\rem}$, since $\|w(t)\|_\infty$ decays exponentially, the perturbation terms $(w\cdot\nabla)u_{\rem}$ and $(u_{\rem}\cdot\nabla )w$ in \eqref{SP1}$_1$ become asymptotically small in the long-time dynamics. Therefore, $u_{\rem}$ behaves asymptotically as a finite-energy solution of the stochastic NS equation. To be precise, in Section \ref{MIE}, we shall construct a family of finite-energy solutions $u_{\app}(t;T)$ that approximates the remainder $u_{\rem}(t)$; see Proposition \ref{propMIE2}. Consequently, by applying Theorem \ref{MT2}, we obtain the convergence of $\DDD (u_{\rem}(t))$ towards the invariant measure $\mu$ obtained in the finite-energy case. Together with \eqref{SP2}, this yields the convergence of $u(t)$ and the uniqueness of the invariant measure of \eqref{INTRO1} in $\PPPP(X)$.

\subsection{Further discussions on the whole-space case}
\label{sec:whole-space}

The framework developed here also applies to the damped NS equation in the whole space
    \begin{align}
        \label{R2}\begin{cases}
    \partial_t u+(u\cdot\nabla)u-\nu\Delta u+au+\nabla p=h+\eta, \qquad x=(x_1,x_2)\in\R^2, 
    \\ \dive u=0,\\
    u|_{t=0}=u_0,
    \end{cases}
    \end{align}
where $a,\nu>0$. In this subsection, we take
\[H:=\{u\in L^2(\R^2;\R^2)| \dive u=0\},\]
and
\[H_b:=\{u\in L^2_{\loc}(\R^2;\R^2)|\dive u=0,\ \|u\|_{L^2_b}<\infty\},\qquad\|u\|_{L^2_b}:=\sup_{x\in\R^2}\|u\|_{L^2(B(x,1))},\]
where $B(x,1):=\{x'\in \R^2||x'-x|<1\}$. The phase space is $X:=(H_b,d_{\loc})$ with $d_{\loc}$ defined in \eqref{INTRO5} with $I_k$ replaced by the rectangle $(-k,k)^2$. No spatial decay or zero-mean condition is imposed on the initial velocity and the phase space. The pressure gradient is fixed by the canonical uniformly local extension of the usual pressure-gradient operator; see Definition~3.2 and Remark~3.3 in \cite{Zel13}. In particular, no additional spatially constant and time-dependent pressure gradient is allowed.

We adopt the assumptions on the forcing term introduced in \cite{NZ24-1}. More precisely, $\eta:=\partial_t\sum_{j\ge1}b_j e_j\beta_j$, where $\{e_j\}_{j\ge1}$ is a fixed orthonormal basis in $H$ satisfying
\[e_j\in H^2,\qquad \varphi e_j\in H^1,\qquad \varphi(x):=(1+|x|^2)^{1/4}\]
for all $j\ge 1$. The deterministic force and the noise coefficients satisfy
\[h\in H,\qquad \varphi h\in L^2,\qquad\FB_\varphi:=\ssum_{j\ge1}b_j^2\big(\|\varphi e_j\|^2+\|\varphi\curl e_j\|^2\big)<\infty,\]
and $b_j\neq0$ for $j=1,\ldots,N_0$ with $N_0$ given in Main Theorem in \cite{NZ24-1}. The global well-posedness in $X$ and smoothing effect can be established by applying the decomposition strategy in Appendix~\ref{GWP} combined with the whole-space theory in \cite{Zel13}. For $u_0\in H_b\cap L^\infty$, let $u=w+u_{\mathrm{rem}}$, where $w$ solves the homogeneous damped NS equation with $w(0)=u_0$, and $u_{\mathrm{rem}}$ is the finite-energy remainder. The homogeneous estimates in \cite{Zel13}, combined with positive-time smoothing, give
\[\|w(t)\|_\infty\le C(u_0)e^{-ct},\qquad \forall t\ge1.\]
Consequently, by applying the comparison argument in Section \ref{MIE} combined with the mixing result in \cite{NZ24-1} and the smoothing property of \eqref{R2}, one may obtain polynomial mixing under the dual-Lipschitz metric associated with $d_{\loc}$, which gives uniqueness of the invariant measure in $\mathcal P(X)$. This provides an affirmative answer to Open Problem 2 in \cite{KS17}.

 \subsection*{Organization of the paper}
    The paper is organized as follows. Section~2 establishes finite-dimensional stabilization estimates through the asymptotic compactness property of the NS dynamics. This yields polynomial mixing for finite-energy solutions in Section~3. Section~4 derives decay estimates for infinite-energy solutions of the homogeneous NS equation, which are used in Section~5 to establish mixing for infinite-energy solutions. The Appendix collects a priori energy estimates, global well-posedness in uniformly local spaces, estimates for the Stokes semigroup in uniformly local spaces, and proofs of auxiliary results.

\subsection*{Acknowledgements}
    The author was supported by the ANR project DYNACQUS through the grant ANR-24-CE40-5714-01, and by the National Natural Science Foundation of China under Grant Nos. 12331008, 12571156. The author would like to express his sincere gratitude to Professors Armen Shirikyan and Vahagn Nersesyan for helpful comments and valuable suggestions.

\subsubsection*{Notations}
We shall use the following standard notations.
\begin{itemize}   
\item For $1\le p\le \infty$, $L^p:=L^p(D)$ is the classical Lebesgue space with the norm denoted by $\|\cdot\|_p$, $L^p_b$ is the uniform-local Lebesgue space defined in \eqref{INTRO3}, and $L^p_{b,\dive}$ is the uniform-local velocity class introduced in \eqref{INTRO4}. For $p=2$, we write $H_b:=L^2_{b,\dive}$.
\item Let $H$ denote the $L^2$-based velocity class introduced in \eqref{INTRO1-1}. For $p\in (1,\infty)$, we define
\[L^p_{\dive}:=\{u\in L^p(D)| \dive u=0,\ (u\cdot n)|_{\partial D}=0\}.\]
\item For $k\ge 1$ and $p\in[1,\infty]$, $W^{k,p}$ is the classical Sobolev space with the norm denoted by $\|\cdot\|_{k,p}$, $W^{1,p}_0$ is the space of all functions in $W^{1,p}$ with zero trace, and $W^{-1,p}:=(W^{1,p'}_0)^*$ is the dual of $W^{1,p'}_0$ with its norm denoted by $\|\cdot\|_{-1,p}$. Here, $p'$ is the H\"older conjugate exponent of $p$. 
\item For $k\ge 1$ and $p\in[1,\infty)$, we define the uniformly local Sobolev space
\[W^{k,p}_b:=\left\{u\in W^{k,p}_{\loc}\bigg| \|u\|_{W^{k,p}_b}:=\sup_{\fz\in\R}\|u\|_{W^{k,p}(D_{\fz})}<\infty\right\},\quad D_{\fz}:=(\fz,\fz+1)\times (-1,1).\]
The fractional uniformly local Sobolev space $W^{s,p}_b$ is defined in the same way.
\item For $p=2$, we write $H^s:=W^{s,2}$, $H^s_b:=W_b^{s,2}$, and $H^1_0:=W^{1,2}_0$. Define $V:=H^1_0\cap H$ endowed with the norm $\|u\|_V:=\|\nabla u\|.$ Let $V^*$ denote the dual space of $V$.
\item Let $X:=(H_b,d_{\loc})$, where the metric $d_{\loc}$ is defined in \eqref{INTRO5}. Let
\begin{align}
    \label{notation1}H_{\loc}=\left\{u\in L^2_{\loc}\bigg| \dive u=0,\ (u\cdot n)|_{\partial D}=0,\ \int_{-1}^1u_1\ddd x_2=0\mbox{ for a.e. $x_1$}\right\}.
\end{align}
The space $(H_{\loc},d_{\loc})$ is Polish.
\item For a metric space $(\XXX,d_{\XXX})$, $\A\subset \XXX$, and $x\in\XXX$, we define
\[\dist_{\XXX}(x,\A):=\inf_{y\in \A}d_{\XXX}(x,y).\]
\item For a separable metric space $(\XXX,d_{\XXX})$, $\PPPP(\XXX)$ denotes the space of all probability measures on $\XXX$ endowed with the weak topology. For $f\in C_b(\XXX)$ and $\mu\in \PPPP(\XXX)$, we write
\[\langle f,\mu\rangle=\int_{\XXX} f(y)\mu(\ddd y),\]
where $C_b(\XXX)$ is the space of all bounded continuous functions. The weak topology in $\PPPP(\XXX)$ is metrized by the dual-Lipschitz metric
\begin{align}
    \label{notation2}
    \|\mu_1-\mu_2\|_{\lip(\XXX)}^*:=\sup_{\substack{f\in\lip(\XXX)\\\|f\|_{\lip(\XXX)}\le1}}|\langle f,\mu_1\rangle-\langle f,\mu_2\rangle|.
\end{align}
Here, $\lip(\XXX)$ is the space of all bounded Lipschitz functions endowed with the norm
\[\|f\|_{\lip(\XXX)}:=\|f\|_\infty+\sup_{x,y\in\XXX,\ x\neq y}\frac{|f(x)-f(y)|}{d_{\XXX}(x,y)}.\]
\item For a metric space $(\XXX,d_{\XXX})$, 
\[B_\XXX(x,R):=\{y\in\XXX| d_{\XXX}(x,y)< R\},\qquad B_\XXX(R):=B_\XXX(0,R).\]
\item In the current domain $D:=\R\times (-1,1)$, one has the following Poincar\'e inequality
\[\|\nabla u\|^2\ge \frac{\pi^2}{4}\|u\|^2,\qquad \forall u\in H^1_0.\]
To track the dependence of our estimates on the Poincar\'e inequality, we write $\alpha_0:=\frac{\pi^2}{4}$.
\item We write
\[\FB_0:=\ssum_{j\ge 1}b_j^2,\qquad \FB_\fw:=\ssum_{j\ge 1} b_j^2(\|\fw e_j\|^2+\|e_j\|^2_{H^1}).\]
\end{itemize}

\section{Finite-dimensional stabilization}\label{FDS}
In this section, we study the finite-dimensional stabilization problem for the finite-energy solutions of \eqref{INTRO1}. By applying the Leray projection $\Pi$ and writing 
\begin{align}
    \label{MFE0}L:=\Pi\Delta,\qquad B(u):=B(u,u),\qquad B(u,v)=\Pi(u\cdot\nabla )v,
\end{align}
we eliminate the pressure and therefore consider \eqref{INTRO1} as an abstract evolutionary equation
\begin{align}
    \label{MFE1}
    \partial_t u-\nu Lu+B(u)=h+\eta
\end{align}
in the energy space $H$. It is known that \eqref{MFE1} defines a Markov family $(u^H_t,\PP_u^H)$ parameterized by the initial data $u\in H$. In particular, under $\PP_u^H$, the process $\{u^H_t\}_{t\ge 0}$ is distributed as the solution of \eqref{MFE1} issued from $u$. Throughout Sections \ref{FDS} and \ref{MFE}, we shall drop the superscript and denote this Markov family by $(u_t,\PP_u)$. This will not lead to any confusion. Define the transition operators associated with this Markov family by 
	\begin{align}
		\PPPPP_tf(u)&:=\int_{H}f(v)P_t(u,\ddd  v), &\PPPPP_t&:C_b(H)\to C_b(H),\label{Markov1}\\\label{Markov2}\PPPPP^{*}_t\lambda(\Gamma)&:=\int_{H}P_t(u,\Gamma)\lambda(\ddd  u),  &\PPPPP^{*}_t&:\PPPP(H)\to \PPPP(H),\end{align}
	where \begin{align}\label{Markov3}P_t(u,\Gamma):=\PP_u\left\{u_t\in\Gamma\right\}\end{align}
	is the transition function.

\subsection{Asymptotic compactness and finite-dimensional stabilization}
    The goal of this subsection is to establish finite-dimensional stabilization for \eqref{MFE1}, where the dimension $N$ depends on the noise coefficients $\{b_j\}_{j\ge 1}$. Such dependence shall be removed in the next subsection. 
 
    Let $\{v(t)\}_{t\ge0}$ be the solution of the auxiliary equation 
    \begin{align}
    \label{MFE2}
    \partial_t v-\nu Lv+B(v)+a\PPP_N(v-u)=h+\eta,\qquad v|_{t=0}=u'\in H,
    \end{align}
    where $a>0$, $\PPP_N$ is the projection operator onto the space spanned by $\{e_1,\ldots,e_N\}$, and $\{u(t)\}_{t\ge 0}$ is the solution of \eqref{MFE1} issued from $u\in H$. The main results are presented below.
\begin{prop}\label{propMFE1}
    For any finite set
    \[K:=\{\fu_1,\ldots,\fu_l\}\subset C_0^\infty\cap H,\]
    there is an integer $N\ge 1$ and a constant $a>0$ such that for any $u,u'\in H$,
    \begin{align}
    \label{MFE3}
    \|u(t)-v(t)\|^2\le\|u-u'\|^2\exp\left(-\nu\alpha_0 t+\CC_1 \int_0^t\dist_V(u(s),K)^2 \ddd s\right),
    \end{align}
    where $\CC_1$ depends only on $\nu$, $\alpha_0$ is the constant in the Poincar\'e inequality, and
    \begin{align}
    \dist_V(u,K):=\inf_{w\in K}\|u-w\|_{V}\label{MFE1-1}
    \end{align}
    with $V:=H_0^1\cap H$.
\end{prop}

\begin{prop}\label{propMFE2}
    Suppose that $\FB_0:=\ssum_{j\ge 1}b_j^2<\infty$. Then, for any $\epsilon>0$, there is a finite set $K\subset C_0^\infty\cap H$ depending on $\epsilon,\nu,h,\{b_j\}_{j\ge1}$ such that for any $T,\rho\ge0$, $q\ge 2$, and $u\in H$, 
    \begin{align}
        \label{MFE4}
        \PP\left\{\sup_{t\ge T}\left(\int_0^t \dist_V(u(s),K)^2 \ddd s-\epsilon t\right)\ge \rho\right\}\le C_{q}\frac{1+\|u\|^{8q}}{(1+T+\rho)^q},
    \end{align}
    where $C_q$ depends on $q,\epsilon,\nu,h,\{b_j\}_{j\ge 1}$.
\end{prop}
\subsubsection{The Foia\c{s}--Prodi estimate}
This part is devoted to the proof of Proposition \ref{propMFE1}, which is based on the following lemma. 
\begin{lemma}\label{lemmaMFE1}
    For any $\epsilon>0$ and $\OO \Subset D$, there is an integer $N\ge 1$ and a constant $C_\epsilon>0$ such that 
    \begin{align}\|g\|_{L^2(\OO)}\le \epsilon\|\nabla g\|+C_\epsilon\|\PPP_N g\|
        \label{MFE5}
    \end{align}
    for any $g\in V$.
\end{lemma}
\begin{proof}\textbf{Step I.} By contradiction, suppose that there is $\epsilon_0>0$ such that for any integer $N\ge 1$, one can find a function $g_N\in V$ satisfying
    \[\|g_N\|_{L^2(\OO)}>\epsilon_0 \|\nabla g_N\|+N\|\PPP_N g_N\|.\]
    Let $h_N:=\frac{g_N}{\|g_N\|_{L^2(\OO)}}$. Then, one has
    \begin{align}
        \label{MFE6}\epsilon_0 \|\nabla h_N\|+N\|\PPP_N h_N\|<1,
    \end{align}
    which, together with the Poincar\'e inequality, implies that there is a function $h\in V$ such that, up to a subsequence, $h_N$ converges weakly to $h$, as $N\to\infty$. In particular, for each $j\ge 1$, it follows that
    \[\langle h,e_j\rangle= \lim_{N\to\infty}\langle h_N,e_j\rangle.\]
    Combining this with \eqref{MFE6}, one gets
    \begin{align*}|\langle h,e_j\rangle|= \lim_{N\to\infty}|\langle h_N,e_j\rangle|\le \lim_{N\to\infty}\|\PPP_N h_N\|=0,
    \end{align*}
    and thus $h\equiv0$.

    \noindent\textbf{Step I\!I.} Since $\OO\Subset D$, by compact Sobolev embedding, it follows
    \[h_N\to h\]
    strongly in $L^2(\OO)$, as $N\to\infty$. Combining this with the fact that $\|h_N\|_{L^2(\OO)}=1$, one has
    \[\|h\|_{L^2(\OO)}=1,\]
    which is a contradiction. Therefore, estimate \eqref{MFE5} holds.
\end{proof}
Now we turn to the proof of Proposition \ref{propMFE1}.
\begin{proof}[Proof of Proposition \ref{propMFE1}]Let $g(t):=u(t)-v(t)$. The equation for $g$ is given by
\begin{align}
    \label{MFE7}\partial_t g-\nu Lg+B(g,u)+B(v,g)+a\PPP_Ng=0.
\end{align}
Therefore, by testing against $g$ on both sides, one has
\begin{align}\label{MFE8}
    \frac{1}{2}\frac{\ddd }{\ddd t}\|g\|^2+\nu\|\nabla g\|^2+a\|\PPP_N g\|^2=-\langle B(g,u),g\rangle.
\end{align}
Let $\OO:=\cup_{i=1}^l\supp \fu_i$. For any $i=1,\ldots,l$, it follows that
\begin{align}\label{MFE9}
    \langle B(g,u),g\rangle=\langle B(g,u-\fu_i),g\rangle+\langle B(g,\fu_i),g\rangle,
\end{align}
where, by the Ladyzhenskaya inequality, the Poincar\'e inequality, and Lemma \ref{lemmaMFE1}, one has
\begin{align}\label{MFE10}
    |\langle B(g,u-\fu_i),g\rangle|\le C \|g\|\|\nabla g\|\|\nabla (u-\fu_i)\|\le \frac{\nu}{4}\|\nabla g\|^2+\CC_1\|g\|^2\|u-\fu_i\|^2_V,
\end{align}
and 
\begin{align}\label{MFE11}
    |\langle B(g,u_i),g\rangle|&\le\left( \max_{i=1,\ldots,l}\|\nabla \fu_i\|_{\infty}\right)\|g\|_{L^2(\OO)}^2\le\epsilon^2C_{K}\|\nabla g\|^2+C_{K,\epsilon}\|\PPP_N g\|^2.
\end{align}
By plugging \eqref{MFE9}--\eqref{MFE11} into \eqref{MFE8} and choosing $\epsilon$ sufficiently small so that $\epsilon^2 C_{K}\le \frac{\nu}{8}$, it follows that
\begin{align*}
    \frac{\ddd }{\ddd t}\|g\|^2+\nu\|\nabla g\|^2+2a\|\PPP_N g\|^2\le \CC_1\|g\|^2\|u-\fu_i\|^2_V+ C_{\nu,K}\|\PPP_N g\|^2,
\end{align*}
which, together with setting $2a=C_{\nu,K}$, taking minimum in $i$, and applying the Poincar\'e and Gronwall inequalities, yields \eqref{MFE3}.
\end{proof}
\subsubsection{Asymptotic compactness}
We establish the following asymptotic compactness property for equation \eqref{MFE1}, which plays an important role in the proof of Proposition \ref{propMFE2}.
\begin{prop}\label{propMFE3}
    For any $R,\epsilon>0$, there exists a compact subset $\KK$ such that for any $\delta>0$, there is an integer $\fn_0\ge 1$ satisfying
    \begin{align}
        \label{MFE12}
        \sup_{\|u\|\le R}\PP\{\dist_H(u(t),\KK)>\delta\}\le \epsilon,\qquad \forall t\ge \fn_0,
    \end{align}
    where $\{u(t)\}_{t\ge0}$ is the solution of \eqref{MFE1} issued from $u$. Moreover, the compact set $\KK$ depends on $R,\epsilon,\nu,h,\{b_j\}_{j\ge1}$, and the integer $\fn_0$ additionally depends on $\delta$ and the choice of $\KK$.
\end{prop}
To prove the above proposition, we need the following lemma, which is a generalization of Proposition 2.4 in \cite{N22} to the case of white-in-time noise.
\begin{lemma}\label{lemmaMFE2}
    For any $\{u_n\}_{n\ge 1}$ and $\{t_n\}_{n\ge 1}$ satisfying 
    \[\sup_{n\ge 1}\|u_n\|<\infty,\qquad\lim_{n\to\infty}t_n=\infty,\]
    the sequence $\{P_{t_n}(u_n,\cdot)\}_{n\ge 1}$ is tight, where $P_t(u,\cdot)$ is the transition function defined in \eqref{Markov3}.
\end{lemma}
\begin{proof}
    \textbf{Step I.} In this step, we establish an auxiliary estimate. Let $\{v_n(t)\}_{t\ge0}$ be the solution of \eqref{MFE1} issued from $v_n(0)$. Suppose that 
    \begin{align}
        \label{MFE13}
        \MM:=\sup_n \E\|v_n(0)\|^4<\infty,
    \end{align}
    and
    \begin{align}\label{MFE14}
    v_n(0)\rightharpoonup v(0)\qquad \mbox{in $H$}
    \end{align}
    almost surely, as $n\to\infty$. Let $\{v(t)\}_{t\ge 0}$ be the solution of \eqref{MFE1} issued from $v(0)$. We claim that 
    \begin{align}
        \label{MFE15}
        \limsup_{n\to\infty}\E\|v_n(t)-v(t)\|^2\le e^{-\nu\alpha_0t}\limsup_{n\to\infty}\E\|v_n(0)-v(0)\|^2,
    \end{align}
    where $\alpha_0$ is the constant in the Poincar\'e inequality. 

    First, we identify the limit of $\{v_n(\cdot)\}_{n\ge 1}$. Let $W(t):=\int_0^t\eta(s)\ddd s$ denote the driving Wiener process in \eqref{MFE1}. Fix arbitrary $T>0$. By applying \eqref{A6} and \eqref{A7} with $q=1$ and $\gamma\in (0,\frac{1}{2})$, one has
    \[\sup_{n\ge 1}\E\left(\|v_n(\cdot)\|_{L^2(0,T;V)}+\|v_n(\cdot)\|_{W^{\gamma,2}(0,T;V^*)}\right)\lesssim_{\gamma,T,\nu,h,\FB_0}1+\MM.\]
    Together with \eqref{MFE13}, Theorem 2.1 in \cite{FG95}, and the Prokhorov theorem, this implies that the sequence $\DDD(v_n(\cdot),v(\cdot),v_n(0),v(0),W)$ is tight in 
    \[\PPPP\left([L^2(0,T;H_{\loc})\cap L^2(0,T;V)_{\textup{w}}]^2\times H_{\textup{w}}\times H_{\textup{w}}\times C([0,T];H)\right),\]
    where $H_{\textup{w}}$ stands for the space $H$ endowed with the weak topology and $L^2(0,T;V)_{\textup{w}}$ is understood in the same way. Thus, by applying the Jakubowski--Skorohod theorem, see Theorem 1.1 in \cite{Jak98}, and changing the probability space if necessary, there exists a sequence $(\tv_{n_j}(\cdot),\tv_{n_j}'(\cdot),\tv_{n_j}(0),\tv_{n_j}'(0),\tilde W_{n_j})$, which has the same law as $(v_{n_j}(\cdot),v(\cdot),v_{n_j}(0),v(0),W)$ and converges $\PP$-almost surely. Such convergence is sufficient to pass to the limit and therefore indicates that the limit of $\tv_{n_j}(\cdot)$ is a solution of \eqref{MFE1} driven by $\tilde W$ with the initial data distributed as $\DDD(v(0))$. This, together with the moment estimate
    \begin{align}
        \label{MFE15-1}\E\|v(0)\|^4\le \E\liminf_{n\to\infty}\|v_n(0)\|^4\le \liminf_{n\to\infty}\E\|v_n(0)\|^4\le \MM,
    \end{align}
    implies that the limit is a finite-energy solution. On the other hand, by \eqref{MFE14}, the limits of $\tv_{n_j}(0)$ and $\tv_{n_j}'(0)$ coincide. Together with the pathwise uniqueness of \eqref{MFE1}, one has 
    \[(\tv_{n_j}(\cdot),\tv'_{n_j}(\cdot))\to (\tv(\cdot),\tv(\cdot))\qquad \mbox{in $[L^2(0,T;H_{\loc})]^2$ a.s.}\]
    for some process $\tv$, which implies
    \begin{align}\label{MFE16-0}
        v_{n_j}(\cdot)\to  v(\cdot)\qquad\mbox{in $L^2(0,T;H_{\loc})$ and in probability},
    \end{align}
    as $j\to\infty$. Consequently, any subsequence of $\{v_n(\cdot)\}_{n\ge 1}$ has a subsequence converging to the same limit. Thus, \eqref{MFE16-0} holds for the whole sequence $\{v_n(\cdot)\}_{n\ge 1}$. Furthermore, by \eqref{A4} with $q=2$, the sequence $\{\|v_n-v\|^2_{L^2((0,T);L^2(\OO))}\}_{n\ge 1}$ is uniformly integrable. An application of the Vitali convergence theorem, see Theorem~4.5.4 in \cite{Bog07}, then implies 
    \begin{align}\label{MFE16}
        v_{n}(\cdot)\to  v(\cdot)\qquad\mbox{in $L^2(\Omega\times(0,T);L^2(\OO))$},
    \end{align}
    as $n\to\infty$ for any $\OO\Subset D$.
    
    Next, by applying \eqref{A6} with $q=1$, $\{v_n(\cdot)\}_{n\ge 1}$ is uniformly bounded in $L^2(\Omega\times(0,T); V)$. Thus, there exists a subsequence $n_k$ and a process $\bar{v}(\cdot)$ such that 
    \begin{align*}v_{n_k}(\cdot)\rightharpoonup \bar{v}(\cdot)\qquad\mbox{in $L^2(\Omega\times(0,T);V)$},
    \end{align*}
    as $k\to\infty$. Since $V\subset L^2(\OO)$ for any $\OO\Subset D$, one has 
    \[v_{n_k}(\cdot)\rightharpoonup \bar{v}(\cdot)\qquad\mbox{in $L^2(\Omega\times(0,T);L^2(\OO))$},\]
    which, together with \eqref{MFE16}, implies that $v(\cdot)=\bar{v}(\cdot)$. As $v(\cdot)$ is the only accumulation point of $\{v_n(\cdot)\}_{n\ge 1}$ in the weak topology of $L^2(\Omega\times(0,T); V)$, one gets
    \begin{align}
        \label{MFE17}v_{n}(\cdot)\rightharpoonup v(\cdot)\qquad\mbox{in $L^2(\Omega\times(0,T);V)$},
    \end{align}
    as $n\to\infty$.

    Finally, we claim that for any $t>0$,
    \begin{align}
        \label{MFE18}
        v_n(t)\rightharpoonup v(t)\qquad\mbox{in $L^2(\Omega;H)$},
    \end{align}
    as $n\to\infty$. Indeed, by density argument, to establish \eqref{MFE18}, it suffices to show
    \begin{align}
        \label{MFE19}
        \lim_{n\to\infty}\E|\langle v_n(t)-v(t),\varphi\rangle|=0
    \end{align}
    for any test function $\varphi\in C_0^\infty\cap H$. For simplicity, we write $z_n(t):=v_n(t)-v(t)$ and set $\OO:=\supp \varphi$. Then, for any $s\le t\le T$, it follows 
    \begin{align*}
        |\langle z_n(t)-z_n(s),\varphi\rangle|\le\nu\int_s^t|\langle \nabla z_n(r),\nabla \varphi\rangle|\ddd r+\int_s^t|\langle v_n\otimes v_n(r)-v\otimes v(r),\nabla \varphi\rangle|\ddd r,
    \end{align*}
    where, by applying \eqref{A6} with $q=1$, one has 
    \begin{align*}
        \E\int_s^t|\langle \nabla z_n(r),\nabla \varphi\rangle|\ddd r&\le \sqrt{t-s}\|\nabla \varphi\|\left(\E\int_0^T \|\nabla z_n(r)\|^2 \ddd r\right)^{\frac{1}{2}}\\&\lesssim_{T,\nu,h,\FB_0} \sqrt{t-s}\|\nabla \varphi\|(1+\MM),
    \end{align*}
    and 
    \begin{align*}
        \E\int_s^t|&\langle v_n\otimes v_n(r)-v\otimes v(r),\nabla \varphi\rangle|\ddd r\\&\le C\sqrt{\alpha_0} \sqrt{t-s}\|\nabla \varphi\|_{\infty} \left(\E\int_0^T\|\nabla z_n(r)\|^2\ddd r\right)^{\frac{1}{2}}\left(\E\sup_{r\le T}\left(\|v_n(r)\|^2+\|v(r)\|^2\right)\right)^{\frac{1}{2}}\\&\lesssim_{T,\nu,h,\FB_0}\sqrt{t-s}\|\nabla \varphi\|_{\infty}(1+\MM).
    \end{align*}
    Consequently, 
    \begin{align*}
        \E|\langle z_n(t)-z_n(s),\varphi\rangle| \lesssim_{\MM,T,\nu,h,\FB_0}\sqrt{t-s}\left(\|\nabla \varphi\|+\|\nabla\varphi\|_{\infty}\right).
    \end{align*}
    Together with \eqref{MFE16}, this implies 
    \begin{align*}
        \E|\langle z_n(t),\varphi\rangle|&\le \frac{1}{\delta}\int_{t-\delta}^t\E|\langle z_n(t)-z_n(s),\varphi\rangle|\ddd s+\frac{1}{\delta}\E\int_{t-\delta}^t|\langle z_n(s),\varphi\rangle|\ddd s\\&\lesssim_{\varphi,\MM,T,\nu,h,\FB_0}\frac{1}{\delta}\int_{t-\delta}^t \sqrt{t-s}\ddd s+\frac{1}{\sqrt{\delta}}\left(\E\int_{t-\delta}^t\|z_n(s)\|_{L^2(\OO)}^2 \ddd s\right)^{\frac{1}{2}}\\&\lesssim_{\varphi,\MM,T,\nu,h,\FB_0} \sqrt{\delta}+\frac{1}{\sqrt{\delta}}\left(\E\int_{0}^T\|z_n(s)\|_{L^2(\OO)}^2 \ddd s\right)^{\frac{1}{2}},
    \end{align*}
    which ensures \eqref{MFE19} and thus \eqref{MFE18}.
    
    Now we are in a position to prove \eqref{MFE15}. From \eqref{A1}, it follows 
    \begin{align}\label{MFE20}
        \E\|v_n(t)\|^2&=e^{-\nu\alpha_0 t}\E\|v_n(0)\|^2+\E\int_0^te^{-\nu\alpha_0 (t-s)}\left(2\langle v_n(s),h\rangle+\FB_0\right)\ddd s\notag\\&\qquad-\E\int_0^te^{-\nu\alpha_0 (t-s)}\left(2\nu\|\nabla v_n(s)\|^2-\nu\alpha_0\|v_n(s)\|^2\right)\ddd s.
    \end{align}
    By applying \eqref{MFE17}, one has 
    \begin{align}\label{MFE21}
        \lim_{n\to\infty}\E\int_0^te^{-\nu\alpha_0 (t-s)}\langle v_n(s),h\rangle\ddd s=\E\int_0^te^{-\nu\alpha_0 (t-s)}\langle v(s),h\rangle\ddd s.
    \end{align}
    Furthermore, by the Poincar\'e inequality, it follows that
    \[\nu\|\nabla v\|^2\le 2\nu\|\nabla v\|^2-\nu\alpha_0\|v\|^2\le 2\nu\|\nabla v\|^2.\]
    Consequently, the functional 
    \[u\mapsto\left(\E\int_0^te^{-\nu\alpha_0 (t-s)}\left(2\nu\|\nabla u(s)\|^2-\nu\alpha_0\|u(s)\|^2\right)\ddd s\right)^{\frac{1}{2}}\]
    defines an equivalent norm on $L^2((0,t)\times \Omega;V)$, which implies 
    \begin{align}
        \label{MFE22}
        \E\int_0^te^{-\nu\alpha_0 (t-s)}&\left(2\nu\|\nabla v(s)\|^2-\nu\alpha_0\|v(s)\|^2\right)\ddd s\notag\\&\le \liminf_{n\to\infty}\E\int_0^te^{-\nu\alpha_0 (t-s)}\left(2\nu\|\nabla v_n(s)\|^2-\nu\alpha_0\|v_n(s)\|^2\right)\ddd s.
    \end{align}
    Since the energy equality \eqref{MFE20} also holds for $\{v(t)\}_{t\ge0}$, by applying \eqref{MFE20}--\eqref{MFE22}, one obtains
    \begin{align}\label{MFE23}
        \limsup_{n\to\infty}\E\left(\|v_n(t)\|^2-\|v(t)\|^2\right)\le e^{-\nu\alpha_0 t}\limsup_{n\to\infty}\E\left(\|v_n(0)\|^2-\|v(0)\|^2\right).
    \end{align}
    To conclude the proof of \eqref{MFE15}, notice that $v(t)\in L^2(\Omega;H)$. By combining this with \eqref{MFE18}, one gets
    \begin{align}\label{MFE24}
        \limsup_{n\to\infty}\E\|v_n(t)-v(t)\|^2&=\limsup_{n\to\infty}\E\left(\|v_n(t)\|^2+\|v(t)\|^2-2\langle v_n(t),v(t)\rangle\right)\notag\\&=\limsup_{n\to\infty}\E\left(\|v_n(t)\|^2-\|v(t)\|^2\right).
    \end{align}
    On the other hand, by using \eqref{MFE14} and \eqref{MFE15-1}, the sequence $\{\langle v_n(0),v(0)\rangle\}_{n\ge 1}$ is uniformly integrable. Since 
    \[\langle v_n(0),v(0)\rangle\to\|v(0)\|^2,\]
    almost surely, as $n\to\infty$, an application of the Vitali convergence theorem then gives 
    \[\lim_{n\to\infty}\E\langle v_n(0),v(0)\rangle=\E\|v(0)\|^2.\]
    As shown in \eqref{MFE24}, this yields
    \[\limsup_{n\to\infty}\E\|v_n(0)-v(0)\|^2=\limsup_{n\to\infty}\E\left(\|v_n(0)\|^2-\|v(0)\|^2\right).\]
    By plugging this and \eqref{MFE24} into \eqref{MFE23}, one obtains \eqref{MFE15}.

    \noindent\textbf{Step I\!I.} Fix an arbitrary integer $m\ge 1$. For $t_n\ge m$, we write
    \[\mu_n:=P_{t_n}(u_n,\cdot),\qquad \lambda_{n,m}:=P_{t_n-m}(u_n,\cdot).\]
    By applying \eqref{A4} with $q=2$, it follows that
    \begin{align}\label{MFE25}
        \int_{H}\|v\|^4 \lambda_{n,m}(\ddd v)\le \|u_n\|^4+C_{\nu,h,\FB_0}\le \sup_{n\ge 1}\|u_n\|^4+C_{\nu,h,\FB_0}.
    \end{align}
    Together with the Jakubowski--Skorohod representation theorem, this implies that there are $H$-valued random variables $\{\fv_{n_k,m}\}_{k\ge 1}$ and $\fv_m$ satisfying
    \begin{align}
        \label{MFE26}\DDD(\fv_{n_k,m})=\lambda_{n_k,m},\qquad\forall k\ge 1,
    \end{align}
    and 
    \begin{align}
        \label{MFE27}\fv_{n_k,m}\rightharpoonup \fv_m \qquad \mbox{in $H$}
    \end{align}
    almost surely, as $k\to\infty$. After a diagonal argument, one may find a common subsequence (still denoted by) $n_k$ such that \eqref{MFE26} and \eqref{MFE27} hold for all $m\ge 1$. Let $\{\fv_{n_k,m}(t)\}_{t\ge 0}$ and $\{\fv_m(t)\}_{t\ge 0}$ denote the solutions of \eqref{MFE1} issued from $\fv_{n_k,m}$ and $\fv_{m}$, respectively. Since \eqref{MFE13} and \eqref{MFE14} hold for $\{\fv_{n_k,m}\}_{k\ge 1}$ and $\fv_m$, due to \eqref{MFE25} and \eqref{MFE27}, by applying \eqref{MFE15}, it follows
    \begin{align*}
        \limsup_{k\to\infty}\E\|\fv_{n_k,m}(t)-\fv_{m}(t)\|^2&\le e^{-\nu\alpha_0t}\limsup_{k\to\infty}\E\|\fv_{n_k,m}(0)-\fv_{m}(0)\|^2\\&\le e^{-\nu\alpha_0t}\left(\sup_{n\ge 1}\|u_n\|^4+C_{\nu,h,\FB_0}\right)^{\frac{1}{2}}.
    \end{align*}
    Combining this with the fact
    \[\DDD(\fv_{n_k,m}(m))=\mu_{n_k},\qquad \forall k\ge 1\]
    and the definition \eqref{notation2} of the dual-Lipschitz metric, one gets
    \begin{align}
        \label{MFE28}
        \limsup_{k\to\infty} \|\mu_{n_k}-\lambda_m\|^*_{\lip(H)}\le \left(\sup_{n\ge 1}\|u_n\|^4+C_{\nu,h,\FB_0}\right)^{\frac{1}{4}}\exp\left(-\frac{\nu\alpha_0}{2}m\right),
    \end{align}
    where $\lambda_m:=\DDD(\fv_m(m))$.

    \noindent\textbf{Step I\!I\!I.} Since $m$ is arbitrary in \eqref{MFE28}, by the triangle inequality, one has 
    \[\limsup_{k,j\to\infty} \|\mu_{n_k}-\mu_{n_j}\|^*_{\lip(H)}\le 2\left(\sup_{n\ge 1}\|u_n\|^4+C_{\nu,h,\FB_0}\right)^{\frac{1}{4}}\exp\left(-\frac{\nu\alpha_0}{2}m\right).\]
    Together with sending $m\to\infty$, this shows that $\{\mu_{n_k}\}_{k\ge 1}$ is a Cauchy sequence and thus that $\{\mu_n\}_{n\ge 1}$ is compact. An application of the Prokhorov theorem then completes the proof.
\end{proof}
Now we turn to the proof of Proposition \ref{propMFE3}.
\begin{proof}[Proof of Proposition \ref{propMFE3}] \textbf{Step I.} Fix $R,\epsilon>0$. Define $\KKK_R$ as the set of probability measures $\lambda$ such that there are $\{u_n\}_{n\ge 1}$ and $\{t_n\}_{n\ge 1}$ satisfying 
\[\sup_{n\ge 1}\|u_n\|\le R,\qquad \lim_{n\to\infty}t_n=\infty,\]
and 
\[\lambda=\lim_{n\to\infty}P_{t_n}(u_n,\cdot).\]
We claim that $\KKK_R$ is compact. To see this, for any sequence $\{\lambda_k\}_{k\ge 1}\subset \KKK_R$, choose 
\[\tilde \lambda_k:=P_{t_k}(u_k,\cdot)\]
such that $t_k\ge k$, $\|u_k\|\le R$, and 
\begin{align}
    \label{MFE29}\|\lambda_k-\tilde \lambda_k\|^*_{\lip(H)}\le \frac{1}{k}.
\end{align}
By Lemma \ref{lemmaMFE2}, $\{\tilde \lambda_k\}_{k\ge 1}$ has a convergent subsequence $\{\tilde \lambda_{k_j}\}_{j\ge 1}$. This, together with \eqref{MFE29}, implies the convergence of $\{\lambda_{k_j}\}_{j\ge 1}$ and thus the compactness of $\KKK_R$.

\noindent\textbf{Step I\!I.} By applying the Prokhorov theorem, there exists a compact subset $\KK=\KK_{R,\epsilon}$ of $H$ such that 
\[\sup_{\lambda\in \KKK_R}\lambda(\KK^c)\le \frac{\epsilon}{2}.\]
We claim that \eqref{MFE12} holds for the above set $\KK$. Indeed, by contradiction, suppose that there is $\delta_0>0$, $\{u_n\}_{n\ge 1}$, and $\{t_n\}_{n\ge 1}$ satisfying  
\[\sup_{n\ge 1}\|u_n\|\le R,\qquad \lim_{n\to\infty}t_n=\infty,\]
and 
\[\PP\{\dist_H(u(t_n;u_n),\KK)>\delta_0\}>\epsilon,\]
where $\{u(t;u)\}_{t\ge 0}$ denotes the solution of \eqref{MFE1} issued from $u$. Choose $\lambda\in\KKK_R$ such that, up to a subsequence,
\[\lambda=\lim_{n\to\infty}P_{t_n}(u_n,\cdot).\]
Since $F_\delta:=\{v|\dist_H(v,\KK)\ge \delta\}$ is closed and $F_\delta\subset \KK^c$, by the Portmanteau theorem, one has
\begin{align*}
    \epsilon&\le \limsup_{n\to\infty}\PP\{\dist_H(u(t_n;u_n),\KK)\ge\delta_0\}=\limsup_{n\to\infty}P_{t_n}(u_n,F_\delta)\\&\le\lambda(F_\delta)\le \lambda(\KK^c)\le \frac{\epsilon}{2},
\end{align*}
which is a contradiction. This gives the validity of \eqref{MFE12}.
\end{proof}
\subsubsection{Growth of distance functionals}
The goal of this part is to establish Proposition \ref{propMFE2}. To this end, we introduce the following lemma, which provides an estimate on the expected growth for the distance functional in \eqref{MFE4} over a unit time interval.
\begin{lemma}\label{lemmaMFE3}
    For any $R,\epsilon>0$, there is an integer $\fn_0\ge 1$ and a finite set 
    \[K:=\{0,\fu_1,\ldots,\fu_l\}\subset C_0^\infty\cap H\]
    such that 
    \begin{align}
        \label{MFE29-1}\sup_{\|u\|\le R}\E\int_{\fn_0}^{\fn_0+1}\dist_V(u(t),K)^2\ddd t\le\epsilon,
    \end{align}
    where $\{u(t)\}_{t\ge 0}$ is the solution of \eqref{MFE1} issued from $u\in H$. Moreover, the finite set $K$ and the integer $\fn_0$ depend on $R,\epsilon,\nu,h,\{b_j\}_{j\ge1}$.
\end{lemma}
\begin{proof}
    \textbf{Step I.} First, we construct the finite set $K$ in \eqref{MFE29-1}. Let $\{\fu_j\}_{j\ge 1}\subset C_0^\infty\cap H$ be a dense sequence in $V$. Define
    \[K_m:=\{0,\fu_1,\ldots,\fu_m\},\]
    and 
    \[\Phi_m(u):=\E_u\int_0^1\dist_V(u_t,K_m)^2\ddd t,\]
    where $(u_t,\PP_u)$ is the Markov family associated with \eqref{MFE1}. We claim that for any $R>0$, the sequence $\{\Phi_m\}_{m\ge 1}$ is uniformly bounded and equicontinuous on $B_H(R)$. Indeed, since $0\in K_m$, by applying \eqref{A6} with $q=1$, one has 
    \begin{align}\label{MFE33}
        \Phi_m(u)\le \E_u\int_0^1\|u_t\|^2_V\ddd t\lesssim_{\nu,h,\FB_0}1+\|u\|^2,
    \end{align}
    which ensures the uniform boundedness on $B_H(R)$. On the other hand, by \eqref{A8}, \eqref{MFE33}, and the Minkowski inequality, for any $u,u'\in B_H(R)$, it follows that 
    \begin{align*}
        |\Phi_m(u)-\Phi_m(u')|&\lesssim_{\nu,h,\FB_0} (1+\|u\|+\|u'\|)\left|\sqrt{\Phi_m(u)}-\sqrt{\Phi_m(u')}\right|\\&\lesssim_{\nu,h,\FB_0,R} \left(\E\int_0^1\|u(t)-u'(t)\|_V^2\ddd t\right)^{\frac{1}{2}}\\&\lesssim_{\nu,h,\FB_0,R}\|u-u'\|^{\frac{1}{2}}+\|u-u'\|^{\frac{c_0}{2}}.
    \end{align*}
    This gives the uniform equicontinuity of $\{\Phi_m\}_{m\ge 1}$ on $B_H(R)$.

    Now fix $R,\epsilon>0$. Let $\KK$ be the compact subset in Proposition \ref{propMFE3} corresponding to $R,\epsilon'$, where $\epsilon'$ is to be determined in the next step. Since $\{\fu_j\}_{j\ge 1}$ is dense in $V$, by the dominated convergence theorem combined with \eqref{MFE33}, it follows
    \[\lim_{m\to\infty}\Phi_m(u)=0\]
    for any $u\in H$. This, together with the uniform equicontinuity of $\{\Phi_m\}_{m\ge 1}$, gives
    \[\lim_{m\to\infty}\sup_{u\in \KK}\Phi_m(u)=0.\]
    In particular, choose $m$ sufficiently large so that 
    \begin{align}
        \label{MFE34}\sup_{u\in \KK}\Phi_m(u)<\frac{\epsilon}{4},
    \end{align}
    and set $K:=K_m$.

    \noindent\textbf{Step I\!I.} Since $\Phi_m$ is continuous and $\KK$ is compact, there is $\delta>0$ such that if $\dist_H(v,\KK)\le \delta$, then $\Phi_m(v)<\frac{\epsilon}{2}$. Let $\fn_0\ge 1$ be the integer in Proposition \ref{propMFE3} corresponding to $\delta$. Introduce 
    \[\GG_{\fn_0}:=\{\dist_H(u_{\fn_0},\KK)\le \delta\}.\]
    By the Markov property, Proposition \ref{propMFE3}, \eqref{MFE33}, and \eqref{A4} with $q=2$, for any $u\in B_H(R)$, it follows
    \begin{align*}
        \E_u\int_{\fn_0}^{\fn_0+1}\dist_V(u_t,K)^2 \ddd t&=\E_u\Phi_m(u_{\fn_0})=\E_u\left(\I_{\GG_{\fn_0}}\Phi_m(u_{\fn_0})\right)+\E_u\left(\I_{\GG^c_{\fn_0}}\Phi_m(u_{\fn_0})\right)\\&<\frac{\epsilon}{2}+C_{\nu,h,\FB_0}\sqrt{\epsilon'}\left(\E_u\left(1+\|u_{\fn_0}\|^4\right)\right)^{\frac{1}{2}}<\frac{\epsilon}{2}+C_{\nu,h,\FB_0,R}\sqrt{\epsilon'}
    \end{align*}
    Together with choosing $\epsilon'$ sufficiently small, this implies \eqref{MFE29-1}.
\end{proof}
Now we turn to the proof of Proposition \ref{propMFE2}. 
\begin{proof}
    \textbf{Step I.} We derive some auxiliary estimates. Fix arbitrary $\epsilon>0$. Applying Lemma~\ref{lemmaMFE3} with $\epsilon$ replaced by $\frac{\epsilon}{16}$ and $R$ to be chosen in the next step, there is an integer $\fn_0\ge 1$ and a finite set $K:=\{0,\fu_1,\ldots,\fu_l\}\subset C_0^\infty\cap H$ such that \eqref{MFE29-1} holds. On the other hand, for $u\notin \overline{B_H(R)}$, by the Markov property, \eqref{A4}, and \eqref{A6} with $q=1$, it follows 
    \begin{align*}
        \E_u \int_{\fn_0}^{\fn_0+1}\dist_V(u_t,&K)^2\ddd t\le \E_u \int_{\fn_0}^{\fn_0+1}\|u_t\|_V^2\ddd t=\E_u \left(\E_{u_{\fn_0}}\left(\int_{0}^{1}\|u_t\|_V^2\ddd t\right)\right)\\&\lesssim_{\nu,h,\FB_0}\E_u(1+\|u_{\fn_0}\|^2)\lesssim_{\nu,h,\FB_0}1+\|u\|^2\lesssim_{\nu,h,\FB_0}\frac{(1+\|u\|^2)^2}{R^2}.
    \end{align*}
    Consequently, by combining this with \eqref{MFE29-1}, for any $u\in H$, one has
    \begin{align}
        \label{MFE35}
        \E_u \int_{\fn_0}^{\fn_0+1}\dist_V(u_t,K)^2\ddd t\le \frac{\epsilon}{16}+\CC_2\frac{(1+\|u\|^2)^2}{R^2},
    \end{align}
    where the constant $\CC_2$ depends on $\nu,h,\FB_0$.

    \noindent\textbf{Step I\!I.} In this step, we derive an estimate for the expected growth of the distance functional in \eqref{MFE4} over larger time-blocks. Let $M\ge \fn_0$ be an integer. It follows that 
    \begin{align}\label{MFE36}
        \E_u\int_{0}^{M}\dist_V(u_t,K)^2\ddd t&=\E_u\int_{0}^{\fn_0}\dist_V(u_t,K)^2\ddd t\notag\\&\quad+\ssum^{M-1-\fn_0}_{j=0}\E_u\int_{\fn_0+j}^{\fn_0+1+j}\dist_V(u_t,K)^2\ddd t,
    \end{align}
    where, by the Markov property, \eqref{MFE35}, and \eqref{A4} with $q=2$, it follows 
    \begin{align*}
        \E_u\int_{\fn_0+j}^{\fn_0+1+j}\dist_V(u_t,K)^2\ddd t&=\E_u\left(\E_{u_{j}}\int_{\fn_0}^{\fn_0+1}\dist_V(u_t,K)^2\ddd t\right)\\&\le \frac{\epsilon}{16}+\frac{\CC_2}{R^2}\E_u\left(1+\|u_{j}\|^2\right)^2\\&\le \frac{\epsilon}{16}+\frac{C_{\nu,h,\FB_0}}{R^2}+\frac{\CC_2}{R^2}e^{-\nu\alpha_0j}\left(1+\|u\|^2\right)^2.
    \end{align*}
    On the other hand, by \eqref{A6} with $q=1$, one has
    \begin{align*}
        \E_u\int_{0}^{\fn_0}\dist_V(u_t,K)^2\ddd t\le C_{\nu,h,\FB_0}(\fn_0+\|u\|^2)\le C_{\nu,h,\FB_0}\fn_0\left(1+\|u\|^2\right)^2.
    \end{align*}
    Plugging the above estimates into \eqref{MFE36}, one gets
    \begin{align}\label{MFE36-1}
        \E_u\int_{0}^{M}\dist_V(u_t,K)^2\ddd t\le \CC_3\left(\fn_0+R^{-2}\right)\left(1+\|u\|^2\right)^2+\left(\frac{\epsilon}{16}+\CC_3 R^{-2}\right)M
    \end{align}
    for some constant $\CC_3$ depending on $\nu,h,\FB_0$. Choosing $R$ sufficiently large so that 
    \[\frac{\epsilon}{16}+\CC_3 R^{-2}\le \frac{\epsilon}{8},\]
    and writing 
    \[\V(u):=\left(1+\|u\|^2\right)^2,\]
    it follows
    \begin{align}
        \label{MFE37}
         \E_u\int_{0}^{M}\dist_V(u_t,K)^2\ddd t\le \CC_4\V(u)+\frac{\epsilon}{8}M,
    \end{align}
    where $\CC_4$ is a constant depending on $\epsilon,\nu,h,\{b_j\}_{j\ge 0}$.

    \noindent\textbf{Step I\!I\!I.} We shall divide the whole time interval into blocks of length $M$ and apply growth estimate on $[0,M]$ combined with the Markov property. However, estimate \eqref{MFE37} is not sufficient, since this would yield a sum of Lyapunov functionals $\ssum_{j\ge 0}\V(u_{jM})$. To this end, we aim to improve \eqref{MFE37}. Let $\PPPPP_t$ and $\PPPPP_t^{*}$ be the Markov operators introduced in \eqref{Markov1} and \eqref{Markov2}, respectively. Applying \eqref{A4} combined with the Markov property, it follows that 
    \begin{align}
        \label{MFE38}
        \PPPPP_M\V(u)\le \fc^M\V(u)+\frac{1-\fc^M}{1-\fc}C_{\nu,h,\FB_0},\qquad \fc:=e^{-\nu\alpha_0}.
    \end{align}
    By multiplying both sides by $\fa_M:=(1-\fc^M)^{-1}\CC_4$, this implies 
    \[\fa_M\left(\PPPPP_M\V(u)-\V(u)\right)\le -\CC_4\V(u)+\frac{\CC_4C_{\nu,h,\FB_0}}{1-\fc}.\]
    Adding the above estimate to \eqref{MFE37}, one eliminates the Lyapunov functional $\CC_4\V(u)$ on the right-hand side of \eqref{MFE37} and thereby obtains
    \begin{align*}
    \E_u\int_{0}^{M}\dist_V(u_t,K)^2\ddd t+\fa_M\left(\PPPPP_M\V(u)-\V(u)\right)\le \frac{\epsilon}{8}M+\frac{\CC_4C_{\nu,h,\FB_0}}{1-\fc}.
    \end{align*}
    Consequently, by choosing $M\ge 2$ sufficiently large, it follows
    \begin{align}
        \label{MFE39}
        \E_u\int_{0}^{M}\dist_V(u_t,K)^2\ddd t+\fa_M\left(\PPPPP_M\V(u)-\V(u)\right)\le \frac{\epsilon}{4}M.
    \end{align}

    \noindent \textbf{Step I\!V.} In this step, we derive a pathwise estimate for the distance functional
    \[\FI(kM):=\int_0^{kM}\dist_V(u_t,K)^2\ddd t. \]
    Define 
    \begin{align}\label{MFE40}
        \Gamma_{j}:=\FI((j+1)M)-\FI(jM)+\fa_M\left(\V(u_{(j+1)M})-\V(u_{jM})\right).
    \end{align}
    Applying the Markov property combined with \eqref{MFE39}, it follows 
    \begin{align}
        \label{MFE41}\E_u(\Gamma_{j}|\FF_{jM})\!=\E_{u_{jM}}\!\int_{0}^{M}\!\!\dist_V(u_t,K)^2\ddd t+\fa_M\!\left(\PPPPP_M\V(u_{jM})-\V(u_{jM})\right)\le \frac{\epsilon}{4}M.
    \end{align}
    By summing over $j$ from $0$ to $k-1$ on both sides of \eqref{MFE40} and using \eqref{MFE41}, one gets
    \begin{align}
        \label{MFE42}
        \FI(kM)+\fa_M\left(\V(u_{kM})-\V(u)\right)=\ssum_{j=0}^{k-1}\Gamma_{j}\le \frac{\epsilon kM}{4}+\ssum_{j=0}^{k-1}d_{j}
    \end{align}
    where 
    \[d_j:=\Gamma_j-\E_u(\Gamma_{j}|\FF_{jM})\]
    is a sequence of martingale differences.

    \noindent\textbf{Step V.} We establish moment estimates for the martingale differences. Fix $q\ge 2$. By the conditional Jensen inequality, \eqref{A4}, \eqref{A6}, the Markov property, and the fact that $0\in K$, one has 
    \begin{align*}
        \E_u(|d_j|^q|\FF_{jM})&\lesssim_q\E_u(|\Gamma_{j}|^q|\FF_{jM})\\&\lesssim_q \E_u\left(\left(\int_{jM}^{(j+1)M}\|u_t\|_V^2\ddd t\right)^q\bigg|\FF_{jM}\right)\\&\qquad+\fa_M^q\E_u\left(\V(u_{(j+1)M})^q\Big|\FF_{jM}\right)+\fa_M^q\V(u_{jM})^q\\&\lesssim_{q,\epsilon,\nu,h,\{b_j\}_{j\ge 1}}\V(u_{jM})^q.
    \end{align*}
    By taking expectation on both sides and using \eqref{A4}, it follows
    \begin{align}\label{MFE42-1}
        \sup_{j\ge 0}\E_u|d_j|^q\lesssim_{q,\epsilon,\nu,h,\{b_j\}_{j\ge 1}}\V(u)^q.
    \end{align}
    As a result, by the discrete-time Burkholder--Gundy--Davis inequality (Theorem 6.3.6 in \cite{Str10}) and the generalized Minkowski inequality, one obtains
    \begin{align}
        \label{MFE43}
        \E_u\left(\max_{l\le k} \left|\ssum_{j=0}^{l-1}d_j\right|^q\right)&\lesssim_q \E_u \left(\ssum_{j=0}^{k-1}|d_j|^2\right)^{\frac{q}{2}}\lesssim_q \left(\ssum_{j=0}^{k-1}\left(\E_u|d_j|^q\right)^{\frac{2}{q}}\right)^{\frac{q}{2}}\notag\\&\lesssim_{q,\epsilon,\nu,h,\{b_j\}_{j\ge 1}}k^{\frac{q}{2}}\V(u)^q.
    \end{align}

    \noindent\textbf{Step V\!I.} Finally, we are ready to prove \eqref{MFE4}. For $t\ge 0$, define $k(t):=\lceil \frac{t}{M} \rceil$. Then, from \eqref{MFE42}, it follows that
    \begin{align*}
        \FI(t)&\le \FI(k(t)M)\le \fa_M\V(u)+\frac{\epsilon}{4}k(t)M+\ssum_{j=0}^{k(t)-1}d_j \notag\\&\le \CC_5\V(u)+\frac{\epsilon}{4}t+\ssum_{j=0}^{k(t)-1}d_j
    \end{align*}
    for some constant $\CC_5$ depending on $\epsilon,\nu,h,\{b_j\}_{j\ge 1}$. This implies 
    \begin{align}
        \label{MFE44}
        \FI(t)-\epsilon t\le \CC_5\V(u)+\ssum_{j=0}^{k(t)-1}d_j-\frac{3\epsilon t}{4}.
    \end{align}
    Introduce a constant $\CC_6$, which is determined in \eqref{MFE46}. We distinguish two different cases.

    \noindent\textit{Case I.} Suppose that $1+T+\rho\le \CC_6 \V(u)$. Then, since $\V(u)\ge 1$, one has 
    \[\frac{\V(u)^{2q}}{(1+T+\rho)^q}\ge \CC_6^{-q}.\]
    This implies 
    \begin{align}
        \label{MFE45}\PP_u\left\{\sup_{t\ge T}\left(\int_0^t \dist_V(u(s),K)^2 \ddd s-\epsilon t\right)\ge \rho\right\}\le 1\lesssim_{q,\epsilon,\nu,h,\{b_j\}_{j\ge 1}} \frac{\V(u)^{2q}}{(1+T+\rho)^q}.
    \end{align}
    
    \noindent\textit{Case I\!I.} Now assume that $1+T+\rho> \CC_6 \V(u)$. Choose $\CC_6$ sufficiently large so that
    \begin{align}
        \label{MFE46} \CC_6\ge 2,\qquad \CC_5\CC_6^{-1}\le \frac{1}{8}\min\left\{1,\frac{3\epsilon}{4}\right\}.
    \end{align}
    This yields
    \begin{align}
        \label{MFE47}\CC_5\V(u)\le\frac{\fc_\epsilon}{8}\LLLL,\qquad \fc_\epsilon:=\min\left\{1,\frac{3\epsilon}{4}\right\},\qquad \LLLL:=1+T+\rho.
    \end{align}
    Notice that $\LLLL>T$. We introduce the following dyadic decomposition
    \[I_0:=[T,2\LLLL),\qquad I_j:=[2^j\LLLL,2^{j+1}\LLLL),\qquad j\ge 1.\]
    Suppose that there is $t\in I_0$ satisfying 
    \begin{align}
        \label{MFE47-1}\FI(t)-\epsilon t\ge \rho.
    \end{align}
    Then, by applying \eqref{MFE44} and \eqref{MFE47}, it follows
    \begin{align}\label{MFE48}
        \ssum_{j=0}^{k(t)-1}d_j\ge \rho-\CC_5\V(u)+\frac{3\epsilon t}{4}\ge \rho+\frac{3\epsilon T}{4}-\frac{\fc_\epsilon \LLLL}{8}.
    \end{align}
    On the other hand, since $1+T+\rho >\CC_6\V(u)\ge 2$, one has
    \[T+\rho >\frac{\LLLL}{2},\]
    which implies 
    \[\rho+\frac{3\epsilon T}{4}\ge \fc_\epsilon(\rho+T)>\frac{\fc_\epsilon \LLLL}{2}.\]
    Plugging the above estimate into \eqref{MFE48} gives
    \[\ssum_{j=0}^{k(t)-1}d_j\ge \frac{3\fc_\epsilon \LLLL}{8}.\]
    Consequently, by using 
    \[k(t)\le \frac{t}{M}+1<\frac{2\LLLL}{M}+1\le \CC_6 \LLLL,\]
    one gets
    \[\left\{\sup_{t\in I_0}\left(\FI(t)-\epsilon t\right)\ge \rho\right\}\subset \left\{\max_{l\le \CC_6\LLLL}\left|\ssum_{j=0}^{l-1}d_j\right|\ge \frac{3\fc_\epsilon \LLLL}{8}\right\}.\]
    An application of the Chebyshev inequality combined with \eqref{MFE43} yields
    \begin{align}
        \label{MFE49}
        \PP_u\left\{\sup_{t\in I_0}\left(\FI(t)-\epsilon t\right)\ge \rho\right\}&\lesssim_{q,\epsilon}\LLLL^{-2q}\E_u\max_{l\le \CC_6 \LLLL} \left|\ssum_{j=0}^{l-1}d_j\right|^{2q}\notag\\&\lesssim_{q,\epsilon,\nu,h,\{b_j\}_{j\ge 1}} \LLLL^{-q}\V(u)^{2q}.
    \end{align}
    Now assume that for $j\ge 1$, there is $t\in I_j$ satisfying \eqref{MFE47-1}. Then, one has
    \[\ssum_{j=0}^{k(t)-1}d_j\ge \fc_{\epsilon}2^j\LLLL-\frac{\fc_\epsilon \LLLL}{8}\ge 2^{j-1}\LLLL\fc_\epsilon,\]
    which, together with the estimate 
    \[k(t)\le \frac{t}{M}+1<\frac{2^{j+1}\LLLL}{M}+1\le 2^{j+1}\LLLL,\]
    implies 
    \[\left\{\sup_{t\in I_j}\left(\FI(t)-\epsilon t\right)\ge \rho\right\}\subset \left\{\max_{l\le 2^{j+1}\LLLL}\left|\ssum_{j=0}^{l-1}d_j\right|\ge  2^{j-1}\LLLL\fc_\epsilon\right\}\]
    and thus 
    \begin{align*}
        \PP_u\left\{\sup_{t\in I_j}\left(\FI(t)-\epsilon t\right)\ge \rho\right\}&\lesssim_{q,\epsilon}\LLLL^{-2q}2^{-2q(j-1)}\E_u\max_{l\le 2^{j+1}\LLLL} \left|\ssum_{j=0}^{l-1}d_j\right|^{2q}\\&\lesssim_{q,\epsilon,\nu,h,\{b_j\}_{j\ge 1}}\LLLL^{-q}2^{-qj}\V(u)^{2q}.
    \end{align*}
    Combining this with \eqref{MFE45} and \eqref{MFE49}, one obtains \eqref{MFE4}.
    \end{proof}
    
    \subsection{Uniform finite-dimensional stabilization via compactness of noise coefficients}\label{UFDS}
    It is noteworthy that the number $N$ of perturbed modes in \eqref{MFE2} depends on the noise coefficients $\{b_j\}_{j\ge 1}$. Indeed, the finite set $K$ in Proposition \ref{propMFE2} relies on $\{b_j\}_{j\ge 1}$, since the proof is based on Lemma \ref{lemmaMFE2}, namely the asymptotic compactness property for the transition kernels. Such dependence is not sufficient to establish mixing in the degenerate case through the coupling method.

    In this subsection, we aim to remove this dependence. More precisely, we show that $N$ can be chosen uniformly over each bounded subset of 
    \[\ell^2_\fw:=\{b_\cdot\in \ell^2|\|b_\cdot\|_{\ell^2_{\fw}}^2:= \FB_\fw<\infty\},\]
    where $\FB_\fw$ is introduced in \eqref{INTRO10}. Together with a Girsanov-type estimate carried out later, this is sufficient to prove Theorem \ref{MT2}. The main result of this subsection is presented below.
    \begin{prop}\label{propMFE4}
        For any $\epsilon,\RR>0$, there is a finite set $K\subset C^\infty_0\cap H$ depending only on $\epsilon,\RR,\nu,h$ such that for any $T,\rho\ge0$, $q\ge 2$, $u\in H$, and $\BB:=\{b_j\}_{j\ge 1}\in \overline{B}_{\ell^2_\fw}(\RR)$, 
        \begin{align}
        \label{MFE56}
        \PP^\BB_u\left\{\sup_{t\ge T}\left(\int_0^t \dist_V(u_s,K)^2 \ddd s-\epsilon t\right)\ge \rho\right\}\lesssim_{q,\RR,\epsilon,\nu,h}\frac{1+\|u\|^{8q}}{(1+T+\rho)^q},
    \end{align}
    where $(u_t,\PP_u^\BB)$ denotes the Markov family associated with \eqref{MFE1} whose noise coefficients are given by $\BB$.
    \end{prop}
    The proof is based on the following lemmas.
    \begin{lemma}\label{lemmaMFE5}
        $\ell^2_{\fw}$ is compactly embedded into $\ell^2$.
    \end{lemma}
    \begin{proof}
        \textbf{Step I.} We claim that 
        \begin{align}
            \label{MFE50}\lim_{j\to\infty}\left(\|e_j\|^2_{H^1}+\|\fw e_j\|^2\right)=\infty.
        \end{align}
        By contradiction, suppose that there is a subsequence $j_k$ such that 
        \[\sup_{k\ge 1}\left(\|e_{j_k}\|^2_{H^1}+\|\fw e_{j_k}\|^2\right)<\infty.\]
        Since $\fw$ is an admissible weight function, for any $\epsilon>0$, one may choose $R>0$ sufficiently large so that 
        \begin{align}
            \label{MFE51}\int_{-1}^1\int_{|x|>R}e_{j_k}^2\ddd x\ddd y\le \frac{1}{\inf_{|x|\ge R}\fw(x)^2}\|\fw e_{j_k}\|^2<\frac{\epsilon^2}{16},\qquad \forall k\ge 1.
        \end{align}
        On the other hand, by applying the compact Sobolev embedding theorem, $\{e_{j_k}\}_{k\ge 1}$ is compact in $L^2(\OO_R)$ with $\OO_R:=(-R,R)\times (-1,1)$. Consequently, up to a subsequence, there is a function $e$ satisfying
        \[e_{j_k}\to e,\qquad \mbox{in $L^2(\OO_R)$},\]
        as $k\to\infty$. In particular, there exists $k_0\ge 1$ such that 
        \[\|e_{j_k}-e_{j_l}\|_{L^2(\OO_R)}<\frac{\epsilon}{2},\qquad \forall k,l\ge k_0.\]
        Together with \eqref{MFE51}, this implies that
        \[\|e_{j_k}-e_{j_l}\|<\epsilon,\qquad \forall k,l\ge k_0.\]
        However, as $\{e_j\}_{j\ge 1}$ is an orthonormal basis in $H$, one has 
        \[\|e_j-e_k\|^2=\|e_j\|^2+\|e_k\|^2=2,\qquad \forall j\neq k.\]
        This is a contradiction. 

        \noindent\textbf{Step I\!I.} Now let $\{b_\cdot^{(n)}\}_{n\ge 1}$ be a bounded sequence of elements in $\ell^2_{\fw}$. For any integer $J\ge 1$, one has 
        \begin{align*}
            \ssum_{j\ge J} (b_j^{(n)})^2&\le \frac{1}{\inf_{j\ge J}(\|\fw e_j\|^2+\|e_j\|^2_{H^1})}\ssum_{j\ge J} (b^{(n)}_j)^2(\|\fw e_j\|^2+\|e_j\|^2_{H^1})\\&\le \frac{\sup_{n\ge 1}\|b^{(n)}_\cdot\|^2_{\ell^2_\fw}}{\inf_{j\ge J}(\|\fw e_j\|^2+\|e_j\|^2_{H^1})},
        \end{align*}
        which, together with \eqref{MFE50}, implies that for any $\epsilon>0$, there is $J\ge 1$ such that 
        \begin{align}\label{MFE52}
            \ssum_{j\ge J} (b_j^{(n)})^2\le\frac{\epsilon^2}{16},\qquad\forall n\ge 1.
        \end{align}
        For $j<J$, since 
        \[\sup_{n\ge 1}(b_j^{(n)})^2\le \sup_{n\ge 1}\|b^{(n)}_{\cdot}\|^2_{\ell^2_{\fw}},\]
        by compactness and a diagonal argument, up to a subsequence, $\{(b^{(n)}_1,\ldots,b^{(n)}_{J-1})\}_{n\ge 1}$ is convergent. In particular, there is $n_0\ge 1$ such that 
        \[\ssum_{j<J}(b_j^{(n)}-b_j^{(m)})^2\le \frac{\epsilon^2}{4},\qquad\forall n,m\ge n_0.\]
        Combining this with \eqref{MFE52}, $\{b_\cdot^{(n)}\}_{n\ge 1}$ is Cauchy. This yields the compact embedding of $\ell^2_\fw$ into $\ell^2$. 
    \end{proof}
    Next, we present a uniform version of Lemma \ref{lemmaMFE3}.
    \begin{lemma}\label{lemmaMFE4}
    For any $\RR,R,\epsilon>0$, there exist integers $\fn_0,\ldots,\fn_k\ge 1$ and a finite set 
    \[K:=\{0,\fu_1,\ldots,\fu_l\}\subset C_0^\infty\cap H\]
    such that for any $\BB=\{b_j\}_{j\ge 1}\in \overline{B}_{\ell^2_\fw}(\RR)$, 
    \begin{align}
        \label{MFE53}\min_{i=0,\ldots,k}\sup_{\|u\|\le R}\E^\BB_u\int_{\fn_i}^{\fn_i+1}\dist_V(u_t,K)^2\ddd t\le\epsilon.
    \end{align}
    \end{lemma}
    \begin{proof}\textbf{Step I.} Fix $\RR,R,\epsilon>0$. Let $u\in B_H(R)$ and   
    \[\BB:=\{b_j\}_{j\ge 1},\qquad\BB^{(0)}:=\{b^{(0)}_j\}_{j\ge 1}\]
    belong to $\overline{B}_{\ell^2_\fw}(\RR)$. Denote by $\{u^{(0)}(t)\}_{t\ge0}$ and $\{u(t)\}_{t\ge0}$ the solutions of \eqref{MFE1} issued from the same initial data $u$ with the noise coefficients given by $\BB^{(0)}$ and $\BB$, respectively. By applying Lemma \ref{lemmaMFE3}, there is an integer $\fn_0\ge 1$ and a finite set 
    \[K_0:=\{0,\fu_1^{(0)},\ldots,\fu_{l_0}^{(0)}\}\subset C_0^\infty\cap H\]
    depending on $R,\epsilon,\BB^{(0)}$ such that
    \begin{align}
        \label{MFE54}\sup_{\|u\|\le R}\E\int_{\fn_0}^{\fn_0+1}\dist_V(u^{(0)}(t),K_0)^2\ddd t\le\frac{\epsilon}{4}.
    \end{align}
    On the other hand, by using \eqref{A8}, there is $\delta_0>0$ depending on $R,\RR,\epsilon,\nu,h,\BB^{(0)}$ such that for any $\BB\in \ell^2$ satisfying 
    \begin{align*}\|\BB-\BB^{(0)}\|_{\ell^2}\le \delta_0,
    \end{align*}
    one has
    \[\E\int_0^{\fn_0+1}\|u(t)-u^{(0)}(t)\|_V^2\ddd t\le \frac{\epsilon}{4}.\]
    Together with \eqref{MFE54}, this gives
    \begin{align}
        \label{MFE55}\sup_{\|u\|\le R}\E_u^\BB\int_{\fn_0}^{\fn_0+1}\dist_V(u_t,K_0)^2\ddd t\le\epsilon,\qquad \forall \BB\in B_{\ell^2}(\BB^{(0)},\delta_0).
    \end{align}

    \noindent\textbf{Step I\!I.} By the definition \eqref{INTRO10} of $\FB_\fw$, the set $\overline{B}_{\ell^2_\fw}(\RR)$ is closed in $\ell^2$. This, together with Lemma \ref{lemmaMFE5}, implies the existence of $\BB^{(0)},\ldots,\BB^{(k)}\in \overline{B}_{\ell^2_\fw}(\RR)$ and $\delta_0,\ldots,\delta_k>0$ such that 
    \[\overline{B}_{\ell^2_\fw}(\RR)\subset \cup_{i=0}^kB_{\ell^2}(\BB^{(i)},\delta_i).\]
    For $i=0,\ldots,k$, we denote by $\fn_i$ and $K_i$ the integer and the finite set corresponding to $R,\epsilon,\BB^{(i)}$, as indicated in Lemma \ref{lemmaMFE3}. After relabeling, we write  
    \[K:=\cup_{i=0}^k K_i=\{0,\fu_1,\ldots,\fu_l\}\subset C^\infty_0\cap H.\]
    Now for any $\BB\in \overline{B}_{\ell^2_\fw}(\RR)$, there is $\fn_i$ such that
    \[\sup_{\|u\|\le R}\E_u^\BB\int_{\fn_i}^{\fn_i+1}\dist_V(u_t,K_i)^2\ddd t\le\epsilon.\]
    Since $\dist_V(u_t,K)\le \dist_V(u_t,K_i)$, this ensures \eqref{MFE53}.
    \end{proof}
    Now we turn to the proof of Proposition \ref{propMFE4}.
    \begin{proof}[Proof of Proposition \ref{propMFE4}]The main strategy is to proceed as in the proof of Proposition \ref{propMFE2} and avoid the dependence of the finite set $K$ in \eqref{MFE56} on the specific noise coefficients, through an application of Lemma \ref{lemmaMFE4}. Thus, we only indicate the necessary modifications. 

    Fix $\epsilon,\RR>0$ and $\BB\in \overline{B}_{\ell^2_\fw}(\RR)$. Applying Lemma \ref{lemmaMFE4} with $\epsilon$ replaced by $\frac{\epsilon}{16}$ and $R$ determined in \eqref{MFE56-1}, one may find integers $\fn_0,\ldots,\fn_m\ge 1$ and a finite set 
    \[K:=\{0,\fu_1,\ldots,\fu_l\}\subset C^\infty_0\cap H\]
    depending only on $\RR,\epsilon,\nu,h$ such that \eqref{MFE53} holds. In particular, there is $\fn_i$ such that
    \[\sup_{\|u\|\le R}\E_u^\BB\int_{\fn_i}^{\fn_i+1}\dist_V(u_t,K)^2\ddd t\le\frac{\epsilon}{16}.\]
    Together with the Markov property, \eqref{A4}, and \eqref{A6}, this implies that for any $u\in H$,
    \begin{align*}\E_u^\BB\int_{\fn_i}^{\fn_i+1}\dist_V(u_t,K)^2\ddd t\le\frac{\epsilon}{16}+\CC_7\frac{\V(u)}{R^2},
    \end{align*}
    where $\CC_7$ depends on $\RR,\nu,h$ and $\V(u):=(1+\|u\|^2)^2$. Let $M\ge \max_{i=0,\ldots,m}\fn_i$. Proceeding as in the proof of \eqref{MFE36-1}, it follows 
    \begin{align*}
        \E_u^\BB\int_0^M\dist_V(u_t,K)^2\ddd t\le \CC_8(\fn_i+R^{-2})\V(u)+\left(\frac{\epsilon}{16}+\CC_8R^{-2}\right)M
    \end{align*}
    for some constant $\CC_8$ depending on $\RR,\nu,h$. By taking $R$ large enough so that
    \begin{align}\label{MFE56-1}
        \frac{\epsilon}{16}+\CC_8 R^{-2}\le \frac{\epsilon}{8},
    \end{align}
    one has 
    \begin{align}
        \label{MFE57}\E_u^\BB\int_0^M\dist_V(u_t,K)^2\ddd t\le \CC_9\V(u)+\frac{\epsilon}{8}M,
    \end{align}
    where $\CC_9$ depends on $\RR,\epsilon,\nu,h$.

    Let $\PPPPP^\BB_t$ and $\PPPPP^{\BB*}_t$ denote the Markov operators associated with the Markov family $(u_t,\PP^\BB_u)$. Applying \eqref{A4} combined with the Markov property, one has 
    \begin{align*}
        \PPPPP^\BB_M\V(u)\le \fc^M\V(u)+\frac{1-\fc^M}{1-\fc}C_{\RR,\nu,h},\qquad \fc:=e^{-\nu\alpha_0},
    \end{align*}
    which, together with \eqref{MFE57} and choosing $M$ sufficiently large, yields
    \begin{align}\label{MFE58}
        \E^\BB_u\int_{0}^{M}\dist_V(u_t,K)^2\ddd t+\tilde\fa_M\left(\PPPPP^\BB_M\V(u)-\V(u)\right)\le \frac{\epsilon M}{4},
    \end{align}
    where $\tilde\fa_M:=(1-\fc^M)^{-1}\CC_9.$ Define $\FI(jM):=\int_0^{jM}\dist_V(u_t,K)^2\ddd t$ and
    \[\tilde \Gamma_j:=\FI((j+1)M)-\FI(jM)+\tilde \fa_M\left(\V(u_{(j+1)M})-\V(u_{jM})\right).\]
    An application of the Markov property combined with \eqref{MFE58} yields
    \[\E_u^\BB(\tilde\Gamma_j|\FF_{jM})\le \frac{\epsilon M}{4} .\]
    Consequently, it follows that
    \begin{align}
        \label{MFE59}\FI(kM)+\tilde \fa_M(\V(u_{kM})-\V(u))\le \frac{k\epsilon M}{4}+\ssum_{j=0}^{k-1}\tilde d_j,
    \end{align}
    where $\tilde d_j:=\tilde\Gamma_j-\E_u^\BB(\tilde\Gamma_j|\FF_{jM})$ is a sequence of martingale differences.

    Repeating the derivation of \eqref{MFE42-1}, for any $q\ge 2$, one has 
    \[\sup_{j\ge 0}\E^\BB_u|\tilde d_j|^q\lesssim_{q,\RR,\epsilon,\nu,h}\V(u)^q.\]
    Together with the pathwise estimate \eqref{MFE59} and the argument carried out in Step V\!I in the proof of Proposition \ref{propMFE2}, this completes the proof.
    \end{proof}
\section{Mixing for finite-energy solutions}\label{MFE}
    The goal of this section is to establish polynomial mixing for the finite-energy solutions of the stochastic NS equation \eqref{INTRO1}. The proof is based on the classical coupling framework; see Section 3.1.3 in \cite{KS12} and \cite{Shi08}. For the reader's convenience, we briefly describe the main ingredients of this framework below.

    Let $\{\uu_t\}_{t\ge 0}$ be a process valued in $H\times H$, and let $\PP_{\uu}$ be a probability measure on the path space $C([0,\infty);H\times H)$ parameterized by $\uu\in H\times H$. Given a Markov family $(u_t,\PP_u)$, the pair $(\uu_t,\PP_\uu)$ is said to be a coupling of two copies of $(u_t,\PP_u)$, if under $\PP_\uu$, the distributions of $\{\pi_1\uu_t\}_{t\ge 0}$ and $\{\pi_2\uu_t\}_{t\ge 0}$ coincide with that of $\{u_t\}_{t\ge 0}$ under $\PP_{\pi_1\uu}$ and $\PP_{\pi_2\uu}$, respectively. Here, for $i=1,2$, $\pi_i$ denotes the projection onto the $i$-th component. As presented in \cite{NZ24-1}, the uniqueness of the invariant measure and the polynomial convergence estimate \eqref{INTRO13} hold, provided that the following result is true.
    \begin{prop}
        \label{propMFE5}Under the assumption of Theorem \ref{MT2}, for any $q\ge 2$, there exists a time $T>0$, a closed ball $B\subset H$, a coupling $(\uu_t,\PP_{\uu})$ of $(u_t,\PP_u)$, and a stopping time $\sigma$ such that $(\uu_{kT},\PP_{\uu})$ is a Markov process satisfying the following properties.
    \begin{enumerate}
    \item \textbf{Recurrence.}
    There are constants $\delta,C_{\rec}>0$ such that
    \begin{align}\label{MFE60}
        \E_{\uu}e^{\delta\tau_B}\le C_{\rec}\big(1+\|u\|^2+\|u'\|^2\big)
    \end{align}
    for any $\uu=(u,u')\in H\times H$, where
    \begin{align}\label{MFE61}
        \tau_B:=\inf\{k\ge0:\uu_{kT}\in B\times B\}.
    \end{align}
    \item \textbf{Polynomial squeezing.}
    There is a constant $C_{\sq}>0$ such that
    \begin{align}\label{MFE62}
        \|u_t-u'_t\|^2\le C_{\sq}e^{-\frac{\nu\alpha_0}{2}t}\|u-u'\|^2,\qquad 0\le t\le\sigma
    \end{align}
    for any $\uu=(u,u')\in B\times B$, where
    \[u_t:=\pi_1\uu_t,\qquad u'_t:=\pi_2\uu_t.\] 
    Moreover, the stopping time $\sigma$ satisfies the following properties.
    \begin{enumerate}
        \item For any $\uu=(u,u')\in B\times B$,
        \begin{align}\label{MFE63}
        \PP_{\uu}\{\sigma=\infty\}\ge\tfrac{1}{2}.
        \end{align}
        \item There is a constant $C_q>0$
    such that
    \begin{align}
        \label{MFE64}
        \E_{\uu}[\I_{\{\sigma<\infty\}}\sigma^q]\le C_q,
    \end{align}
    and
    \begin{align}
        \label{MFE65}
        \E_{\uu}[\I_{\{\sigma<\infty\}}(\|u_{\sigma}\|^{2q}+\|u'_{\sigma}\|^{2q})]\le C_q
    \end{align}
    for any $\uu\in B\times B$.
    \end{enumerate}
    \end{enumerate}
    \end{prop}
    The proof of this proposition is carried out in Sections \ref{recurrence} and \ref{PS}. To prepare for the proof, we establish necessary groundwork in the next two subsections.

    \subsection{Construction of the coupling}\label{construction}
    In this subsection, we construct a coupling $(\uu_t,\PP_{\uu})$, following the approach developed in \cite{NZ24,NZ24-1}, which is briefly described below. 
    
    Introduce parameters $a>0$, $N\ge1$, and $T>0$, whose values will be specified later. For any $u,u'\in H$, we denote by $\{u(t)\}_{t\ge0}$ and $\{u'(t)\}_{t\ge0}$ the solutions of \eqref{MFE1} issued from $u$ and $u'$, respectively. Let $\{v(t)\}_{t\ge0}$ satisfy the auxiliary equation \eqref{MFE2}. Define
    \begin{align}
        \label{MFE65-1}
        \lambda_T(u,u'):=\DDD(\{v(t)\}_{t\in[0,T]}),\qquad \lambda'_T(u'):=\DDD(\{u'(t)\}_{t\in[0,T]}).
    \end{align}
    By Theorem~1.2.28 in~\cite{KS12}, these measures have a maximal coupling $(\V_T(u,u'),\V'_T(u,u'))$ on a probability space $(\tilde{\Omega},\tilde{\FF},\tilde{\PP})$. In particular,
    \begin{align}
        \label{MFE65-2}\tilde{\PP}\{\V_T(u,u')\neq\V'_T(u,u')\}=\|\lambda_T(u,u')-\lambda'_T(u')\|_{\var}.
    \end{align}
    Let $\{\tv(t)\}_{t\in[0,T]}$ and $\{\tu'(t)\}_{t\in[0,T]}$ denote the flows of $\V_T(u,u')$ and $\V'_T(u,u')$, respectively. Then, $\{\tv(t)\}_{t\in[0,T]}$ satisfies
    \begin{align}
    \label{MFE66}
    \partial_t \tv-\nu L\tv+B(\tv)+a\PPP_N\tv=h+\tilde\eta,\qquad \tv|_{t=0}=u',
    \end{align}
    where $\tilde\eta$ has the same distribution as that of $\eta+a\PPP_N u$. Define $\{\tu(t)\}_{t\in[0,T]}$ as the solution of
    \begin{align}
    \label{MFE67}
    \partial_t \tu-\nu L\tu+B(\tu)+a\PPP_N\tu=h+\tilde\eta,\qquad \tu|_{t=0}=u.
    \end{align}
    As presented in Section 2 of \cite{NZ24-1}, by the uniqueness-in-law of the above stochastic PDE, one has
    \[\DDD(\{\tu(t)\}_{t\in[0,T]})=\DDD(\{u(t)\}_{t\in[0,T]}).\]
    
    For any $u,u'\in H$, $\omega\in \tilde\Omega$, and $t\in[0,T]$, we define
    \[\mathcal{R}_t(u,u',\omega):=\tu(t),\qquad\mathcal{R}'_t(u,u',\omega):=\tu'(t),\qquad\mathcal{R}''_t(u,u',\omega):=\tv(t).\]
    Let $\{(\Omega^k,\FF^k,\PP^k)\}_{k\ge0}$ be independent copies of $(\tilde{\Omega},\tilde{\FF},\tilde{\PP})$ and let $(\Omega,\FF,\PP)$ denote their product. For $\omega=(\omega^0,\omega^1,\ldots)\in\Omega$ and $t=kT+s$ with $s\in(0,T]$, $k\ge0$, we recursively define
    \begin{align} \label{MFE67-1}\tu_t:=\mathcal{R}_s\big(\tu_{kT},\tu'_{kT},\omega^k\big),\quad\tu'_t:=\mathcal{R}'_s\big(\tu_{kT},\tu'_{kT},\omega^k\big),\quad\tv_t:=\mathcal{R}''_s\big(\tu_{kT},\tu'_{kT},\omega^k\big).
    \end{align}
    Let $\uu_t:=(\tu_t,\tu'_t)$ and denote the resulting family by $(\uu_t,\PP_{\uu})$. It is clear that this family is a coupling of two copies of $(u_t,\PP_u)$ and that $(\uu_{kT},\PP_{\uu})_{k\ge0}$ is a time-homogeneous Markov chain. In the following subsections, we shall choose $a$, $N$, and $T$ so that $(\uu_t,\PP_{\uu})$ satisfies the recurrence and polynomial squeezing properties stated in Proposition \ref{propMFE5}.
    \subsection{Girsanov estimate}
    In this subsection, we establish an estimate on the total variational distance between the distributions $\DDD(\{v(t)\}_{t\ge0})$ and $\DDD(\{u'(t)\}_{t\ge0})$. The proof is based on Propositions \ref{propMFE1} and \ref{propMFE4} combined with the Girsanov theorem.
    \begin{lemma}\label{lemmaMFE6}
    For any $\RR>0$, there is an integer $N_1\ge 1$ and a constant $a>0$ depending on $\RR,\nu,h$ such that for any $q\ge 2$, $T,\rho>0$, and $u,u'\in H$,
    \begin{align}\label{MFE68}
    \|\lambda_T(u,u')-\lambda'_T(u')\|_{\var}\le \CC_q\frac{1+\|u\|^{8q}}{(1+\rho)^q}+\frac{\sqrt{6}a}{b^*_{N}\sqrt{\nu\alpha_0}}e^{\frac{\CC_1\rho}{2}}\|u-u'\|,
    \end{align}
    provided that the conditions \eqref{INTRO10} and \eqref{INTRO11} are satisfied with $N\ge N_1$. Here, the probability distributions $\lambda_T(u,u'),\lambda'_T(u')$ are defined in \eqref{MFE65-1},
    \[b^*_{N}:=\min_{1\le j\le N}|b_j|>0,\]
    $\alpha_0$ is the constant in the Poincar\'e inequality, $\CC_1$ is the constant given in \eqref{MFE3}, and $\CC_q$ depends on $q,\RR,\nu,h$.
    \end{lemma}

    \begin{proof} Let $\{u(t)\}_{t\ge0}$ be the solution of \eqref{MFE1} issued from $u$, and let $\{v(t)\}_{t\ge0}$ satisfy the auxiliary equation \eqref{MFE2}. The proof is divided into the following three steps.
    
    \noindent\textbf{Step I.} Fix $\RR>0$. Applying Proposition \ref{propMFE4} with $\epsilon=\frac{\nu\alpha_0}{2\CC_1}$, there is a finite set $K\subset C_0^\infty\cap H$ such that the probabilistic growth estimate \eqref{MFE56} holds. Let $N_1,a$ be the constants indicated in Proposition \ref{propMFE1}, which corresponds to the finite set $K$ and thus depends only on $\RR,\nu,h$. Suppose that the noise coefficients $\{b_j\}_{j\ge 1}$ satisfy \eqref{INTRO10} and \eqref{INTRO11} with $N\ge N_1$. Define
    \begin{align}
        \label{MFE68-1}\Lambda_u(t):=\int_0^t\dist_V(u(s),K)^2\ddd s-\frac{\nu\alpha_0}{2\CC_1}t,\qquad\tau^u_\rho:=\inf\{t\ge0| \Lambda_u(t)\ge \rho\}.
    \end{align}
    Introduce $\{v^\rho(t)\}_{t\ge0}$ as the solution of the truncated auxiliary equation
    \begin{align}
    \label{MFE69}
    \partial_t v^\rho-\nu Lv^\rho+B(v^\rho)+a\I_{\{t<\tau^u_\rho\}}\PPP_N(v^\rho-u)=h+\eta,\qquad v^\rho|_{t=0}=u'.
    \end{align}
    In particular, one has
    \begin{align}\label{MFE69-1}
        v^\rho(t)=v(t),\qquad 0\le t\le\tau^u_\rho.
    \end{align}
    Therefore, for any $t\le\tau^u_\rho$, an application of \eqref{MFE3} yields
    \begin{align}\label{MFE70}
    \|u(t)-v^\rho(t)\|^2&\le\|u-u'\|^2\exp\left(-\nu\alpha_0t+\CC_1\int_0^t\dist_V(u(s),K)^2\ddd s\right)\notag\\&\le e^{\CC_1\rho}e^{-\frac{\nu\alpha_0}{2}t}\|u-u'\|^2.
    \end{align}
    On the other hand, since $\{v^\rho(t)\}_{t\ge0}$ coincides with $\{v(t)\}_{t\ge0}$ on $[0,T]$ whenever $\tau^u_\rho>T$, by applying Proposition~\ref{propMFE4}, its distribution $\lambda^\rho_T(u,u'):=\DDD(\{v^\rho(t)\}_{t\in[0,T]})$ satisfies 
    \begin{align}\label{MFE71}
    \|\lambda^\rho_T(u,u')-\lambda_T(u,u')\|_{\var}&\le \mathbb P\{\tau^u_\rho\le T\}\notag\\&\le\mathbb P\Big\{\sup_{t\ge0}\Lambda_u(t)\ge\rho\Big\}\le C_{q,\RR,\nu,h}\frac{1+\|u\|^{8q}}{(1+\rho)^q}.
\end{align}

    \noindent\textbf{Step I\!I.} It remains to estimate the total variational distance between $\lambda^\rho_T(u,u')$ and $\lambda'_T(u')$. Such an estimate is obtained through the Girsanov theorem. Without loss of generality, we assume that the probability space $(\Omega,\FF,\PP)$ has the following specific structure. The sample space $\Omega:= C_0([0,\infty);H)$ is the space of continuous functions valued in $H$ and vanishing at $t=0$, endowed with the topology of uniform convergence on compact intervals. The probability $\PP$ is the distribution of the Wiener process
    \begin{align}\label{MFE72}
    W(t):=\ssum_{j\ge 1}b_je_j\beta_j(t),
    \end{align}
    and $\FF$ is the completion of the Borel $\sigma$-algebra of $\Omega$. Define a measurable transform
    \[\bfPhi^{u,u'}:\omega_t\mapsto \omega_t-a\int_0^t\I_{\{s<\tau^u_\rho\}}\PPP_N(v^\rho(s)-u(s))\ddd s,\qquad\Omega\to\Omega.\]
    By the pathwise uniqueness of the stochastic NS equation \eqref{MFE1}, one has 
    \[\PP\big\{u'(t,\bfPhi^{u,u'}(\omega))=v^\rho(t,\omega)\mbox{ for any $t\ge0$}\big\}=1.\]
    This yields
    \begin{align}
        \label{MFE73}
        \|\lambda^\rho_T(u,u')-\lambda'_T(u')\|_{\var}\le \|\PP-\bfPhi^{u,u'}_*\PP\|_{\var},
    \end{align}
    where $\bfPhi^{u,u'}_*\PP$ denotes the push-forward measure under the transform $\bfPhi^{u,u'}$. To bound the right-hand side of \eqref{MFE73}, we introduce the following decomposition 
    \[\Omega=\PPP_N\Omega \oplus \QQQQQ_N\Omega, \]
    where $\PPP_N$ and $\QQQQQ_N:=I-\PPP_N$ are understood as natural projections
    \begin{align*}
        \PPP_N:\Omega\to C_0([0,\infty);\PPP_NH),\qquad \QQQQQ_N:\Omega\to C_0([0,\infty);\QQQQQ_NH).
    \end{align*}
    Accordingly, we write $\omega=(\omega^{(1)},\omega^{(2)})$. Thus, the transformation~$\bfPhi^{u,u'}$ can be expressed as
    \[\bfPhi^{u,u'}(\omega^{(1)},\omega^{(2)})=(\bfPsi^{u,u'}(\omega^{(1)},\omega^{(2)}),\omega^{(2)}),\]
    where 
    \[\bfPsi^{u,u'}: \omega_t\mapsto\omega_t^{(1)}-a\int_0^t\I_{\{s<\tau^u_\rho\}}\PPP_N(v^\rho(s,\omega)-u(s,\omega))\ddd s,\qquad\Omega\to C_0([0,\infty);\PPP_NH).\]
    Let $\PP_N:=(\PPP_N)_*\PP$ and $\PP_N^{\perp}:=(\QQQQQ_N)_*\PP$. An application of Lemma 3.3.13 in \cite{KS12} implies
    \begin{align}\label{MFE74}
        \|\PP-\bfPhi^{u,u'}_*\PP\|_{\var}\le \int_{\QQQQQ_N\Omega}\|\bfPsi^{u,u'}_*(\PP_N,\omega^{(2)})-\PP_N\|_{\var}\PP^{\perp}_N(\mathrm{d}\omega^{(2)}).
    \end{align}
    Furthermore, by the Girsanov theorem, for each $\omega^{(2)}$, one has
    \begin{align}\label{MFE75}
    	\|&\bfPsi^{u,u'}_*(\PP_N,\omega^{(2)})-\PP_N\|_{\var} \notag\\&\quad\le\frac{1}{2}\left(\left(\E_N\exp\left(6a^2(b_N^*)^{-2}\int_0^{\tau^u_\rho}\|\PPP_N(v^\rho(t,\omega)-u(t,\omega))\|^2\mathrm{d} t\right)\right)^{\frac{1}{2}}-1\right)^{\frac{1}{2}},
    \end{align}
    if the following Novikov condition  
    \begin{align*}
        \E_N\exp\left(c\int_0^{\tau^u_\rho}\|\PPP_N(v^\rho(t,\omega)-u(t,\omega))\|^2\mathrm{d}t\right)<\infty
    \end{align*}
    is satisfied for any $ c>0$ and $\omega^{(2)}$. Here, $\E_N$ is the expectation with respect to $\PP_N$. 

    \noindent\textbf{Step I\!I\!I.} By applying \eqref{MFE70}, it follows that
    \begin{align*}
        \int_0^{\tau^u_\rho}\|\PPP_N(v^\rho(t)-u(t))\|^2\ddd t\le \int_0^{\tau^u_\rho}\|v^\rho(t)-u(t)\|^2\ddd t\le \frac{2}{\nu\alpha_0}e^{\CC_1\rho}\|u-u'\|^2,
    \end{align*}
    which verifies the Novikov condition  
    \begin{align}
        \label{MFE76}
        \E_N\exp\left(c\int_0^{\tau^u_\rho}\|\PPP_N(v^\rho(t,\omega)-u(t,\omega))\|^2\mathrm{d}t\right)\le \exp\left(\frac{2c}{\nu\alpha_0}e^{\CC_1\rho}\|u-u'\|^2\right).
    \end{align}
    Combining the above estimate and \eqref{MFE73}--\eqref{MFE75}, one has
    \begin{align*}
        \|\lambda^\rho_T(u,u')-\lambda'_T(u')\|_{\var}\le 1\wedge\left(\frac{1}{2}\left(\exp\left(\frac{6a^2}{\nu\alpha_0(b_N^*)^2}e^{\CC_1\rho}\|u-u'\|^2\right)-1\right)^{\frac{1}{2}}\right).
    \end{align*}
    This, together with \eqref{MFE71} and the estimate 
    \[1\wedge\left(\frac{1}{2}\sqrt{e^{s}-1}\right)\le \sqrt{s},\qquad \forall s\ge0,\]
    implies \eqref{MFE68} and thus completes the proof.
    \end{proof}

    \subsection{Recurrence}\label{recurrence}
    In this subsection, we verify the recurrence property stated in Proposition \ref{propMFE5}, by following the argument in Section~5.1 of~\cite{NZ24-1}. We begin with the following irreducibility property of the Markov family $(u_t,\PP_u)$ associated with the stochastic NS equation \eqref{MFE1} in the finite-energy setting.

    \begin{lemma}\label{lemmaMFE7} 
        There is an integer $N_2\ge 1$ depending only on $\nu,h$ such that if \eqref{INTRO11} is satisfied for some $N\ge N_2$, then there exists a function $\fu\in H^1$ with $\|\fu\|\le 1$ depending on $\nu,h,N$ and satisfying the following assertion. For any $R,d>0$, there is $T_0>0$ such that
        \begin{align}\label{MFE77}
        \inf_{u_0\in B_H(R)} P_{T}(u_0,B_{H}(\fu,d))\ge p>0
        \end{align}
        for any $T\ge T_0$. Here, $T_0$ depends on $R,d,\nu$ and $p$ depends on $T,d,\nu,h,N,\{b_j\}_{j\ge 1}$.
    \end{lemma}
    The proof follows that of Lemma~5.1 in~\cite{NZ24-1}, with the Poincar\'e inequality providing the dissipation supplied there by the damping term. We therefore omit the details. Moreover, combining Lemmas~\ref{lemmaMFE6} and~\ref{lemmaMFE7} with Proposition~\ref{propMFE4}, one can obtain the irreducibility of the coupling $(\uu_t,\PP_{\uu})$. The argument follows that of Proposition~4.2 in~\cite{NZ24} for the complex Ginzburg--Landau equation and is likewise omitted.
    \begin{lemma}\label{lemmaMFE8} 
        For any $\RR>0$, there is an integer $N_{\rec}\ge 1$ depending on $\RR,\nu,h$ such that if \eqref{INTRO10} and \eqref{INTRO11} are satisfied for some $N\ge N_{\rec}$, then there exists a function $\fu\in H^1$ with $\|\fu\|\le 1$ depending on $\nu,h,N$ and satisfying the following assertion. For any $R,d>0$, there is $T_0>0$ such that 
        \begin{align}\label{MFE78}
        \inf_{\uu\in B_H(R)\times B_H(R)}\PP_{\uu}\{\uu_T\in B_H(\fu,d)\times B_H(\fu,d)\}\ge p>0
        \end{align}
        for every $T\ge T_0$. Here, $T_0$ depends on $\RR,R,d,\nu,h,N,\{b_j\}_{j\ge1}$, while $p$ additionally depends on $T$.
    \end{lemma}
    Finally, by \eqref{A4}, the function
    \[F(u,u'):=1+\|u\|^2+\|u'\|^2\]
    is a Lyapunov function for the coupling process $(\uu_{kT},\PP_{\uu})$. Combining this fact with \eqref{MFE78}, we obtain the following recurrence property, as explained in Section 3 in \cite{Shi08}.

    \begin{prop}\label{propMFE6}
    Fix arbitrary $\RR>0$ and let $N_{\rec}$ be the integer given in Lemma \ref{lemmaMFE8}. Suppose that \eqref{INTRO10} and \eqref{INTRO11} are satisfied for some $N\ge N_{\rec}$. Let $\fu$ be the corresponding function given in Lemma \ref{lemmaMFE8}. Then, for any $d>0$, there is $T_{\rec}>0$ such that for any $T\ge T_{\rec}$ and $\uu=(u,u')\in H\times H$,
    \begin{align}\label{MFE79}
    \E_{\uu}e^{\delta_T\tau_d}\le C_T\big(1+\|u\|^2+\|u'\|^2\big),
    \end{align}
    for some constants $\delta_T,C_T$. Here
    \[\tau_d:=\inf\left\{k\ge0|\uu_{kT}\in \overline B_H(\fu,d)\times\overline B_H(\fu,d)\right\},\]
    and the constant $T_{\rec}$ depends on $\RR,d,\nu,h,N,\{b_j\}_{j\ge1}$, while $\delta_T$ and $C_T$ additionally depend on $T$.
    \end{prop}

    \subsection{Polynomial squeezing and proof of Theorem \ref{MT2}}\label{PS}
    This subsection is devoted to the verification of the polynomial squeezing property stated in Proposition \ref{propMFE5} for the coupling process $(\uu_t,\PP_\uu)$. To simplify the presentation, we denote by $u_t,u_t',v_t$ the processes constructed in \eqref{MFE67-1}. Let
    \begin{align}
        \label{MFE79-0}\sigma:=\sigma_1\wedge\tau,\qquad \sigma_1:=\inf\{t\ge0|u_t'\neq v_t\},\qquad\tau:=\tau^u_\rho\wedge\tau^{u'}_\rho,
    \end{align}
    where $\tau^u_\rho,\tau^{u'}_\rho$ are defined in \eqref{MFE68-1}. Introduce the random events
    \[\QQQ_k':=\{kT\le \sigma< (k+1)T,\ \sigma_1\ge \tau\},\quad \QQQ_k^{''}:=\{kT\le \sigma<(k+1)T,\ \sigma_1<\tau\},\quad k\ge 0.\]
    Here, we recall that $T>0$ is the time step used in the construction \eqref{MFE67-1}.
    \begin{lemma}\label{lemmaMFE9}
        For any $\RR>0$, there is an integer $N_{\sq}\ge 1$ such that if \eqref{INTRO10} and \eqref{INTRO11} are satisfied for some $N\ge N_{\sq}$, then the following assertion holds. For any $q\ge 2$, there exist constants $\rho>0$, $d\in(0,1]$, and $T_{\sq}>0$ depending on $q,\RR,\nu,h,N,\{b_j\}_{j\ge 1}$ such that 
        \begin{align}
            \label{MFE79-1}\PP_{\uu}(\QQQ_k')\vee \PP_{\uu}(\QQQ_k^{''})\le \frac{q-1}{16q}(k+1)^{-4q}
        \end{align}
        for any $k\ge 0$, $T\ge T_{\sq}$, and $\uu=(u,u')$ satisfying 
        \begin{align}
            \label{MFE79-2}\uu\in\overline{B}_{H}(0,2)\times \overline{B}_{H}(0,2),\qquad \|u-u'\|\le 2d.
        \end{align}
    \end{lemma}
    \begin{proof}Fix arbitrary $\RR>0$. Let $N_{\sq}=N_1$, where $N_1$ is given in Lemma \ref{lemmaMFE6}. Suppose that \eqref{INTRO10} and \eqref{INTRO11} are satisfied for some $N\ge N_{\sq}$. 
    
        \noindent\textbf{Step I.} We begin with the estimate for the random event $\QQQ_k'$. Applying Proposition~\ref{propMFE4} with $q$ replaced by $4q$ and $\epsilon=\frac{\nu\alpha_0}{2\CC_1}$, where $\CC_1$ is the constant in \eqref{MFE3} and $\alpha_0$ is the constant in the Poincar\'e inequality, it follows 
        \begin{align}\label{MFE80}
            \PP_{\uu}(\QQQ_k')&\le \PP_{u}\{kT\le \tau^u_\rho<(k+1)T\}+\PP_{u'}\{kT\le \tau^{u'}_\rho<(k+1)T\}\notag\\&\le \PP_{u}\Big\{\sup_{t\ge kT}\Lambda_u(t)\ge\rho\Big\}+\PP_{u'}\Big\{\sup_{t\ge kT}\Lambda_{u'}(t)\ge\rho\Big\}\le \frac{C_{q,\RR,\nu,h}}{(1+kT+\rho)^{4q}}.
        \end{align}

        \noindent\textbf{Step I\!I.} Next, we turn to the random event $\QQQ_k''$. For $k=0$, by applying Lemma \ref{lemmaMFE6} combined with \eqref{MFE65-2}, one has 
        \begin{align}
            \label{MFE81}
            \PP_{\uu}(\QQQ_0'')&\le \PP_u\{0\le \sigma_1<T\}=\|\lambda_T(u,u')-\lambda'_T(u')\|_{\var}\notag\\&\le \frac{\CC_q}{(1+\rho)^{4q}}+\frac{\sqrt{6}a}{b^*_{N}\sqrt{\nu\alpha_0}}e^{\frac{\CC_1\rho}{2}}\|u-u'\|,
        \end{align}
        where $\CC_q,a$ are given in Lemma \ref{lemmaMFE6}. The case $k\ge 1$ is estimated by applying the Markov property, which gives
        \begin{align}\label{MFE82}
            \PP_{\uu}(\QQQ_k'')&=\PP(\QQQ_k''\cap\{\sigma\ge kT\})=\E_{\uu}\left(\I_{\{\sigma\ge kT\}}\E_{\uu}\left(\I_{\QQQ_k''}\big|\FF_{kT}\right)\right)\notag\\&\le\E_{\uu}\left(\I_{\{\sigma\ge kT\}}\PP_{\uu_{kT}}\{0\le \sigma_1<T\}\right),
        \end{align}
        where, by applying Lemma \ref{lemmaMFE6} with $q,\rho$ replaced by $4q,\rho+\frac{\nu\alpha_0}{4\CC_1}kT$, respectively, one has 
        \begin{align}\label{MFE83}
            \PP_{\uu_{kT}}\{0\le \sigma_1<T\}\le \CC_q\frac{1+\|u_{kT}\|^{32q}}{(1+\rho+\frac{\nu\alpha_0}{4\CC_1}kT)^{4q}}+\frac{\sqrt{6}a}{b^*_{N}\sqrt{\nu\alpha_0}}e^{\frac{\CC_1\rho}{2}}e^{\frac{\nu\alpha_0}{8}kT}\|u_{kT}-u_{kT}'\|.
        \end{align}
        On the other hand, by \eqref{MFE69-1}, \eqref{MFE70}, and the definition \eqref{MFE79-0} of $\sigma$, it follows that
        \begin{align}
            \label{MFE84}v^\rho_t=v_t=u'_t,\qquad \|u_t-v^\rho_t\|\le e^{\frac{\CC_1\rho}{2}}e^{-\frac{\nu\alpha_0}{4}t}\|u-u'\|
        \end{align}
        for any $t\le \sigma$. Here, $v^\rho_t$ is the solution of the truncated auxiliary equation \eqref{MFE69}. Thus, by combining \eqref{MFE82}--\eqref{MFE84} with the moment estimate \eqref{A4}, it follows that
        \begin{align}
            \label{MFE85}
            \PP_{\uu}(\QQQ_k'')\le \frac{\CC_qC_{q,\nu,h,\FB_0}}{(1+\rho+\frac{\nu\alpha_0}{4\CC_1}kT)^{4q}}+\frac{\sqrt{6}a}{b^*_{N}\sqrt{\nu\alpha_0}}e^{\CC_1\rho}e^{-\frac{\nu\alpha_0}{8}kT}\|u-u'\|.
        \end{align}
        Summarizing the estimates \eqref{MFE81} and \eqref{MFE85} gives
        \begin{align}
            \label{MFE86}
            \PP_{\uu}(\QQQ_k'')\le\frac{C_{q,\RR,\nu,h}}{(1+\rho+ kT)^{4q}}+\frac{C_{\RR,\nu,h}}{b^*_N}e^{\CC_1\rho}e^{-\frac{\nu\alpha_0}{8}kT}d.
        \end{align}
        \noindent\textbf{Step I\!I\!I.} Now, we choose $\rho,T_{\sq}$ sufficiently large so that 
        \[\frac{C_{q,\RR,\nu,h}}{\rho^{4q}\wedge T_{\sq}^{4q}}<\frac{q-1}{32q},\]
        and then take $d\ll1$ sufficiently small such that
        \[\frac{C_{\RR,\nu,h}}{b^*_N}e^{\CC_1\rho}e^{-\frac{\nu\alpha_0}{8}kT}d\le \frac{q-1}{32q}(1+k)^{-4q}.\]
        This, together with \eqref{MFE80} and \eqref{MFE86}, yields the desired estimate \eqref{MFE79-1}. 
    \end{proof}
    Now we are in a position to prove Proposition \ref{propMFE5}.
    \begin{proof}[Proof of Proposition \ref{propMFE5}]
        Fix arbitrary $q\ge 2$ and $\RR>0$. We choose the parameter $a$ as given in Lemma \ref{lemmaMFE6} and $N_0:=N_{\rec}\vee N_{\sq}$. Suppose that \eqref{INTRO10} and \eqref{INTRO11} are satisfied with $N\ge N_0$ and set $T:=T_{\sq}\vee T_{\rec}$ as the time step corresponding to $N,d$ with the parameter $d$ given in Lemma \ref{lemmaMFE9}. Let $(\uu_t,\PP_{\uu})$ be the coupling process constructed in Section \ref{construction}, which corresponds to these parameters $a,N_0,T$. Based on these parameters, Proposition \ref{propMFE6} and Lemma \ref{lemmaMFE9} are now applicable.

        The closed ball in Proposition \ref{propMFE5} is defined as $B:=\overline{B}_H(\fu,d)$, and the recurrence property is obtained by applying Proposition \ref{propMFE6}. We turn to the verification of the polynomial squeezing properties \eqref{MFE62}--\eqref{MFE65}. Since $\|\fu\|\le 1$ and $d\le 1$, every $\uu\in B\times B$ satisfies \eqref{MFE79-2}. This, together with \eqref{MFE84}, gives \eqref{MFE62}. As for \eqref{MFE63}, we apply Lemma \ref{lemmaMFE9}, which gives 
        \[\PP_{\uu}\{\sigma=\infty\}\ge 1-\ssum_{k=0}^{\infty}\PP_{\uu}\{\sigma\in[kT,(k+1)T)\}\ge \frac{1}{2}>0.\]
        To prove \eqref{MFE64}, we derive from Lemma \ref{lemmaMFE9} that
        \begin{align*}
            \E_{\uu}\left(\I_{\{\sigma<\infty\}}\sigma^q\right)&\le \ssum_{k=0}^{\infty}\E_{\uu}\left(\I_{\{\sigma\in[kT,(k+1)T)\}}\sigma^q\right)\\&\lesssim_{q,T} \ssum_{k=0}^{\infty}(k+1)^q\PP_{\uu}\{\sigma\in[kT,(k+1)T)\}\\&\lesssim_{q,T} \ssum_{k=0}^{\infty}(k+1)^{-3q}\le C_q.
        \end{align*}
        Finally, to establish \eqref{MFE65}, we need the following auxiliary estimate
        \begin{align}
            \label{MFE87}\sup_{k\ge 0} \E_u\sup_{t\in[kT,(k+1)T)}\|u_t\|^{4q}\le C_{q,T},
        \end{align}
        which follows from an application of \eqref{A4}, \eqref{A6}, and the Markov property. This, together with Lemma \ref{lemmaMFE9} and the Cauchy--Schwarz inequality, implies that
        \begin{align*}
        \E_{\uu}&\left(\I_{\{\sigma<\infty\}}\left(\|u_\sigma\|^{2q}+\|u'_\sigma\|^{2q}\right)\right)\\&\le \ssum_{k=0}^{\infty}\E_{\uu}\left(\I_{\{\sigma\in[kT,(k+1)T)\}}\sup_{t\in[kT,(k+1)T)}\left(\|u_t\|^{2q}+\|u'_t\|^{2q}\right)\right)\\&\lesssim_{q,T}\ssum_{k=0}^{\infty}\left(\PP_{\uu}\{\sigma\in[kT,(k+1)T)\}\right)^{\frac{1}{2}}\le C_{q,T}.
        \end{align*} 
        This completes the proof of Proposition \ref{propMFE5}.
    \end{proof}
    Finally, we prove Theorem \ref{MT2}.
    \begin{proof}[Proof of Theorem \ref{MT2}]
        The uniqueness of the invariant measure and the second statement are direct consequences of Proposition \ref{propMFE5}; see Section 2 in \cite{NZ24-1}. We turn to the third statement. Fix arbitrary $\lambda\in \PPPP(H)$ and function $f\in\lip(H)$ satisfying $\|f\|_{\lip(H)}\le 1$. An application of \eqref{INTRO13} gives  
    \begin{align*}
    |\langle f,\DDD(u(t;\lambda))\rangle-\langle f,\mu\rangle|&=|\langle f,\PPPPP^*_t\lambda\rangle-\langle f,\mu\rangle|=|\langle \PPPPP_tf,\lambda\rangle-\langle f,\mu\rangle|\\&\le \int_{B_{H}(R)}|\PPPPP_tf(u)-\langle f,\mu\rangle|\lambda(\ddd u)+2\lambda(B_{H}(R)^c)\\&\le C_{q}(1+R^2)(1+t)^{-q}+2\lambda(B_{H}(R)^c).
    \end{align*}
    Since $\lim_{R\to\infty}\lambda(B_{H}(R)^c)= 0$, by taking supremum in $f$, sending $t\to\infty$, and then letting $R\to\infty$, the convergence \eqref{INTRO14} holds. Now it remains to prove \eqref{INTRO12}. By the Portmanteau theorem, \eqref{INTRO13}, and \eqref{A4} with $q=1$, one has
    \begin{align*}
        \int_{H}\|u\|^2\ddd \mu\le \liminf_{t\to\infty}\E\|u(t;0)\|^2\le C_{\nu,h}(1+\FB_0),
    \end{align*}
    which, together with taking $t\to\infty$, implies
    \begin{align}\label{MFE88}
        \int_{H}\|u\|^2\ddd \mu<\infty.
    \end{align}
    To prove the support property of $\mu$, we apply the stationarity of the invariant measure $\mu$ again combined with \eqref{MFE88} and the energy equality \eqref{A1}, which gives
    \[\nu\int_{H} \|\nabla u\|^2\mu(\ddd u)=\nu\E\int_0^1\|\nabla u(t;\mu)\|^2\ddd t\lesssim_\nu \|h\|^2+\FB_0<\infty.\]
    Therefore, any stationary process $\{u_{\inv}(t)\}_{t\ge0}$ has $V$-valued initial data. Together with the fact $\FB_1<\infty$ and standard regularity result for the stochastic NS equation \eqref{MFE1}, this implies 
    \[u_{\inv}\in L^2_{\loc} (0,\infty;H^2\cap V).\]
    By stationarity again, it follows
    \[\mu(H^2\cap V)=\mathbb E\int_0^1\I_{\{u_{\inv}(t)\in H^2\cap V\}}\ddd t=1,\]
    which, combined with the Sobolev embedding, implies $\mu(H\cap L^\infty)=1$.
    \end{proof}

\section{Decay estimates for infinite-energy solutions of the homogeneous NS equation}\label{DE}
In this section, our goal is to establish decay estimates for infinite-energy solutions of the homogeneous NS equation \eqref{INTRO8}. The main result is presented below.
\begin{prop}\label{propDE0}
    Let $w(t)$ be an infinite-energy solution of the homogeneous NS equation \eqref{INTRO8} issued from the initial data $w_0\in H_b\cap L^\infty$. Then, for any $R>0$, there are constants $\CCC_1,\alpha_1$ depending on $R,\nu$ such that for any $w_0$ satisfying $\|w_0\|_{L^2_b}\le R$,
    \begin{align}
        \label{DE1}
        \int_0^\infty\|w(t)\|_{\infty}^2\ddd t\le \CCC_1\big(1+\log(1+\|w_0\|_{\infty})\big),
    \end{align}
    and 
    \begin{align}
        \label{DE2}\|w(t)\|_{L^2_b}+\|w(t)\|_{\infty}\le \CCC_1e^{-\alpha_1 t},\qquad \forall t\ge 1.
    \end{align}
\end{prop}
We divide the proof into two steps, each treated in a separate subsection below. Throughout the paper, we need the following weight function 
\begin{align}
    \label{DE3}\theta_{\varepsilon,\fz}(x):=\frac{1}{\sqrt{1+\varepsilon^2(x-\fz)^2}},
\end{align}
which was introduced in \cite{Zel07}. The following properties play an important role. Although proofs of these properties can be found in \cite{Zel03,Zel07,AZ14}, we collect their proofs in Section~\ref{weight} for the reader's convenience.
\begin{prop}\label{propDE1}
    The following statements hold for the weight function $\theta_{\varepsilon,\fz}$.
    \begin{enumerate}
        \item For any $\varepsilon\in (0,1)$ and $p\in(1,\infty)$, it follows
        \begin{align}
            \label{DE4}
            \|f\|_{L^p_b}\lesssim_p \sup_{\fz\in\R}\|\theta_{\varepsilon,\fz}f\|_p\lesssim_p \varepsilon^{-\frac{1}{p}}\|f\|_{L^p_b}.
        \end{align}
        \item For any $\varepsilon\in (0,1)$, it follows
        \begin{align}
            \label{DE5}|\partial_x\theta_{\varepsilon,\fz}(x)|\le \varepsilon\theta_{\varepsilon,\fz}^2(x)
        \end{align}
        \item For any $\varepsilon\in (0,1)$ and integer $j\ge 1$, it follows
        \begin{align}
            \label{DE6}\|\partial_x^j\log \theta_{\varepsilon,\fz}\|_\infty\lesssim_j \varepsilon^j.
        \end{align}
    \end{enumerate}
\end{prop}

\subsection{Estimate on unit time interval}
We begin with an estimate of $w(t)$ on the unit time interval. The main result is given below.
\begin{prop}\label{propDE3}
    Let $w(t)$ be an infinite-energy solution of the homogeneous NS equation \eqref{INTRO8} issued from the initial data $w_0\in H_b\cap L^\infty$. Then, for any $R>0$, there is a constant $\CCC_2$ depending on $R,\nu$ such that for any $w_0$ satisfying $\|w_0\|_{L^2_b}\le R$,
    \begin{align}
        \label{DE7}
        \int_0^1\|w(t)\|_{\infty}^2\ddd t\le \CCC_2\big(1+\log(1+\|w_0\|_{\infty})\big).
    \end{align}
\end{prop}
\begin{proof}
    \textbf{Step I.} First, we establish an estimate around $t=0$ through the Duhamel formulation 
    \begin{align}
        \label{DE8}
        w(t)=S(t)w_0-\int_0^tS(t-s)\Pi\dive (w\otimes w)(s)\ddd s,
    \end{align}
    and the estimates for the Stokes semigroup in uniformly local spaces. For any $T>0$, define
    \[\|w\|_{\YYY_T}:=\sup_{0<t\le T}(\|w(t)\|_{L^4_b}+t^{\frac{1}{4}}\|w(t)\|_\infty).\]
    By applying \eqref{STOKES1} combined with the Gagliardo-–Nirenberg inequality
    \begin{align}
        \label{DE9}\|f\|_{\infty}\lesssim\|f\|^{\frac{1}{2}}_{L^4_b}\| f\|_{W^{1,4}_b}^{\frac{1}{2}}
    \end{align}
    and the Poincar\'e inequality, one has 
    \begin{align*}
        \|S(t)w_0\|_{L^4_b}\lesssim_{\nu}\|w_0\|_{L^4_b},\qquad \|S(t)w_0\|_\infty\lesssim_\nu t^{-\frac{1}{4}}\|w_0\|_{L^4_b},
    \end{align*}
    which implies
    \begin{align}\label{DE10}
        \|S(\cdot)w_0\|_{\YYY_T}\lesssim_{\nu} \|w_0\|_{L^4_b}.
    \end{align}
    Similarly, an application of \eqref{STOKES2} gives 
    \begin{align}\label{DE11}
        \|S(t-s)\Pi\dive (w\otimes w)(s)\|_{L^4_b}\!\lesssim_\nu \frac{1}{(t-s)^{\frac{1}{2}}}\|(w\otimes w)(s)\|_{L^4_b}\lesssim_\nu \!\frac{1}{s^{\frac{1}{4}}(t-s)^{\frac{1}{2}}}\|w\|^2_{\YYY_T},
    \end{align}
    and 
    \begin{align*}
        \|\nabla S(t-s)\Pi\dive (w\otimes w)(s)\|_{L^4_b}\lesssim_\nu \frac{1}{t-s}\|(w\otimes w)(s)\|_{L^4_b}\lesssim_\nu \frac{1}{s^{\frac{1}{4}}(t-s)}\|w\|^2_{\YYY_T}.
    \end{align*}
    Together with the Gagliardo--Nirenberg inequality \eqref{DE9}, this yields
    \begin{align*}
        \|S(t-s)\Pi\dive (w\otimes w)(s)\|_\infty\lesssim_\nu \frac{1}{s^{\frac{1}{4}}(t-s)^{\frac{3}{4}}}\|w\|^2_{\YYY_T}.
    \end{align*}
    As a result, by combining this and \eqref{DE11}, one has
    \begin{align*}
        \left\|\int_0^tS(t-s)\Pi\dive (w\otimes w)(s)\ddd s\right\|_{L^4_b}&\lesssim_\nu \|w\|^2_{\YYY_T}\int_0^ts^{-\frac{1}{4}}(t-s)^{-\frac{1}{2}}\ddd s\\&=t^{\frac{1}{4}}B\left(\tfrac{3}{4},\tfrac{1}{2}\right)\|w\|^2_{\YYY_T},
    \end{align*}
    and 
    \begin{align*}
        t^{\frac{1}{4}}\left\|\int_0^tS(t-s)\Pi\dive (w\otimes w)(s)\ddd s\right\|_{\infty}&\lesssim_\nu  t^{\frac{1}{4}}\|w\|^2_{\YYY_T}\int_0^ts^{-\frac{1}{4}}(t-s)^{-\frac{3}{4}}\ddd s\\&=t^{\frac{1}{4}}B\left(\tfrac{3}{4},\tfrac{1}{4}\right)\|w\|^2_{\YYY_T},
    \end{align*}
    where $B(\cdot,\cdot)$ denotes the Beta function. Consequently, it follows that 
    \begin{align*}
        \left\|\int_0^\cdot S(\cdot-s)\Pi\dive (w\otimes w)(s)\ddd s\right\|_{\YYY_T}\lesssim_\nu T^{\frac{1}{4}}\|w\|^2_{\YYY_T},
    \end{align*}
    which, combined with \eqref{DE10}, leads to
    \[\|w\|_{\YYY_T}\le C_\nu \|w_0\|_{L^4_b}+C_\nu T^{\frac{1}{4}}\|w\|^{2}_{\YYY_T}.\]
    This, together with a bootstrap argument, yields 
    \begin{align}\label{DE12}
        \|w\|_{\mathcal Y_T}\le 2C_\nu \|w_0\|_{L^4_b},
    \end{align}
    provided that $T^{1/4}\|w_0\|_{L^4_b}\le \frac{1}{8C_\nu^2}$.

    \noindent\textbf{Step I\!I.} Let 
    \[T_0:=1\wedge\frac{1}{4096 C_\nu^8(1+\|w_0\|_\infty^2)(1+R^2)}.\]
    By the H\"older inequality, it follows that 
    \[T_0^{\frac{1}{4}}\|w_0\|_{L^4_b}\le \frac{1}{8C_\nu^2},\]
    which, together with \eqref{DE12} and the H\"older inequality again, gives
    \begin{align}
        \label{DE13}\int_0^{T_0}\|w(t)\|^2_\infty \ddd t\le 8C^2_\nu \|w_0\|^2_{L^4_b} T_0^{\frac{1}{2}}\le \frac{1}{8C_\nu^2}.
    \end{align}
    We turn to the estimate of $w(t)$ on $[T_0,1]$. By applying Theorem 4.1 and Corollary 6.4 in \cite{AZ14}, one has 
    \[\|w(t)\|_{W^{2,2}_b}\le C_R t^{-1},\qquad \forall t\in(0,1],\]
    and 
    \[\|w(t)\|_{L^{2}_b}\le C_R,\qquad \forall t\ge 0.\]
    Together with the Gagliardo–-Nirenberg inequality
    \begin{align}
        \label{DE13-1}\|f\|_\infty\lesssim\|f\|^{\frac{1}{2}}_{L^2_b}\|f\|^{\frac{1}{2}}_{W^{2,2}_b},
    \end{align}
    this implies 
    \begin{align*}
        \int_{T_0}^1\|w(t)\|^2_\infty\ddd t\lesssim_R \log T_0^{-1}\lesssim_R C_\nu+\log(1+R^2)+\log(1+\|w_0\|^2_\infty)
    \end{align*}
    By combining this and \eqref{DE13}, we obtain \eqref{DE7} as desired.
\end{proof}
\subsection{Long-time estimate}
In this subsection, we aim to establish the following exponential decay estimate for $w(t)$ on $[1,\infty).$ The proof follows the idea of Theorem 6.1 in \cite{AZ14} with a refined treatment of weighted energy estimates. This yields exponential decay to $0$, when the cross-sectional flux vanishes, sharpening the dissipative bound stated in that theorem.
\begin{prop}\label{propDE4}
    Let $w(t)$ be an infinite-energy solution of the homogeneous NS equation \eqref{INTRO8} issued from the initial data $w_0\in H_b\cap L^\infty$. Then, for any $R>0$, there are constants $\CCC_1,\alpha_1$ depending on $R,\nu$ such that for any $w_0$ satisfying $\|w_0\|_{L^2_b}\le R$,
    \begin{align}
        \label{DE14}
        \|w(t)\|_{L^2_b}+\|w(t)\|_{\infty}\le \CCC_1e^{-\alpha_1 t}
    \end{align}
    for any $t\ge 1$.
\end{prop}
\begin{proof} \textbf{Step I.} Let $w(t)$ be an infinite-energy solution of \eqref{INTRO8} issued from $w_0$. Introduce an auxiliary function $v_{\varepsilon,\fz}(t)$ which satisfies the following backward generalized Stokes equation
\begin{align}
    \label{DE15}\begin{cases}
    -\partial_t v_{\varepsilon,\fz}+\nabla \pp_{\varepsilon,\fz}=\nu\Delta v_{\varepsilon,\fz},\\
    \dive v_{\varepsilon,\fz}=2\theta_{\varepsilon,\fz}\partial_x\theta_{\varepsilon,\fz}w_1,\\
    v_{\varepsilon,\fz}|_{t=T}=\left(0,2\theta_{\varepsilon,\fz}\partial_x\theta_{\varepsilon,\fz}\JJ w(T)\right),\qquad v_{\varepsilon,\fz}|_{\partial D}=0,
\end{cases}
\end{align}
where $\JJ w:=\int_{-1}^y w_{1}(x,y')\ddd y'$ denotes the stream function of the velocity field $w$. According to Corollary 3.4 in \cite{AZ14}, the above equation is well-posed and there exists a constant $\CCC_3$ depending only on $\nu$ such that 
\begin{align}
    \label{DE16}
    \|\theta^{-2}_{\varepsilon,\fz}v_{\varepsilon,\fz}\|_{C([0,T];L^2)}+\|\theta^{-2}_{\varepsilon,\fz}v_{\varepsilon,\fz}\|_{C([0,T];L^3)}\le \CCC_3\varepsilon \|\theta_{\varepsilon,\fz} w\|_{C([0,T];L^2)}.
\end{align}
By testing \eqref{INTRO8} against $\theta_{\varepsilon,\fz}^2 w-v_{\varepsilon,\fz}$ and using the fact
\[\dive (\theta_{\varepsilon,\fz}^2 w-v_{\varepsilon,\fz})=0\]
and \eqref{DE15}$_1$, it follows that
\begin{align}
    \label{DE17}
    \frac{\ddd}{\ddd t}\left(\frac{1}{2}\|\theta_{\varepsilon,\fz}w\|^2-\langle w,v_{\varepsilon,\fz}\rangle\right)+\nu\langle \nabla w,\nabla (\theta_{\varepsilon,\fz}^2 w)\rangle=-\langle (w\cdot\nabla) w, \theta_{\varepsilon,\fz}^2 w-v_{\varepsilon,\fz}\rangle.
\end{align}
The above weighted energy identity can be rigorously justified by an approximation argument; see Corollary 3.5 in \cite{AZ14}. We estimate the terms in \eqref{DE17} in the following way. First, by applying \eqref{DE5}, the fact $\theta_{\varepsilon,\fz}\le 1$, and the Poincar\'e inequality, one has
\begin{align}\label{DE18}
        \langle \nabla w,\nabla (\theta_{\varepsilon,\fz}^2 w)\rangle&=\|\theta_{\varepsilon,\fz}\nabla w\|^2+2\langle \partial_xw,\theta_{\varepsilon,\fz}\partial_x\theta_{\varepsilon,\fz} w\rangle\notag\\&\ge\|\theta_{\varepsilon,\fz}\nabla w\|^2-2\varepsilon\|\theta_{\varepsilon,\fz} \nabla w\|\|\theta_{\varepsilon,\fz}  w\|\ge (1-\sqrt{\alpha_0}\varepsilon)\|\theta_{\varepsilon,\fz}\nabla w\|^2.
\end{align}
Next, by integrating by parts combined with an application of \eqref{DE5} and the Ladyzhenskaya inequality, it follows
\begin{align}
    \label{DE19}
    |\langle (w\cdot\nabla) w, \theta_{\varepsilon,\fz}^2 w\rangle|&=|\langle |w|^2, w_1\theta_{\varepsilon,\fz}\partial_x \theta_{\varepsilon,\fz} \rangle|\notag\\
    &\le \varepsilon\|\theta_{\varepsilon,\fz} w\|\|\theta_{\varepsilon,\fz} w\|_4^2\lesssim \varepsilon\|\theta_{\varepsilon,\fz} w\|^2\|\theta_{\varepsilon,\fz}\nabla w\|.
\end{align}
On the other hand, by using the fact $\theta_{\varepsilon,\fz}\le 1$ and the Sobolev embedding $H^1\hookrightarrow L^6$, one has 
\begin{align}
    \label{DE20}
    |\langle (w\cdot\nabla) w, v_{\varepsilon,\fz}\rangle|\le \|\theta_{\varepsilon,\fz} w\|_6\|\theta_{\varepsilon,\fz} \nabla w\|\|\theta^{-2}_{\varepsilon,\fz}v_{\varepsilon,\fz}\|_3\lesssim\|\theta_{\varepsilon,\fz} \nabla w\|^2\|\theta^{-2}_{\varepsilon,\fz}v_{\varepsilon,\fz}\|_3,
\end{align}
and 
\begin{align}
    \label{DE21}
    |\langle w, v_{\varepsilon,\fz}\rangle|\le\|\theta_{\varepsilon,\fz} w\|\|\theta^{-2}_{\varepsilon,\fz}v_{\varepsilon,\fz}\|.
\end{align}
Combining \eqref{DE16}--\eqref{DE21} and choosing $\varepsilon>0$ such that $1-\sqrt{\alpha_0}\varepsilon>\frac{1}{2}$, it follows that for any $T>0$,
\begin{align*}
    \|\theta_{\varepsilon,\fz}w(T)\|^2+&\nu\int_0^T\|\theta_{\varepsilon,\fz}\nabla w\|^2\ddd t\lesssim \|\theta_{\varepsilon,\fz} w(T)\|\|\theta^{-2}_{\varepsilon,\fz}v_{\varepsilon,\fz}(T)\| +\|\theta_{\varepsilon,\fz} w_0\|\|\theta^{-2}_{\varepsilon,\fz}v_{\varepsilon,\fz}(0)\|\\&+\|\theta_{\varepsilon,\fz}w_0\|^2+\varepsilon \int_0^T\|\theta_{\varepsilon,\fz} w\|^2\|\theta_{\varepsilon,\fz}\nabla w\|\ddd t+\int_0^T\|\theta_{\varepsilon,\fz} \nabla w\|^2\|\theta^{-2}_{\varepsilon,\fz}v_{\varepsilon,\fz}\|_3\ddd t,
\end{align*}
which, together with \eqref{DE15}$_3$, \eqref{DE16}, and \eqref{DE5}, implies that 
\begin{align*}
    \|\theta_{\varepsilon,\fz}w(T)\|^2+\nu\int_0^T\|\theta_{\varepsilon,\fz}\nabla w\|^2\ddd t&\le \CCC_4\varepsilon\|\theta_{\varepsilon,\fz}w(T)\|^2+\CCC_4\varepsilon\MMM_T \|\theta_{\varepsilon,\fz} w_0\|\\&\quad+\|\theta_{\varepsilon,\fz}w_0\|^2+\CCC_4\varepsilon\MMM_T \int_0^T\|\theta_{\varepsilon,\fz}\nabla w\|^2\ddd t,
\end{align*}
where $\CCC_4$ depends only on $\nu$ and 
\[\MMM_T:= \|\theta_{\varepsilon,\fz} w\|_{C([0,T];L^2)}.\]
Consequently, by shrinking $\varepsilon>0$ so that $\CCC_4\varepsilon<\frac{1}{4}$, one gets 
\begin{align}\label{DE22}\|\theta_{\varepsilon,\fz}w(T)\|^2+(\nu-\CCC_4\varepsilon\MMM_T)\int_0^T\|\theta_{\varepsilon,\fz}\nabla w\|^2\ddd t\le \CCC_4\varepsilon\MMM_T \|\theta_{\varepsilon,\fz} w_0\|+\|\theta_{\varepsilon,\fz} w_0\|^2.\end{align}

\noindent\textbf{Step I\!I.} We close the above weighted energy estimate via a standard bootstrap argument. Fix $R>0$ and initial data $w_0\in H_b\cap L^\infty$ satisfying $\|w_0\|_{L^2_b}\le R$. Choose $\varepsilon>0$ sufficiently small so that 
\begin{align}
    \label{DE23}1-\sqrt{\alpha_0}\varepsilon>\frac{1}{2},\qquad \CCC_4\varepsilon<\frac{1}{4},\qquad 2\CCC_4C\varepsilon^{\frac{1}{2}} R<\frac{\nu}{2},
\end{align}
where $C$ denotes the constant in \eqref{DE4}. As a result, one has 
\begin{align}
    \label{DE24}2\CCC_4\varepsilon\|\theta_{\varepsilon,\fz} w_0\|\le 2\CCC_4C'\varepsilon^{\frac{1}{2}}\|w_0\|_{L^2_b}<\frac{\nu}{2}.
\end{align}
Suppose that at time $T$, it follows
\[\MMM_T\le 2\|\theta_{\varepsilon,\fz} w_0\|,\]
which, together with \eqref{DE22}--\eqref{DE24}, leads to
\begin{align}
    \label{DE25}\|\theta_{\varepsilon,\fz}w(T)\|^2+\frac{\nu}{2}\int_0^T\|\theta_{\varepsilon,\fz}\nabla w\|^2\ddd t\le \frac{3}{2}\|\theta_{\varepsilon,\fz} w_0\|^2.
\end{align}
Applying the above estimate to $T_*:=\inf\{t\ge 0| \MMM_T=2\|\theta_{\varepsilon,\fz} w_0\|\}$ gives
\[\MMM_{T_*}^2=\sup_{t\le T_*}\|\theta_{\varepsilon,\fz}w(t)\|^2\le \frac{3}{2}\|\theta_{\varepsilon,\fz} w_0\|^2<4\|\theta_{\varepsilon,\fz} w_0\|^2.\]
This contradiction indicates that $T_*=\infty$ and \eqref{DE25} holds for any $T>0$. In particular, one has 
\begin{align}
    \label{DE26}
    \sup_{t\le T}\|\theta_{\varepsilon,\fz}w(t)\|^2 +\frac{\nu}{2}\int_0^T\|\theta_{\varepsilon,\fz}\nabla w\|^2\ddd t\le 3\|\theta_{\varepsilon,\fz} w_0\|^2,\qquad \forall T\ge 0.
\end{align}

\noindent\textbf{Step I\!I\!I.} In this step, we establish the exponential decay of $\|w(t)\|_{L^2_b}$. First, for any $s\ge0$, by shrinking $\varepsilon>0$ so that 
\begin{align}
    \label{DE27}2\CCC_4\varepsilon\|\theta_{\varepsilon,\fz} w(s)\|\le 2\sqrt{3}\CCC_4C'\varepsilon^{\frac{1}{2}}\|w_0\|_{L^2_b}\le2\sqrt{3}\CCC_4C'\varepsilon^{\frac{1}{2}}R<\frac{\nu}{2},
\end{align}
and repeating the bootstrap argument in Step I\!I, it follows that
\begin{align}
    \label{DE28}
    \sup_{s\le t\le s+T}\|\theta_{\varepsilon,\fz}w(t)\|^2 +\frac{\nu}{2}\int_s^{s+T}\|\theta_{\varepsilon,\fz}\nabla w\|^2\ddd t\le 3\|\theta_{\varepsilon,\fz} w(s)\|^2,\qquad \forall T\ge 0.
\end{align}
In particular, one has 
\[\|\theta_{\varepsilon,\fz}w(s+T)\|^2\le 3\|\theta_{\varepsilon,\fz} w(t)\|^2,\qquad \forall t\in[s,s+T].\]
Integrating the above inequality with respect to $t$ yields
\[T\|\theta_{\varepsilon,\fz}w(s+T)\|^2\le 3\int_s^{s+T}\|\theta_{\varepsilon,\fz} w(t)\|^2\ddd t,\]
which, combined with the Poincar\'e inequality and \eqref{DE28}, leads to
\[T\|\theta_{\varepsilon,\fz}w(s+T)\|^2\le 3\alpha^{-1}_0\int_s^{s+T}\|\theta_{\varepsilon,\fz} \nabla w(t)\|^2\ddd t\le \frac{18}{\nu\alpha_0}\|\theta_{\varepsilon,\fz} w(s)\|^2.\]
By choosing $T_\nu:= \max\{1,36\nu^{-1}\alpha_0^{-1}\}$, it follows
\begin{align*}\|\theta_{\varepsilon,\fz}w(s+T_\nu)\|^2\le \frac{1}{2}\|\theta_{\varepsilon,\fz} w(s)\|^2.
\end{align*}
Iterating the above inequality gives
\[\|\theta_{\varepsilon,\fz}w(nT_\nu)\|^2\le \frac{1}{2^n}\|\theta_{\varepsilon,\fz} w_0\|^2,\qquad \forall n\ge 0.\]
Consequently, for any $t\ge0$, by applying the above inequality and \eqref{DE28}, one gets
\begin{align*}
    \|\theta_{\varepsilon,\fz}w(t)\|^2\le 3\|\theta_{\varepsilon,\fz}w(\lfloor tT_\nu^{-1}\rfloor T_\nu)\|^2\le 3e^{-\lfloor tT_\nu^{-1}\rfloor\log 2}\|\theta_{\varepsilon,\fz} w_0\|^2\!\le 6 e^{-(T_\nu^{-1}\log 2)t}\|\theta_{\varepsilon,\fz} w_0\|^2.
\end{align*}
Since $\varepsilon$ depends only on $R,\nu$, see \eqref{DE23} and \eqref{DE27}, by combining the above estimate and \eqref{DE4}, it follows
\begin{align}
    \label{DE29}
    \|w(t)\|_{L^2_b}\lesssim_{R,\nu}e^{-\alpha_1t},\qquad \forall t\ge 0.
\end{align}

\noindent \textbf{Step I\!V.} Finally, we estimate $\|w(t)\|_\infty$. For any $t\ge 1$, by applying Corollary 6.4 in \cite{AZ14}, there is an increasing function $Q:\R_+\to\R_+$ such that 
\begin{align*}
    \|w(t)\|_{W^{2,2}_b}\le 2Q(\|w(t-2^{-1})\|_{L^2_b}),
\end{align*}
which, combined with \eqref{DE29}, yields
\begin{align}
    \label{DE30}\|w(t)\|_{W^{2,2}_b}\le C_{R,\nu}.
\end{align}
Together with \eqref{DE29} and the Gagliardo--Nirenberg inequality \eqref{DE13-1}, this implies
\[ \|w(t)\|_{\infty}\lesssim_{R,\nu}e^{-\frac{\alpha_1}{2}t}\]
for any $t\ge 1$, and thus completes the proof.
\end{proof}

Now we turn to the proof of Proposition \ref{propDE0}.
\begin{proof}[Proof of Proposition \ref{propDE0}] Notice that estimate \eqref{DE2} is obtained in Proposition \ref{propDE4}. As for estimate \eqref{DE1}, by applying Propositions \ref{propDE3} and \ref{propDE4}, one has
\begin{align*}
    \int_0^\infty \|w(t)\|^2_\infty\ddd t&=\int_0^1\|w(t)\|^2_\infty\ddd t+\int_1^\infty \|w(t)\|^2_\infty\ddd t\\& \le C_{R,\nu}\big(1+\log(1+\|w_0\|_\infty)\big)+C_{R,\nu}.
\end{align*}
This completes the proof.
\end{proof}

\section{Mixing for infinite-energy solutions}\label{MIE}
The goal of this section is to establish Theorem \ref{MT3}. To this end, we introduce the following decomposition 
\[u(t)=w(t)+u_{\rem}(t),\]
where $w(t)$ is an infinite-energy solution of the homogeneous NS equation \eqref{INTRO8} issued from $u_0$ and $u_{\rem}(t)$ solves the white-forced NS-type equation \eqref{SP1} with zero initial data.
\subsection{Probabilistic estimates for the remainder}
In this subsection, we establish probabilistic estimates for the finite-energy remainder $u_{\rem}$ that will be used in the proof of Theorem \ref{MT3}.
\begin{prop}\label{propMIE1}
    Suppose that $u_0\in H_b\cap L^\infty$ and $\FB_0<\infty$. Then, the following statements hold.
    \begin{enumerate}
        \item For any $q\in[1,\infty)$, it follows that
        \begin{align}
            \label{MIE1}
            \sup_{t\ge0}\E\|u_{\rem}(t)\|^{2q}\lesssim_{q,\nu,h,\FB_0}\exp\left(\frac{2q}{\nu}\int_0^\infty\|w(s)\|^2_\infty\ddd s\right).
        \end{align}
        \item For any $q\in[1,\infty)$ and $t,T\ge0$, it follows that 
        \begin{align}
            \label{MIE2}
            \E&\left(\sup_{t\le s\le t+T}\|u_{\rem}(s)\|^2+\nu\int_t^{t+T}\|\nabla u_{\rem}(s)\|^2\ddd s\right)^q\notag\\&\quad\qquad\qquad\qquad\qquad\qquad\lesssim_{q,\nu,h,\FB_0}(1+T)^q\exp\left(\frac{2q}{\nu}\int_0^\infty\|w(s)\|^2_\infty\ddd s\right).
        \end{align}
    \end{enumerate}    
\end{prop}
\begin{proof}\textbf{Step I.} Write
\[\Theta(t):=\exp\left(-\frac{2}{\nu}\int_0^t\|w(s)\|_\infty^2\ddd s\right).\]
By applying the It\^o formula, one has
\begin{align}\label{MIE3}
    \ddd\|u_{\rem}(t)\|^2\Theta(t)&=-\frac{2}{\nu}\Theta(t)\|u_{\rem}(t)\|^2\|w(t)\|^2_\infty\ddd t\notag\\&\qquad+2\Theta(t)\langle u_{\rem},-(u_{\rem}\cdot\nabla)w+\nu\Delta u_{\rem}+h\rangle\ddd t\notag\\&\qquad+2\Theta(t)\ssum_{j\ge 1}b_j\langle u_{\rem},e_j\rangle\ddd\beta_j+\FB_0\Theta(t)\ddd t.
\end{align}
Together with the It\^o formula again, it follows that
\begin{align}\label{MIE4}
    \ddd &\left(\|u_{\rem}(t)\|^{2}\Theta(t)\right)^q=-\frac{2q}{\nu}\Theta^q(t)\|u_{\rem}(t)\|^{2q}\|w(t)\|^2_\infty\ddd t\notag\\&\qquad\qquad\qquad\quad+2q\Theta^q(t)\|u_{\rem}(t)\|^{2(q-1)}\langle u_{\rem},-(u_{\rem}\cdot\nabla)w+\nu\Delta u_{\rem}+h\rangle\ddd t\notag\\&\qquad\qquad\qquad\quad+2q\Theta^q(t)\|u_{\rem}(t)\|^{2(q-1)}\ssum_{j\ge 1}b_j\langle u_{\rem},e_j\rangle\ddd\beta_j\notag\\&\qquad\qquad\qquad\quad+\FB_0q\Theta^q(t)\|u_{\rem}(t)\|^{2(q-1)}\ddd t\notag\\&\qquad\qquad\qquad\quad+ 2q(q-1)\Theta^q(t)\|u_{\rem}(t)\|^{2(q-2)}\ssum_{j\ge 1}b_j^2\langle u_{\rem},e_j\rangle^2\ddd t, 
\end{align}
where, by the cancellation property of the convection term combined with the Young and Poincar\'e inequalities, one has
\begin{align}\label{MIE5}
2q\Theta^q(t)&\|u_{\rem}\|^{2(q-1)}\langle u_{\rem},-(u_{\rem}\cdot\nabla)w+\nu\Delta u_{\rem}+h\rangle\notag\\&=2q\Theta^q(t)\|u_{\rem}\|^{2(q-1)}\left(\langle (u_{\rem}\cdot\nabla)u_{\rem},w\rangle-\nu\|\nabla u_{\rem}\|^2+\langle u_{\rem},h\rangle\right)\notag\\&\le 2q\Theta^q(t)\|u_{\rem}\|^{2(q-1)}\left(\frac{1}{2\nu}\|u_{\rem}\|^2\|w\|^2_\infty-\frac{\nu}{4}\|\nabla u_{\rem}\|^2+\nu^{-1}\alpha_0^{-1}\|h\|^2\right)\notag\\&\le \frac{q}{\nu}\Theta^q(t)\|u_{\rem}\|^{2q}\|w\|^2_\infty-\frac{q\nu}{2}\Theta^q(t)\|u_{\rem}\|^{2(q-1)}\|\nabla u_{\rem}\|^2\notag\\&\quad+\frac{2q}{\nu\alpha_0}\Theta^q(t)\|u_{\rem}\|^{2(q-1)}\|h\|^2\notag\\&\le \frac{q}{\nu}\Theta^q(t)\|u_{\rem}\|^{2q}\|w\|^2_\infty-\frac{q\nu}{4}\Theta^q(t)\|u_{\rem}\|^{2(q-1)}\|\nabla u_{\rem}\|^2\notag\\&\quad+C_{q,\nu}\Theta^q(t)\|h\|^{2q},
\end{align}
and 
\begin{align}\label{MIE6}
     &\FB_0q\Theta^q(t)\|u_{\rem}\|^{2(q-1)}+2q(q-1)\Theta^q(t)\|u_{\rem}  \|^{2(q-2)}\ssum_{j\ge 1}b_j^2\langle u_{\rem},e_j\rangle^2
    \notag\\&\quad\le C_q\FB_0\Theta^q(t)\|u_{\rem} \|^{2(q-1)}\le \frac{q\nu}{8}\Theta^q(t)\|u_{\rem}\|^{2(q-1)}\|\nabla u_{\rem}\|^2+C_{q,\nu}\FB_0^q\Theta^q(t).
\end{align}
By taking expectation on both sides of \eqref{MIE4} and applying \eqref{MIE5}, \eqref{MIE6}, and the Gronwall inequality, one obtains
\begin{align}\label{MIE7}
    \Theta^q(t)\E\|u_{\rem}(t)\|^{2q}\le C_{q,\nu,h,\FB_0},\qquad\forall t\ge 0,
\end{align}
which ensures \eqref{MIE1}.

\noindent\textbf{Step I\!I.} Integrating on both sides of \eqref{MIE3} from $t$ to $t+T$ and applying \eqref{MIE5} with $q=1$, one has
\begin{align}
    \label{MIE8}\notag
    \|u_{\rem}(t+T)\|^2\Theta&(t+T)-\|u_{\rem}(t)\|^2\Theta(t)+\frac{\nu}{2}\int_t^{t+T}\Theta(s)\|\nabla u_{\rem}\|^2\ddd s\\&\le C_{\nu,h,\FB_0}\int_t^{t+T}\Theta(s)\ddd s+2\int_t^{t+T}\Theta(s)\ssum_{j\ge 1}b_j\langle u_{\rem},e_j\rangle\ddd\beta_j,
\end{align}
where, by the Burkholder--Davis--Gundy inequality, it follows that
\begin{align*}
    \E\sup_{t'\in[0,T]}\left|\int_t^{t+t'}\Theta(s)\ssum_{j\ge 1}b_j\langle u_{\rem},e_j\rangle\ddd\beta_j\right|^q\lesssim_q \FB_0^{\frac{q}{2}}\E\left(\int_t^{t+T}\Theta^2(s)\|u_{\rem}\|^2\ddd s\right)^{\frac{q}{2}}.
\end{align*}
Consequently, by combining this with \eqref{MIE7}, \eqref{MIE8}, and the H\"older inequality, one gets
\begin{align*}
    \E\sup_{t\le s\le t+T}\|u_{\rem}(s)&\|^{2q}\Theta^q(s)+\E\left(\nu\int_t^{t+T}\Theta(s)\|\nabla u_{\rem}\|^2\ddd s\right)^q\\&\lesssim_{q,\nu,h,\FB_0} T^q+\Theta^q(t)\E\|u_{\rem}(t)\|^{2q}+\E\left(\int_t^{t+T}\Theta^2(s)\|u_{\rem}\|^2\ddd s\right)^{\frac{q}{2}}\\&\lesssim_{q,\nu,h,\FB_0} T^q+1.
\end{align*}
Since $\Theta(t)\ge \exp\left(-2\nu^{-1}\int_0^\infty\|w(s)\|^2_\infty\ddd s\right)$, this implies \eqref{MIE2}.
\end{proof}

\subsection{The asymptotic behavior of the remainder}
In this subsection, we show that the remainder $u_{\rem}$ can be asymptotically approximated by a family of finite-energy solutions of \eqref{INTRO1}, as $t\to\infty$. To formulate the result, we define $\{u_{\app}(t;T)\}_{t\ge T}$ as the finite-energy solution of \eqref{INTRO1} satisfying the initial condition
\begin{align}
    \label{MIE9}u_{\app}(T;T)= u_{\rem}(T).
\end{align}
\begin{prop}\label{propMIE2}
    Suppose that $u_0\in H_b\cap L^\infty$ and $\FB_0<\infty$. Then, for any $R>0$, there is a constant $\alpha_2>0$ depending on $\nu, R$ such that for any $u_0$ satisfying $\|u_0\|_{L^2_b}\le R$, $q\in[1,\infty)$, and $T\ge 4$,
    \begin{align}
        \label{MIE10}
        \E d_{\loc}(u_{\app}(\FT;T),u_{\rem}(\FT))\lesssim_{q,R,\nu,h,\FB_0}\frac{(1+\|u_0\|_\infty)^{\alpha_2 q}}{(1+\FT)^{\frac{q}{2}}},
    \end{align}
    where $\FT:=T+\sqrt{T}.$
\end{prop}
\begin{proof}
    \textbf{Step I.} Let $\fg(t):=u_{\rem}(t)-u_{\app}(t;T)$ for $t\ge T.$ Then, $\fg$ satisfies
    \begin{align*}
        \begin{cases}
            \partial_t \fg-\nu L\fg+ B(\fg,u_{\rem})+B(u_{\app},\fg)+B(w,u_{\rem})+B(u_{\rem},w)=0,\\
            \fg(T)=0,
        \end{cases}
    \end{align*}
    where the operators $L$ and $B(\cdot,\cdot)$ are defined in \eqref{MFE0}. By testing the above equation against $\fg$ and applying the cancellation property of the convection term and the Ladyzhenskaya inequality, it follows
    \begin{align*}
        \frac{1}{2}\frac{\ddd}{\ddd t}\|\fg(t)\|^2+\nu\|\nabla \fg\|^2&=-\langle B(\fg,u_{\rem})+B(w,u_{\rem})+B(u_{\rem},w),\fg\rangle\\&\lesssim \|\fg\|\|\nabla \fg\|\|\nabla u_{\rem}\|+\|w\|_\infty\|\nabla \fg\|\|u_{\rem}\|,
    \end{align*}
    which, together with Cauchy--Schwarz inequality, implies
    \begin{align*}
        \frac{\ddd}{\ddd t}\|\fg(t)\|^2+\nu\|\nabla \fg\|^2\lesssim_\nu\|\fg\|^2\|\nabla u_{\rem}\|^2+\|w\|_\infty^2\|u_{\rem}\|^2.
    \end{align*}
    An application of the Gronwall inequality yields
    \begin{align}
        \label{MIE11}
        \|\fg(\FT)\|^2\lesssim_\nu\int_T^{\FT}\|w(t)\|_\infty^2\|u_{\rem}(t)\|^2\exp\left(C_\nu\int_t^{\FT}\|\nabla u_{\rem}(s)\|^2\ddd s\right)\ddd t.
    \end{align}

    \noindent\textbf{Step I\!I.} Define 
    \[\II_T:=\int_T^\FT\|\nabla u_{\rem}(t)\|^2\ddd t, \qquad \A_T:=\{\II_T\le C_{\nu}^{-1}\alpha_1T\},\]
    where $\alpha_1$ is the constant given in Proposition \ref{propDE0} and $C_\nu$ is given in \eqref{MIE11}. Notice that on the event $\A_T$, one has 
    \begin{align}\label{MIE12}
        \|\fg(\FT)\|^2\lesssim_{R,\nu} Te^{-\alpha_1 T}\lesssim_{R,\nu}e^{-\frac{\alpha_1}{2} T}.
    \end{align}
    On the other hand, an application of \eqref{MIE2} and \eqref{DE1} implies
    \begin{align*}
        \E\II_T^q&\lesssim_{q,\nu,h,\FB_0} (1+\sqrt{T})^q \exp\left(\frac{2q}{\nu}\int_0^\infty\|w(s)\|^2_\infty \ddd s\right)\\&\lesssim_{q,R,\nu,h,\FB_0} (1+\sqrt{T})^q \left(1+\|u_0\|_\infty\right)^{\alpha_2q},
    \end{align*} 
    where $\alpha_2$ is a constant depending on $R,\nu$. Together with the Chebyshev inequality, this yields
    \begin{align*}
        \PP(\A_T^c)\lesssim_{R,\nu} \frac{\E\II_T^q}{T^q}\lesssim_{q,R,\nu,h,\FB_0} \frac{\left(1+\|u_0\|_\infty\right)^{\alpha_2q}}{(1+T)^{\frac{q}{2}}}.
    \end{align*}
    Consequently, by combining the above estimate with \eqref{MIE12}, it follows 
    \begin{align*}
        \notag\E(1\wedge\|\fg(\FT)\|)&\le\E(\|\fg(\FT)\|\I_{\A_T})+\PP(\A_T^c)\lesssim_{q,R,\nu,h,\FB_0} e^{-\frac{\alpha_1}{4} T}+\frac{\left(1+\|u_0\|_\infty\right)^{\alpha_2q}}{(1+T)^{\frac{q}{2}}}\\&\lesssim_{q,R,\nu,h,\FB_0} \frac{\left(1+\|u_0\|_\infty\right)^{\alpha_2q}}{(1+\FT)^{\frac{q}{2}}}.
    \end{align*}
    Since by the definition \eqref{INTRO5}, 
    \[d_{\loc}(u,v)\le 1\wedge\|u-v\|,\qquad\forall u,v\in H,\]
    this implies \eqref{MIE10}.
\end{proof}
\subsection{Proof of Theorem \ref{MT3}}
To prove Theorem \ref{MT3}, we need the following quantitative smoothing estimate for the infinite-energy solution $\{u(t)\}_{t\ge0}$.
\begin{lemma}\label{lemmaMIE1}
    Suppose that $u_0\in H_b$ and $\FB_1<\infty$. Then, for any $\LLLL,R>0$, there is a constant $\CCC_5$ depending on $\LLLL,R,\nu,h$ such that for any $u_0$ satisfying $\|u_0\|_{L^2_b}\le R$, 
    \begin{align}
        \label{MIE13}\|u(t)\|_{L^2_b}\le \CCC_5,\qquad \|u(t)\|_\infty\le \CCC_5t^{-\frac{1}{2}},\qquad \forall t\in(0,1]
    \end{align}
    on the random event
    \begin{align}\label{MIE13-1}\A_{t,\LLLL}:=\bigg\{\sup_{s\in [0,t]}\|\ZZ_s\|_{H^{\frac{3}{2}}}\le \LLLL\bigg\},\qquad \ZZ_t:=\ssum_{j\ge 1} b_j\int_0^t e^{\nu L(t-s)}e_j\ddd \beta_j.
    \end{align}
    Moreover, it follows that
    \begin{align}
        \label{MIE14}
        \PP(\A_{t,\LLLL}^c)\lesssim \exp\left(-\frac{\alpha_3\LLLL^2}{\FB_1\sqrt{t}}\right),
    \end{align}
    where the constant $\alpha_3$ depends only on $\nu$. 
\end{lemma}
We postpone the proof of this lemma to the end of Appendix \ref{GWP} and turn to the proof of Theorem \ref{MT3}.
\begin{proof}[Proof of Theorem \ref{MT3}] Let $(u^H_t,\PP_u^H)$ and $(u^X_t,\PP_u^X)$ be the Markov families associated with the finite-energy and infinite-energy solutions, respectively. We denote the corresponding transition operators by $\PPPPP^H_t$ and $\PPPPP^X_t$; see \eqref{Markov1} for the definition of $\PPPPP^H_t$ and $\PPPPP^X_t$ is defined in the same way. Let $\PPPPP^{H*}_t$ and $\PPPPP^{X*}_t$ be the dual operators of $\PPPPP^H_t$ and $\PPPPP^X_t$, respectively. The proof is divided into the following three steps.

\noindent\textbf{Step I.} Suppose that 
\begin{align*}u_0\in H_b\cap L^\infty,\qquad \|u_0\|_{L^2_b}\le R,\qquad R>0.
\end{align*}
Fix a function $f\in\lip(X)$ satisfying
\begin{align}
    \label{MIE14-1}\|f\|_{\lip(X)}\le 1.
\end{align}
Since $d_{\loc}(u,u')\le \|u-u'\|$, then $f|_H\in \lip(H)$ and the above condition also holds for $f|_H$. For simplicity, we do not distinguish $f$ and its restriction $f|_H$ on $H$. Let $u_{\app}$ be the finite-energy solution constructed in Proposition \ref{propMIE2} and $\FT:=T+\sqrt{T}.$ By the Markov property, the relation \eqref{MIE9}, Theorem \ref{MT2}, and Propositions \ref{propDE0} and \ref{propMIE1}, it follows
\begin{align*}
    |\E f(u_{\app}(\FT;T))-\langle f,\mu\rangle|&=|\E\PPPPP^H_{\sqrt{T}}f(u_{\rem}(T))-\langle f,\mu\rangle|\lesssim_q \frac{1+\E\|u_{\rem}(T)\|^2}{(1+\sqrt{T})^{q}}\\&\lesssim_{q,\nu,h,\{b_j\}_{j\ge 1}}\frac{\exp\left(\frac{2}{\nu}\int_0^\infty\|w(s)\|^2_\infty\ddd s\right)}{(1+\sqrt{T})^{q}}\\&\lesssim_{q,R,\nu,h,\{b_j\}_{j\ge 1}}\frac{(1+\|u_0\|_\infty)^{\alpha_4}}{(1+T)^{\frac{q}{2}}}
\end{align*}
for any $q\ge 1$ and $T\ge 4$, where $\alpha_4>0$ is a constant depending only on $R,\nu$. Together with Proposition \ref{propMIE2}, the decay estimate \eqref{DE2} in Proposition \ref{propDE0}, and the fact 
\[d_{\loc}(u,u')\le\ssum_{j\ge 1}\frac{1\wedge \sqrt{2j}\|u-u'\|_{L^2_b}}{2^j}\le 2(1\wedge \|u-u'\|_{L^2_b}),\]
this implies 
\begin{align}\label{MIE15}
    |\E f(u(\FT))-\langle f,\mu\rangle|&\lesssim_{q,R,\nu,h,\{b_j\}_{j\ge 1}}e^{-\alpha_1\FT}+\frac{(1+\|u_0\|_\infty)^{\alpha_2 q}}{(1+\FT)^{\frac{q}{2}}}+\frac{(1+\|u_0\|_\infty)^{\alpha_4}}{(1+\FT)^{\frac{q}{2}}}\notag\\&\lesssim_{q,R,\nu,h,\{b_j\}_{j\ge 1}}\frac{(1+\|u_0\|_\infty)^{\alpha_5 q}}{(1+\FT)^{\frac{q}{2}}}
\end{align}
for any $q\ge 1$ and $\FT\ge 6$. Here, the constant $\alpha_5:=\alpha_2\vee\alpha_4$ depends only on $R,\nu$.

\noindent\textbf{Step I\!I.} In this step, we establish \eqref{INTRO15}. Suppose that
\[u_0\in H_b,\qquad \|u_0\|_{L^2_b}\le R,\qquad R>0.\]
Define $t_0:=t^{-c}$ with $c>0$ to be determined in \eqref{MIE16}. Let $\A_{t,\LLLL}$ be the random event introduced in \eqref{MIE13-1}. Since $t_0\le 1$, an application of Lemma \ref{lemmaMIE1} gives
\[\|u(t_0)\|_{L^2_b}\le C_{R,\nu,h},\qquad\|u(t_0)\|_{\infty}\le C_{R,\nu,h}t_0^{-\frac{1}{2}}\]
on the event $\A_{t_0,1}$. Consequently, by applying \eqref{MIE15} to $u(t)$ with $t\ge 8$ combined with the Markov property, one has
\begin{align*}
     \left|\E\left(\I_{\A_{t_0,1}} \big(f(u(t))-\langle f,\mu\rangle\big)\right)\right|&\le\E\left(\I_{\A_{t_0,1}}| \PPPPP^X_{t-t_0}f(u(t_0))-\langle f,\mu\rangle|\right)\\&\lesssim_{q,R,\nu,h,\{b_j\}_{j\ge 1}}\frac{1+\E(\I_{\A_{t_0,1}}\|u(t_0)\|^{\alpha_6 q}_\infty)}{(1+t-t_0)^{\frac{q}{2}}}\\&\lesssim_{q,R,\nu,h,\{b_j\}_{j\ge 1}}\frac{1+t^{\frac{c\alpha_6q}{2}}}{(1+t-t_0)^{\frac{q}{2}}}
\end{align*}
for some constant $\alpha_6$ depending on $R,\nu,h$. Together with choosing
\begin{align}
    \label{MIE16}c=\frac{1}{2\alpha_6}
\end{align}
and using the fact $t-t_0\ge \frac{t}{2}$, this yields
\begin{align*}
    \left|\E\left(\I_{\A_{t_0,1}} \big(f(u(t))-\langle f,\mu\rangle\big)\right)\right|\lesssim_{q,R,\nu,h,\{b_j\}_{j\ge 1}}\frac{1}{(1+t)^{\frac{q}{4}}}.
\end{align*}
Combining this with \eqref{MIE14}, one gets
\begin{align*}
    |\E f(u(t))-\langle f,\mu\rangle|\lesssim_{q,R,\nu,h,\{b_j\}_{j\ge 1}}\frac{1}{(1+t)^{\frac{q}{4}}}
\end{align*}
for any $t\ge 8$. The case $t<8$ is bounded by the trivial bound 
\[|\E f(u(t))-\langle f,\mu\rangle|\le2.\]
Since $q$ is arbitrary, by replacing $q$ by $4q$, the convergence estimate \eqref{INTRO15} holds.

\noindent\textbf{Step I\!I\!I.} It remains to establish \eqref{INTRO16}, which ensures the uniqueness of the invariant measure. Let $\lambda\in \PPPP(X)$. For any function $f\in\lip(X)$ satisfying \eqref{MIE14-1}, by applying \eqref{INTRO15}, it follows 
\begin{align*}
    |\langle f,\DDD(u(t;\lambda))\rangle-\langle f,\mu\rangle|&=|\langle f,\PPPPP^{X*}_t\lambda\rangle-\langle f,\mu\rangle|=|\langle \PPPPP^{X}_tf,\lambda\rangle-\langle f,\mu\rangle|\\&\le \int_{B_{H_b}(R)}|\PPPPP^{X}_tf(u)-\langle f,\mu\rangle|\lambda(\ddd u)+2\lambda(B_{H_b}(R)^c)\\&\le C_{q,R}(1+t)^{-q}+2\lambda(B_{H_b}(R)^c).
\end{align*}
Since $\lim_{R\to\infty}\lambda(B_{H_b}(R)^c)= 0$, by taking supremum in $f$, sending $t\to\infty$, and then letting $R\to\infty$, the convergence \eqref{INTRO16} holds.
\end{proof}

\appendix
\section{A priori energy estimates}
This appendix collects basic energy estimates used throughout the paper.
\begin{prop}\label{propA1}
		Let $\{u(t)\}_{t\ge0}$ be the solution of \eqref{MFE1} issued from $u\in H$. Suppose that $\FB_0:=\ssum_{j\ge 1}b_j^2<\infty$ and $h\in H$. Then, the following estimates hold.
		\begin{enumerate}
		\item \textbf{Energy equality.} For any $t\ge0$, it follows
		\begin{align}
			\label{A1}&\E\|u(t)\|^2+2 \nu\E\int_0^t\|\nabla u(s)\|^2\ddd s=2\E\int_0^t\langle u(s),h\rangle\ddd s+\|u\|^2+\FB_0 t.
            \end{align}
		\item \textbf{Moment estimates.} Let $\kappa_0:=\frac{\nu\alpha_0}{16(1+\FB_0)}$ and 
        \begin{align}\label{A2}
        \EE(t):=\sup_{s\le t}\|u(s)\|^2+\nu\int_0^t\|\nabla u(s)\|^2\ddd s,
		\end{align}
        where $\alpha_0$ is the constant in the Poincar\'e inequality. Then, for any $0<\kappa\le \kappa_0$, $q\ge 1$, and $t\ge0$, it follows 
		\begin{align}
			\label{A3}&\E e^{\kappa\|u(t)\|^2}\le e^{-\kappa t}e^{\kappa\|u\|^2}+C_{\nu,h}+\FB_0,\\\label{A4}&\E\|u(t)\|^{2q}\le e^{-q\nu\alpha_0 t}\|u\|^{2q}+C_{q,\nu,h}(1+\FB_0)^q,\\\label{A5}&\E e^{\kappa\EE(t)}\le 2e^{\kappa(C_{\nu,h}+\FB_0)t}e^{2\kappa\|u\|^2},\\ \label{A6}&\E \EE^{q}(t)\lesssim_{q,\nu,h}  (1+\FB_0)^q(t+\|u\|^2)^q.
		\end{align}

        \item \textbf{Fractional regularity estimate.} For any $T>0$ and $\gamma\in (0,\frac{1}{2})$, it follows 
        \begin{align}
            \label{A7}\E\|u(\cdot)\|_{W^{\gamma,2}(0,T;V^*)}\lesssim_{\gamma,T,\nu,h,\FB_0}1+\|u\|^2.
        \end{align}
        \end{enumerate}
    \end{prop}
    \begin{proof}
        The proof of estimates \eqref{A1} and \eqref{A3}--\eqref{A6} can be found in Section 2.4 in \cite{KS12} and is thus omitted. We turn to the proof of \eqref{A7}. Notice that 
        \begin{align*}
            \E\|u(\cdot)\|_{W^{\gamma,2}(0,T;V^*)}&\le \E\left\|\int_0^\cdot( \nu Lu(t)-B(u(t))+h)\ddd t\right\|_{H^{1}(0,T;V^*)}\\&\quad+\E\left\|\int_0^\cdot \eta\ddd t\right\|_{W^{\gamma,2}(0,T;V^*)}+\sqrt{T}\|u\|,
        \end{align*}
        where, by applying \eqref{A6} with $q=1$, one has
        \begin{align*}
            \E\left\|\int_0^\cdot( \nu Lu-B(u)+h)\ddd t\right\|_{H^{1}(0,T;V^*)}&\lesssim_{T,\nu}\|h\|+\E\|u(\cdot)\|_{L^2(0,T;V)}\\&\quad+\E\left(\|u(\cdot)\|_{L^\infty(0,T;H)}\|u(\cdot)\|_{L^2(0,T;V)}\right)\\&\lesssim_{T,\nu,h,\FB_0}1+\|u\|^2.
        \end{align*}
        As for the moment estimate of the noise $\eta$, we use Lemma 2.1 in \cite{FG95} to obtain
        \begin{align*}
            \E\left\|\int_0^\cdot \eta\ddd t\right\|_{W^{\gamma,2}(0,T;V^*)}\lesssim_\gamma (1+T)\sqrt{\FB_0}.
        \end{align*}
        Combining the above estimates, we obtain \eqref{A7}.
    \end{proof}
    Next, we establish the continuity of the finite-energy solutions with respect to the initial data and noise coefficients.
    \begin{prop}Suppose that $\{u(t)\}_{t\ge 0}$ and $\{u'(t)\}_{t\ge 0}$ are the solutions of \eqref{MFE1} issued from $u,u'$ with the noise coefficients given by
    \[\BB:=\{b_j\}_{j\ge 1},\qquad \BB':=\{b'_j\}_{j\ge 1},\]
    respectively. Then, for any $R,\RR,T>0$, $u,u'\in B_H(R)$, and $\BB,\BB'\in B_{\ell^2}(\RR)$,
    \begin{align}
        \label{A8}\E\int_0^T\|\nabla(u(t)-u'(t))\|^2\ddd t\lesssim_{\RR,R,T,\nu,h} d+d^{c_0},
    \end{align}
    where $d:=\|u-u'\|+\|\BB-\BB'\|_{\ell^2}$, and $c_0$ is a constant depending on $\RR,\nu$.
    \end{prop}
    \begin{proof}Notice that if $d\ge 1$, then \eqref{A8} is a direct consequence of \eqref{A6}. Thus, in what follows, we shall assume $d< 1$. The difference $w(t):=u(t)-u'(t)$ satisfies
    \[\partial_t w-\nu Lw+B(w,u(t))+B(u'(t),w)=\partial_t\ssum_{j\ge 1}(b_j-b_j')e_j\beta_j(t).\]
    Applying the It\^o formula, it follows
    \begin{align*}
        \ddd \|w(t)\|^2+2&\left(\nu\|\nabla w\|^2+\langle B(w,u(t)),w\rangle\right)\ddd t\\&\qquad\qquad\qquad\qquad=2\ssum_{j\ge 1}(b_j-b'_j)\langle e_j, w\rangle\ddd \beta_j+\|\BB-\BB'\|_{\ell^2}^2\ddd t.
    \end{align*}
    Introduce $\FJ(t):=\exp\left(-2c\int_0^t\|\nabla u\|^2\ddd s\right)$ with $c$ to be chosen later. By the It\^o product rule, one has 
    \begin{align*}
        \ddd (\FJ(t)\|w(t)\|^2)+2\FJ&\left(\nu\|\nabla w\|^2+\langle B(w,u(t)),w\rangle+c\|\nabla u\|^2\|w\|^2\right)\ddd t\\&\qquad\qquad\qquad\quad=2\FJ\ssum_{j\ge 1}(b_j-b'_j)\langle e_j, w\rangle\ddd \beta_j+\FJ\|\BB-\BB'\|_{\ell^2}^2\ddd t.
    \end{align*}
    Together with taking expectation on both sides and using the Ladyzhenskaya inequality, one gets
    \begin{align*}
        2\E\int_0^T\FJ(t)&\left(\nu\|\nabla w\|^2+c\|\nabla u\|^2\|w\|^2\right)\ddd t\\&\quad\le \|u-u'\|^2+\|\BB-\BB'\|^2_{\ell^2}T+2\E\int_0^T\FJ(t)|\langle B(w,u(t)),w\rangle|\ddd t\\&\quad\le \|u-u'\|^2+\|\BB-\BB'\|^2_{\ell^2}T+\E\int_0^T\FJ(t)\left(\nu\|\nabla w\|^2+C_\nu\|\nabla u\|^2\|w\|^2\right)\ddd t.
    \end{align*}
    Consequently, by taking $2c=C_\nu$, it follows that
    \begin{align*}
     \E\int_0^T\FJ(t)\|\nabla w\|^2\ddd t\lesssim_{T,\nu}d^2.
    \end{align*}
    Let
    \[\A_{M,T}:=\left\{\int_0^T\|\nabla u(t)\|^2\ddd t\le M\right\},\qquad M>0.\]
    Notice that on the event $\A_{M,T}$, it follows that
    \begin{align}\label{A9}
        \E\left(\I_{\A_{M,T}}\int_0^T\|\nabla w\|^2\ddd t\right)\lesssim_{T,\nu} e^{C_\nu M}d^2.
    \end{align}
    On the other hand, by applying \eqref{A5}, one has 
    \begin{align*}
        \PP(\A_{M,T}^c)\le e^{-\kappa_0\nu M}\E e^{\kappa_0 \EE(T)}\lesssim_{R,\RR,T,\nu,h}e^{-\kappa_0\nu M}.
    \end{align*}
    Together with \eqref{A6}, this implies 
    \begin{align*}
        \E\left(\I_{\A_{M,T}^c}\int_0^T\|\nabla w\|^2\ddd t\right)\le \left(\PP(\A^c_{M,T})\right)^{\frac{1}{2}}\left(\E\left(\int_0^T\|\nabla w\|^2\ddd t\right)^2\right)^{\frac{1}{2}}\lesssim_{\RR,R,T,\nu,h}e^{-\frac{\kappa_0\nu M}{2}}.
    \end{align*}
    Choosing $M:=C_{\nu}^{-1}\log d^{-1}$ and combining the above estimate with \eqref{A9}, one gets 
    \[ \E\int_0^T\|\nabla w\|^2\ddd t\lesssim_{\RR,R,T,\nu,h} d+d^{\frac{\kappa_0\nu C^{-1}_\nu}{2}}\lesssim_{\RR,R,T,\nu,h}d+d^{\frac{\nu^2\alpha_0 C^{-1}_\nu}{32(1+\RR^2)}}.\]
    This completes the proof.
    \end{proof}

    \section{Proof of Proposition \ref{propDE1}}\label{weight}
    In this section, we present the proof of Proposition \ref{propDE1}. To establish the first property, notice that 
    \[0<\theta_{\varepsilon,\fz}(x)\le 1,\qquad \inf_{|x-\fz|\le 1}\theta_{\varepsilon,\fz}(x)\ge \frac{1}{\sqrt{2}}.\]
    Consequently, for any $\fz\in\R$, it follows 
    \begin{align*}
        \|f\|_{L^p(D_{\fz})}\le 2^{\frac{p}{2}}\|\theta_{\varepsilon,\fz}f\|_{p},
    \end{align*}
    which, together with taking supremum in $\fz$ on both sides, gives 
    \begin{align}
        \label{weight1}\|f\|_{L^p_b}\lesssim_p \sup_{\fz\in \R}\|\theta_{\varepsilon,\fz} f\|_p.
    \end{align}
    On the other hand, for any $\fz\in\R$, one has
    \begin{align*}
        \|\theta_{\varepsilon,\fz} f\|^p_p&=\ssum_{l\in\Z}\int_{D_{\fz+l}}\theta^p_{\varepsilon,\fz} |f|^p\ddd x\ddd y\\&\le  \ssum_{l\ge 0} \frac{1}{(1+\varepsilon^2l^2)^{\frac{p}{2}}}\int_{D_{\fz+l}}|f|^p\ddd x\ddd y+\ssum_{l<0} \frac{1}{(1+\varepsilon^2(l+1)^2)^{\frac{p}{2}}}\int_{D_{\fz+l}}|f|^p\ddd x\ddd y\\&\le 2\left( 1+\int_0^\infty \theta^p_{\varepsilon,0}\ddd x\ddd y\right)\|f\|^p_{L^p_b}\lesssim_p \varepsilon^{-1}\|f\|^p_{L^p_b}.
    \end{align*}
    Together with \eqref{weight1}, this ensures \eqref{DE4}.

    The second property follows from direct calculation. Indeed, it follows that
    \begin{align}
        \label{weight2}\partial_x\theta_{\varepsilon,\fz}(x)=-\frac{\varepsilon^2(x-\fz)}{(1+\varepsilon^2(x-\fz)^2)^{\frac{3}{2}}},
    \end{align}
    which ensures \eqref{DE5}. We turn to the third property. Notice that 
    \begin{align}
        \label{weight3}\partial_x \log \theta_{\varepsilon,\fz}=\theta^{-1}_{\varepsilon,\fz}\partial_x\theta_{\varepsilon,\fz}=-\frac{\varepsilon^2(x-\fz)}{1+\varepsilon^2(x-\fz)^2}.
    \end{align}
    This yields 
    \[\|\partial_x \log \theta_{\varepsilon,\fz}\|_\infty\le \varepsilon.\]
    Since 
    \[\partial_x\log\theta_{\varepsilon,\fz}(x)=-\varepsilon\ff(\varepsilon(x-\fz)),\qquad \ff(x)=\frac{x}{1+x^2},\]
    and 
    \[\|\partial_x^j\ff\|_\infty\le C_j,\]
    by taking derivatives on both sides of \eqref{weight3}, one obtains \eqref{DE6}.

    \section{Estimate for the Stokes semigroup in the uniformly local spaces}\label{STEOKS}
    In this section, we establish estimates for the linear Stokes equation 
    \begin{align}
        \label{STOKES0}
        \begin{cases}
            \partial_t u+\nabla p=\nu\Delta u,\\
            \dive u=0,\\
            u|_{t=0}=u_0,\qquad u|_{\partial D}=0
        \end{cases}
    \end{align}
    in the uniformly local spaces $L^p_{b,\dive}$. These estimates are used in Section \ref{DE}.
    \begin{prop}\label{propSTOKES}
        Let $S(t):=e^{\nu t L}$ denote the Stokes semigroup. For any $p\in(1,\infty)$ and $t\in(0,1]$, it follows that 
        \begin{align}
            \label{STOKES1}
            \|S(t)u\|_{L^p_b}+t^{\frac{1}{2}}\|\nabla S(t)u\|_{L^p_b}\lesssim_{p,\nu}\|u\|_{L^p_b}
        \end{align}
        for any $u\in L^p_{b,\dive}$ and
        \begin{align}
            \label{STOKES2}
            t^{\frac{1}{2}}\|S(t)\Pi\dive f\|_{L^p_b}+t\|\nabla S(t)\Pi\dive f\|_{L^p_b}\lesssim_{p,\nu}\|f\|_{L^p_b}
        \end{align}
            for any $f\in L^p_{b}$.
    \end{prop}
    \begin{proof}
        \textbf{Step I.} To begin with, we recall that the Stokes operator 
        \[L: \DDD(L)\to L^p_{\dive},\qquad \DDD(L):=L^p_{\dive}\cap W^{2,p}\cap W^{1,p}_0\]
        generates an analytic semigroup; see \cite{AW05}. We claim that 
         \[\JJ: u\mapsto \JJ u(x,y):=\int_{-1}^y u_1(x,y')\ddd y',\qquad L^p_{\dive }\to\ W^{1,p}_0\]
        is an isomorphism. Indeed, for any $u\in L^p_{\dive}$, by the divergence-free and no-penetration conditions, one has 
        \begin{align}
            \label{STOKES3}\nabla^\perp \JJ u=u,\qquad \nabla^\perp:=(\partial_y,-\partial_x).
        \end{align}
        Consequently, $\JJ u\in W^{1,p}$ and $\JJ$ is bounded from $L^p_{\dive}$ to $W^{1,p}$. On the other hand, notice that 
        \[\partial_x\JJ u(x,1)=\int_{-1}^1 \partial_x u_1(x,y')\ddd y'=-\int_{-1}^1 \partial_y u_2(x,y')\ddd y'=0,\]
        which implies that $\JJ u(x,1)=\int_{-1}^1 u_1(x,y')\ddd y'\equiv C$ for some constant $C$. Together with the fact that
        \[\left\|\int_{-1}^1 u_1(\cdot,y')\ddd y'\right\|_{L^p(\R)}\lesssim \|u\|_p,\]
        this implies that $\JJ u(x,1)\equiv 0$ and thus $\JJ u\in W^{1,p}_0$. Moreover, the relation \eqref{STOKES3} ensures that $\JJ$ is both one-to-one and onto. Therefore, $\JJ$ is an isomorphism as claimed.

        Define the Stokes operator on the space of stream functions
        \[\LLL: \DDD(\LLL)\to W^{1,p}_0,\qquad \psi\mapsto \JJ L \nabla^{\perp}\psi,\]
        where 
        \[\DDD(\LLL):=\{\psi \in W^{3,p}| \psi|_{\partial D}=\partial_y\psi|_{\partial D}=0\}.\]
        Since $\LLL=\JJ L\nabla^\perp=\JJ L \JJ^{-1}$ is similar to $L$, it is also sectorial and thus $\nu\LLL$ generates an analytic semigroup $\SSSS(t):=e^{\nu t\LLL}$. Moreover, notice that $\psi(t):=\SSSS(t)\psi_0$ satisfies
        \begin{align}
            \label{STOKES3-1}\nabla^{\perp}\psi(t)=\nabla^\perp\JJ S(t)\JJ^{-1}\psi_0=S(t)\JJ^{-1}\psi_0,
        \end{align}
        which, by taking $\curl$ on both sides of \eqref{STOKES0}, yields the following equation
        \begin{align}\label{STOKES4}
        \begin{cases}
            \partial_t\Delta \psi=\nu\Delta^2 \psi,\\
            \psi|_{\partial D}=\partial_y\psi|_{\partial D}=0,\\
            \psi|_{t=0}=\psi_0.
        \end{cases}
        \end{align}
        Consequently, one has the following representation formula
        \begin{align}
            \label{STOKES5}\LLL= \Delta_D^{-1}\Delta^2,
        \end{align}
        where $\Delta_D$ denotes the classical Dirichlet--Laplacian and $\Delta^2$ is realized as the bi-Laplacian with the boundary conditions \eqref{STOKES4}$_2$.

        \noindent\textbf{Step I\!I.} In view of \eqref{DE4}, it suffices to establish uniform-in-$\fz$ estimate for $S(t)$ in the weighted Lebesgue space $L^p(D;\theta^p_{\varepsilon,\fz}\ddd x\ddd y)$. To avoid the additional complications arising from the interaction between the divergence-free condition and weighted estimates, we consider the Stokes semigroup $\SSSS(t)$ in the stream-function formulation. 
        
        Fix arbitrary $\fz\in\R$ and $\varepsilon\in (0,1)$. Notice that $\phi_{\varepsilon,\fz}(t):=\theta_{\varepsilon,\fz}\SSSS(t)\psi_0$ solves the equation
        \begin{align}\label{STOKES6}
            \begin{cases}
                \partial_t\phi_{\varepsilon,\fz}=\nu\theta_{\varepsilon,\fz}\LLL\theta_{\varepsilon,\fz}^{-1}\phi_{\varepsilon,\fz}=:\nu\LLL_{\varepsilon,\fz} \phi_{\varepsilon,\fz},\\
                \phi_{\varepsilon,\fz}|_{t=0}=\theta_{\varepsilon,\fz}\psi_0.
            \end{cases}
        \end{align}
        This reduces the weighted estimates of $\SSSS(t)$ to the estimate of $\SSSS_{\varepsilon,\fz}(t):=\exp\left(\nu t\LLL_{\varepsilon,\fz}\right)$ in the usual Sobolev space $W^{1,p}_0$. We claim that there exist constants $C_p,\varepsilon$ independent of $\fz\in\R$ such that 
        \begin{align}
            \label{STOKES7}
            \|(\LLL_{\varepsilon,\fz}-\LLL)\psi\|_{1,p}\le C_p\varepsilon\left(\|\psi\|_{1,p}+\|\LLL \psi\|_{1,p}\right)
        \end{align}
        for any $\psi\in \DDD(\LLL)$. To see this, note that from \eqref{STOKES5}, one has
        \[\LLL_{\varepsilon,\fz}=\left(\theta_{\varepsilon,\fz}\Delta_D^{-1}\theta_{\varepsilon,\fz}^{-1}\right)\left(\theta_{\varepsilon,\fz}\Delta^{2}\theta_{\varepsilon,\fz}^{-1}\right),\]
        which implies 
        \begin{align}
            \label{STOKES8}
            \LLL_{\varepsilon,\fz}-\LLL&=\left(\theta_{\varepsilon,\fz}\Delta_D^{-1}\theta_{\varepsilon,\fz}^{-1}\right)\left[\theta_{\varepsilon,\fz}\Delta^{2}\theta_{\varepsilon,\fz}^{-1}-\left(\theta_{\varepsilon,\fz}\Delta_D\theta_{\varepsilon,\fz}^{-1}\right)\Delta_D^{-1}\Delta^2\right]\notag\\&=\left(\theta_{\varepsilon,\fz}\Delta_D^{-1}\theta_{\varepsilon,\fz}^{-1}\right)\left[\theta_{\varepsilon,\fz}\Delta^{2}\theta_{\varepsilon,\fz}^{-1}-\Delta^2-\left(\theta_{\varepsilon,\fz}\Delta_D\theta_{\varepsilon,\fz}^{-1}-\Delta_D\right)\LLL\right].
        \end{align}
        Consequently, it remains to estimate $\theta_{\varepsilon,\fz}\Delta_D\theta_{\varepsilon,\fz}^{-1}-\Delta_D$ and $\theta_{\varepsilon,\fz}\Delta^{2}\theta_{\varepsilon,\fz}^{-1}-\Delta^2$. Let $\psi\in W_0^{1,p}$. By direct calculation, it follows
        \begin{align*}
            (\theta_{\varepsilon,\fz}\Delta_D\theta_{\varepsilon,\fz}^{-1}-\Delta_D)\psi&=2\theta_{\varepsilon,\fz}\partial_x\theta_{\varepsilon,\fz}^{-1}\partial_x\psi+\theta_{\varepsilon,\fz}\partial_x^2\theta_{\varepsilon,\fz}^{-1}\psi\\&=2\partial_x(\log \theta_{\varepsilon,\fz}^{-1})\partial_x\psi+\partial_x(\theta_{\varepsilon,\fz}\partial_x\theta_{\varepsilon,\fz}^{-1})\psi-\partial_x\theta_{\varepsilon,\fz}\partial_x\theta_{\varepsilon,\fz}^{-1}\psi\\&=2\partial_x(\log \theta_{\varepsilon,\fz}^{-1})\partial_x\psi+\partial^2_x(\log \theta_{\varepsilon,\fz}^{-1})\psi+(\partial_x\theta_{\varepsilon,\fz}\theta_{\varepsilon,\fz}^{-1})^2\psi\\&=(\partial_x\log \theta_{\varepsilon,\fz})^2\psi-2\partial_x(\log \theta_{\varepsilon,\fz})\partial_x\psi-\partial^2_x(\log \theta_{\varepsilon,\fz})\psi.
        \end{align*}
        Together with \eqref{DE6}, this implies that for any $\varphi\in W^{1,p'}_0$, 
        \begin{align*}
            \langle (\theta_{\varepsilon,\fz}\Delta_D\theta_{\varepsilon,\fz}^{-1}-\Delta_D)\psi,\varphi\rangle&=\langle(\partial_x\log \theta_{\varepsilon,\fz})^2\psi-2\partial_x(\log \theta_{\varepsilon,\fz})\partial_x\psi-\partial^2_x(\log \theta_{\varepsilon,\fz})\psi,\varphi\rangle\\&\lesssim\varepsilon\|\psi\|_{1,p}\|\varphi\|_{1,p'},
        \end{align*}
        and thus that
        \begin{align}
            \label{STOKES9}
            \|(\theta_{\varepsilon,\fz}\Delta_D\theta_{\varepsilon,\fz}^{-1}-\Delta_D)\psi\|_{-1,p}\lesssim\varepsilon \|\psi\|_{1,p}.
        \end{align}
        Similarly, for any $\psi\in \DDD(\LLL)$, it follows that 
        \begin{align*}
            \theta_{\varepsilon,\fz}\Delta^{2}\theta_{\varepsilon,\fz}^{-1}\psi&=(\theta_{\varepsilon,\fz}\Delta\theta_{\varepsilon,\fz}^{-1})^2\psi\\&=\left[(\partial_x\log \theta_{\varepsilon,\fz})^2-2\partial_x(\log \theta_{\varepsilon,\fz})\partial_x-\partial^2_x(\log \theta_{\varepsilon,\fz})+\Delta\right]^2\psi,
        \end{align*}
        which, combined with \eqref{DE6}, leads to 
        \begin{align}
            \label{STOKES10}\|(\theta_{\varepsilon,\fz}\Delta^{2}\theta_{\varepsilon,\fz}^{-1}-\Delta^2)\psi\|_{-1,p}\lesssim\varepsilon \|\psi\|_{3,p}.
        \end{align}
        Moreover, by combining \eqref{STOKES9} with the Neumann series
        \begin{align*}\left(\theta_{\varepsilon,\fz}\Delta_D^{-1}\theta_{\varepsilon,\fz}^{-1}\right)&=\Delta_D^{-1}(I+(\theta_{\varepsilon,\fz}\Delta_D\theta_{\varepsilon,\fz}^{-1}-\Delta_D)\Delta_D^{-1})^{-1}\\&=\ssum_{j\ge 0}\Delta_D^{-1}\left(-(\theta_{\varepsilon,\fz}\Delta_D\theta_{\varepsilon,\fz}^{-1}-\Delta_D)\Delta_D^{-1}\right)^j,\end{align*}
        one gets
        \begin{align}
            \label{STOKES11}\left\|\left(\theta_{\varepsilon,\fz}\Delta_D^{-1}\theta_{\varepsilon,\fz}^{-1}\right)\right\|_{\LL(W^{-1,p};W^{1,p}_0)}\lesssim_p\ssum_{j\ge 0} \varepsilon^j\|\Delta_D^{-1}\|_{\LL(W^{-1,p};W^{1,p}_0)}^j<\infty,
        \end{align}
        provided that $\varepsilon\ll1$ is sufficiently small. 

        Now by combining \eqref{STOKES8}--\eqref{STOKES11}, it follows that for any $\psi\in\DDD(\LLL)$, 
        \begin{align*}
            \|(\LLL_{\varepsilon,\fz}-\LLL)\psi\|_{1,p}&=\left\|\left(\theta_{\varepsilon,\fz}\Delta_D^{-1}\theta_{\varepsilon,\fz}^{-1}\right)\left[\theta_{\varepsilon,\fz}\Delta^{2}\theta_{\varepsilon,\fz}^{-1}-\Delta^2-\left(\theta_{\varepsilon,\fz}\Delta_D\theta_{\varepsilon,\fz}^{-1}-\Delta_D\right)\LLL\right]\psi\right\|_{1,p}\\&\lesssim_p\varepsilon\|\psi\|_{3,p}+\varepsilon \|\LLL\psi\|_{1,p}\lesssim_p\varepsilon (\|\psi\|_{1,p}+\|\LLL\psi\|_{1,p})
        \end{align*}
        as claimed. Here in the last inequality, we used the fact that the graph norm of $\LLL$ is equivalent to $W^{3,p}$-norm on $\DDD(\LLL)$.

        \noindent\textbf{Step I\!I\!I.} In this step, we aim to show that $\SSSS_{\varepsilon,\fz}(t)$ is an analytic semigroup. For $\vartheta\in(\frac{\pi}{2},\pi)$, define 
        \[\Sigma_\vartheta:=\{z\in \C|z\neq0,\ |\arg z|<\vartheta\}.\]
        Since $\SSSS(t)$ is an analytic semigroup generated by $\nu\LLL$, there is $z_0\ge0$ and $\vartheta_0\in (\frac{\pi}{2},\pi)$ such that
        \[z_0+\Sigma_{\vartheta_0}\subset \rho(\LLL),
        \]
        and
        \begin{align}\label{STOKES12}
        \|(z-\LLL)^{-1}\|_{\LL(W_0^{1,p})}\lesssim_p \frac{1}{|z-z_0|},\qquad \forall z\in z_0+\Sigma_{\vartheta_0},
        \end{align}
        where $\rho(\LLL)$ is the resolvent set of $\LLL$. Fix arbitrary $z\in z_0+1+\Sigma_{\vartheta_0}$ and $\psi\in W^{1,p}_0$. By combining \eqref{STOKES7}, \eqref{STOKES12}, and the relation
        \[(z-\LLL_{\varepsilon,\fz})^{-1}=(z-\LLL)^{-1}(I+(\LLL-\LLL_{\varepsilon,\fz})(z-\LLL)^{-1})^{-1},\]
        one has
        \begin{align*}
            \|(z-\LLL_{\varepsilon,\fz})^{-1}\psi\|_{1,p}\lesssim_p \frac{1}{|z-z_0|}\|(I+(\LLL-\LLL_{\varepsilon,\fz})(z-\LLL)^{-1})^{-1}\psi\|_{1,p},
        \end{align*}
        where 
        \begin{align*}
            \|(\LLL-\LLL_{\varepsilon,\fz})(z-\LLL)^{-1}\psi\|_{1,p}&\lesssim_p \varepsilon \left(\|(z-\LLL)^{-1}\psi\|_{1,p}+\|\LLL(z-\LLL)^{-1}\psi\|_{1,p}\right)\\&\lesssim_p\varepsilon\left(\frac{1}{|z-z_0|}+1+\frac{|z|}{|z-z_0|}\right)\|\psi\|_{1,p}\\&\lesssim_p\varepsilon\|\psi\|_{1,p}.
        \end{align*}
        Consequently, by shrinking $\varepsilon$ if necessary, it follows that
        \begin{align}
            \label{STOKES13}\|(z-\LLL_{\varepsilon,\fz})^{-1}\|_{\LL(W^{1,p}_0)}\lesssim_p \frac{1}{|z-z_0|}\lesssim_p\frac{1}{|z-z_0-1|}.
        \end{align}
        This implies that $\SSSS_{\varepsilon,\fz}(t)$ is an analytic semigroup in $W_0^{1,p}$.

        \noindent\textbf{Step I\!V.} Since the sector $z_0+1+\Sigma_{\vartheta_0}$ and the constant in the resolvent estimate \eqref{STOKES13} are independent of the parameter $\fz$, by the analyticity of $\SSSS_{\varepsilon,\fz}(t)$, one has
        \begin{align}
            \label{STOKES14}\|\theta_{\varepsilon,\fz}\SSSS(t)\psi_0\|_{1,p}+t\|\theta_{\varepsilon,\fz}\LLL\SSSS(t)\psi_0\|_{1,p}\le C_{p,\nu}\|\theta_{\varepsilon,\fz}\psi_0\|_{1,p},\qquad \forall t\in(0,1],
        \end{align}
        where the constant $C_{p,\nu}$ does not depend on $\fz$; see Theorem 2.2 in \cite{Sch04} and Property $(10.5)$ in \cite{DPS87}. On the other hand, for any $\psi\in\DDD(\LLL)$, by applying the standard regularity theory of the Stokes operator combined with \eqref{STOKES7}, one has
        \begin{align*}
            \|\theta_{\varepsilon,\fz}\psi\|_{3,p}&\lesssim_p\|\theta_{\varepsilon,\fz}\psi\|_{1,p}+\|\LLL(\theta_{\varepsilon,\fz}\psi)\|_{1,p}\\&\lesssim_p\|\theta_{\varepsilon,\fz}\psi\|_{1,p}+\|\LLL_{\varepsilon,\fz}(\theta_{\varepsilon,\fz}\psi)\|_{1,p}+\|(\LLL_{\varepsilon,\fz}-\LLL)(\theta_{\varepsilon,\fz}\psi)\|_{1,p},
            \\&\lesssim_p\|\theta_{\varepsilon,\fz}\psi\|_{1,p}+\|\theta_{\varepsilon,\fz}\LLL\psi\|_{1,p}+\varepsilon(\|\theta_{\varepsilon,\fz}\psi\|_{1,p}+\|\LLL(\theta_{\varepsilon,\fz}\psi)\|_{1,p}),
        \end{align*}
        which, together with shrinking $\varepsilon$ if necessary, yields
        \begin{align*}
            \|\theta_{\varepsilon,\fz}\psi\|_{3,p}\lesssim_p\|\theta_{\varepsilon,\fz}\psi\|_{1,p}+\|\theta_{\varepsilon,\fz}\LLL\psi\|_{1,p}.
        \end{align*}
        As a result, for any $t\in (0,1]$, one has
        \[t\|\theta_{\varepsilon,\fz}\SSSS(t)\psi_0\|_{3,p}\lesssim_p t\|\theta_{\varepsilon,\fz}\SSSS(t)\psi_0\|_{1,p}+t \|\theta_{\varepsilon,\fz}\LLL\SSSS(t)\psi_0\|_{1,p}\le C_{p,\nu}\|\theta_{\varepsilon,\fz}\psi_0\|_{1,p}.\]
        This, combined with \eqref{STOKES14} and the standard Sobolev interpolation inequality, gives
        \begin{align}\label{STOKES14-1}
            t^{\frac{1}{2}}\|\nabla(\theta_{\varepsilon,\fz}\SSSS(t)\psi_0) \|_{1,p}\le C_{p,\nu}\|\theta_{\varepsilon,\fz}\psi_0\|_{1,p}.
        \end{align}
        Since 
        \[\nabla(\theta_{\varepsilon,\fz}\psi)=\theta_{\varepsilon,\fz}\nabla \psi+(\partial_x\theta_{\varepsilon,\fz}\psi,0)=\theta_{\varepsilon,\fz}\nabla \psi+(\partial_x\log \theta_{\varepsilon,\fz}(\theta_{\varepsilon,\fz}\psi),0),\]
        by applying \eqref{STOKES14-1} and \eqref{DE6}, it follows that
        \begin{align}
            \label{STOKES14-2}\|\theta_{\varepsilon,\fz}\SSSS(t)\psi_0\|_{1,p}+t^{\frac{1}{2}}\|\theta_{\varepsilon,\fz}\nabla\SSSS(t)\psi_0\|_{1,p}\le C_{p,\nu}\|\theta_{\varepsilon,\fz}\psi_0\|_{1,p},\qquad \forall t\in(0,1].
        \end{align}
        
        Now, we are in a position to prove \eqref{STOKES1}. Let $u_0\in L^p_{b,\dive}$, $\psi_0=\JJ u_0$, $u(t):=S(t)u_0$, and $\psi(t):=\SSSS(t)\psi_0$. From \eqref{STOKES3-1}, it follows that $\nabla^\perp \psi(t)=u(t)$, which implies
        \begin{align*}
            \theta_{\varepsilon,\fz}u(t)=\theta_{\varepsilon,\fz}\nabla^\perp\psi(t)=\nabla^\perp (\theta_{\varepsilon,\fz}\psi(t))+(0,\partial_x(\log\theta_{\varepsilon,\fz})\theta_{\varepsilon,\fz}\psi(t) ).
        \end{align*}
        Together with \eqref{DE6}, this implies 
        \begin{align}
            \label{STOKES15}
            \|\theta_{\varepsilon,\fz} u(t)\|_p\lesssim \| \theta_{\varepsilon,\fz} \psi(t)\|_{1,p}.
        \end{align}
        Similarly, it follows
        \begin{align}
            \label{STOKES16}\|\theta_{\varepsilon,\fz} \nabla u(t)\|_p\lesssim \| \theta_{\varepsilon,\fz} \nabla \psi(t)\|_{1,p}.
        \end{align}
        Plugging \eqref{STOKES15} and \eqref{STOKES16} into \eqref{STOKES14-2} and applying the Poincar\'e inequality, one gets
        \begin{align}\|\theta_{\varepsilon,\fz}S(t)u_0\|_{p}+t^{\frac{1}{2}}\|\theta_{\varepsilon,\fz}\nabla S(t)u_0\|_{p}\le C_{p,\nu}\|\theta_{\varepsilon,\fz}\psi_0\|_{1,p}\le C_{p,\nu}\|\theta_{\varepsilon,\fz}u_0\|_{p}\label{STOKES16-1}
        \end{align}
        for any $t\in(0,1]$, which, together with \eqref{DE4}, leads to the desired estimate \eqref{STOKES1}.

        \noindent\textbf{Step V.} We turn to the proof of \eqref{STOKES2}. To this end, we note that the stability estimate \eqref{STOKES7} relies only on \eqref{DE6}, which is satisfied for the inverse weight $\theta^{-1}_{\varepsilon,\fz}$. As a result, one has 
        \begin{align}\|\theta^{-1}_{\varepsilon,\fz}S(t)u_0\|_{p}+t^{\frac{1}{2}}\|\theta^{-1}_{\varepsilon,\fz}\nabla S(t)u_0\|_{p}\le C_{p,\nu}\|\theta^{-1}_{\varepsilon,\fz}u_0\|_{p},\qquad \forall t\in(0,1].\label{STOKES17}
        \end{align}
        We also need the following weighted estimate for the Leray projection $\Pi$, whose proof is postponed to the end of this section.
        \begin{lemma}
        \label{lemmaSTOKES1} For any $p\in(1,\infty)$ and integers $k\ge 0$, $l\in\Z$, there exists a sufficiently small $\varepsilon=\varepsilon(l,k,p)>0$ such that 
        \begin{align}
            \label{STOKES18}
            \|\theta^{l}_{\varepsilon,\fz}\Pi f\|_{k,p}\lesssim_{k,p} \|\theta^{l}_{\varepsilon,\fz} f\|_{k,p}
        \end{align}
        for any $\fz\in \R$ and $f$ satisfying $\theta^{l}_{\varepsilon,\fz}f\in W^{k,p}$.
        \end{lemma}
        Let $f,g\in C_0^\infty$. Notice that
        \begin{align*}
            \langle S(t)\Pi \dive f, g\rangle=-\langle f,\nabla S(t)\Pi g\rangle\le \|\theta_{\varepsilon,\fz}f\|_{p}\|\theta_{\varepsilon,\fz}^{-1}\nabla S(t)\Pi g\|_{p'},
        \end{align*}
        where, by applying \eqref{STOKES17} and \eqref{STOKES18} with $p$ replaced by $p'$ and $k=0$, $l=-1$, one has 
        \begin{align*}
            \|\theta_{\varepsilon,\fz}^{-1}\nabla S(t)\Pi g\|_{p'}\lesssim_{p,\nu}t^{-\frac{1}{2}}\|\theta_{\varepsilon,\fz}^{-1}\Pi g\|_{p'}\lesssim_{p,\nu}t^{-\frac{1}{2}}\|\theta_{\varepsilon,\fz}^{-1} g\|_{p'}.
        \end{align*}
        Consequently, it follows that
        \[\langle S(t)\Pi \dive f, g\rangle\lesssim_{p,\nu}t^{-\frac{1}{2}}\|\theta_{\varepsilon,\fz}f\|_{p}\|\theta_{\varepsilon,\fz}^{-1} g\|_{p'},\]
        which, by density, implies 
        \begin{align}
            \label{STOKES20}\sqrt{t}\|\theta_{\varepsilon,\fz} S(t)\Pi \dive f\|_p\lesssim_{p,\nu}\|\theta_{\varepsilon,\fz}f\|_{p},\qquad \forall t\in(0,1].
        \end{align}
        On the other hand, note that an application of \eqref{STOKES16-1} and \eqref{STOKES20} gives
        \begin{align*}
            t\|\theta_{\varepsilon,\fz}\nabla S(t)\Pi\dive f\|_{p}&=\sqrt{2t}\sqrt{\frac{t}{2}}\|\theta_{\varepsilon,\fz}\nabla S(t/2)S(t/2)\Pi\dive f\|_{p}\\&\lesssim_{p,\nu}\sqrt{t}\|\theta_{\varepsilon,\fz}S(t/2)\Pi\dive f\|_p\lesssim_{p,\nu} \|\theta_{\varepsilon,\fz}f\|_{p}.
        \end{align*}
        This, together with \eqref{STOKES20} and \eqref{DE4}, gives the estimate \eqref{STOKES2}.
    \end{proof}
   Notice that the squared inverse weight function $\theta_{\varepsilon,\fz}^{-2}$ also satisfies \eqref{DE6}. As discussed in the derivation of \eqref{STOKES17}, this and Lemma \ref{lemmaSTOKES1} with $k=0$ and $l=-2$ yield the following result.
    \begin{coro}\label{coroSTOKES1}
        Let $S(t):=e^{\nu t L}$ denote the Stokes semigroup. For any $p\in(1,\infty)$, there exists a sufficiently small $\varepsilon=\varepsilon(p)>0$ such that 
        \begin{align}
            \label{STOKES21}
            \|\theta^{-2}_{\varepsilon,\fz}S(t)\Pi u\|_{p}\lesssim_{p,\nu}\|\theta^{-2}_{\varepsilon,\fz}u\|_{p}
        \end{align}
        for any $t\in[0,1]$, $\fz\in \R$, and $u$ satisfying $\theta^{-2}_{\varepsilon,\fz}u\in L^p$.
    \end{coro}
    Finally, we present the proof of Lemma \ref{lemmaSTOKES1}.
    \begin{proof}[Proof of Lemma \ref{lemmaSTOKES1}]
        According to Proposition 2.6 in \cite{AZ14}, it suffices to show that $\theta^{l}_{\varepsilon,\fz}$ with $l\in\Z$ is a weight function of exponential growth rate. Indeed, by applying \eqref{DE6}, one has
        \[\|\partial_x \log \theta^{l}_{\varepsilon,\fz}\|_\infty\lesssim |l|\varepsilon,\]
        which, together with the mean value theorem, shows that 
        \[\log \theta^{l}_{\varepsilon,\fz}(x)-\log \theta^{l}_{\varepsilon,\fz}(x')\lesssim |l|\varepsilon |x-x'|,\]
        and thus that 
        \[\theta^{l}_{\varepsilon,\fz}(x)\le e^{C|l|\varepsilon|x-x'|}\theta^{l}_{\varepsilon,\fz}(x').\]
        Therefore, $\theta^{l}_{\varepsilon,\fz}$ has an exponential growth rate $C|l|\varepsilon$, and estimate \eqref{STOKES18} holds.
    \end{proof}

    \section{Global well-posedness in the uniformly local spaces}\label{GWP}
    The goal of this section is to establish Theorem \ref{MT1}. We begin by proving the unique solvability of the stochastic NS equation \eqref{INTRO1} in the uniformly local spaces.
    \begin{prop}
        Under the assumption of Theorem \ref{MT1}, there is a unique infinite-energy solution of \eqref{INTRO1}, which exists globally in time.
    \end{prop}
    \begin{proof}The proof follows the idea in \cite{AZ14} with slight modifications. Introduce the following decomposition 
        \[u(t)=z(t)+u_{\det}(t),\]
        where $\{z(t)\}_{t\ge 0}$ is the solution of the stochastic Stokes equation
        \begin{align}
            \label{GWP1}
            \begin{cases}
                \partial_t z+\nabla \pp_1-\nu\Delta z=h+\eta,\\
                \dive z=0,\\
                z|_{t=0}=0,\qquad z|_{\partial D}=0,
            \end{cases}
        \end{align}
        and $\{u_{\det}(t)\}_{t\ge0}$ satisfies the following PDE with random parameter
        \begin{align}
            \label{GWP2}
            \begin{cases}
                \partial_t u_{\det}+(u_{\det}\cdot\nabla)u_{\det}+(u_{\det}\cdot\nabla)z+(z\cdot\nabla)u_{\det}+\nabla \pp_2-\nu\Delta u_{\det}=-(z\cdot\nabla )z,\\
                \dive u_{\det}=0,\\
                u_{\det}|_{t=0}=u_0,\qquad u_{\det}|_{\partial D}=0.
            \end{cases}
        \raisetag{30pt}
        \end{align}
        Since $\FB_1<\infty$, one has 
        \begin{align}
            \label{GWP3}z\in C([0,\infty);V)\cap L^2_{\loc}(0,\infty;H^2\cap V)
        \end{align}
        $\PP$-almost surely. Consequently, it remains to prove the unique solvability of \eqref{GWP2} for any given trajectory $z(\cdot)$ satisfying \eqref{GWP3}. For simplicity, in what follows, we shall omit the subscripts in $u_{\det},\pp_2$. This will not lead to any confusion.

        \noindent\textbf{Step I.} In this step, we establish a priori estimates in the uniformly local space $H_b$. Notice that such estimates have been established in the proof of Proposition \ref{propDE4}, provided that $z\equiv 0$; see \eqref{DE26}. To simplify the presentation, we only mention the necessary modifications required in the general case. Let $\theta_{\varepsilon,\fz}$ be the weight function defined in \eqref{DE3}. It suffices to estimate
        \begin{align*}
            \II_1:=-\langle(u\cdot\nabla)z+(z\cdot\nabla)u+(z\cdot\nabla )z,\theta^2_{\varepsilon,\fz}u\rangle,
        \end{align*}
        and 
        \begin{align*}
             \II_2:=\langle(u\cdot\nabla)z+(z\cdot\nabla)u+(z\cdot\nabla )z,&v_{\varepsilon,\fz}\rangle.
        \end{align*}
        Indeed, by the H\"older and Gagliardo--Nirenberg inequalities combined with \eqref{DE5}, it follows that
        \begin{align}
            \label{GWP4}
            |\langle(u\cdot\nabla)z,\theta^2_{\varepsilon,\fz}u\rangle|&\le \|\nabla z\|_4\|\theta_{\varepsilon,\fz}u\|_{\frac{8}{3}}^2\le\|\nabla z\|_4\|\theta_{\varepsilon,\fz}u\|^{\frac{3}{2}}\|\theta_{\varepsilon,\fz}\nabla u\|^{\frac{1}{2}}\notag\\&\le \delta\|\theta_{\varepsilon,\fz}\nabla u\|^2+C_\delta\|\nabla z\|_4^{\frac{4}{3}}\|\theta_{\varepsilon,\fz}u\|^{2}.
        \end{align}
        Moreover, by integrating by parts and using \eqref{DE5} and $\dive z=0$, one has
        \begin{align}
            \label{GWP5}
            |\langle(z\cdot\nabla)u,\theta^2_{\varepsilon,\fz}u\rangle|&=|\langle z\theta_{\varepsilon,\fz}\partial_x\theta_{\varepsilon,\fz},|u|^2\rangle|\le \|z\|_4\|\theta_{\varepsilon,\fz}u\|_{\frac{8}{3}}^2\notag\\&\le \delta\|\theta_{\varepsilon,\fz}\nabla u\|^2+C_\delta\| z\|_4^{\frac{4}{3}}\|\theta_{\varepsilon,\fz}u\|^{2}.
        \end{align}
        Similarly, an application of integrating by parts and the Poincar\'e inequality gives 
        \begin{align}
            \label{GWP6}
            |\langle(z\cdot\nabla)z,\theta^2_{\varepsilon,\fz}u\rangle|&=|\langle z\otimes z, \nabla(\theta^2_{\varepsilon,\fz}u)\rangle|\le C\|z\|_4^2 \|\theta_{\varepsilon,\fz}\nabla u\|\notag\\&\le \delta\|\theta_{\varepsilon,\fz}\nabla u\|^2+C_\delta\|z\|_4^4.
        \end{align}
        By combining \eqref{GWP4}--\eqref{GWP6}, one gets
        \begin{align}
            \label{GWP7}
            \II_1\le3\delta\|\theta_{\varepsilon,\fz}\nabla u\|^2+C_\delta\|\nabla z\|_4^{\frac{4}{3}}\|\theta_{\varepsilon,\fz}u\|^{2}+C_\delta\|z\|_4^4.
        \end{align}
        We turn to the estimate of $\II_2$. By applying \eqref{DE5}, the embedding $H^1\hookrightarrow L^6$, and the Poincar\'e inequality, it follows
        \begin{align}
            \label{GWP8}
            \notag\II_2&\le \|\theta_{\varepsilon,\fz}u\|_6\|\theta_{\varepsilon,\fz}\nabla z\|\|\theta_{\varepsilon,\fz}^{-2}v_{\varepsilon,\fz}\|_3\\&\quad+\|\theta_{\varepsilon,\fz}z\|_6\|\theta_{\varepsilon,\fz}\nabla u\|\|\theta_{\varepsilon,\fz}^{-2}v_{\varepsilon,\fz}\|_3+\|\theta_{\varepsilon,\fz}z\|_6\|\theta_{\varepsilon,\fz}\nabla z\|\|\theta_{\varepsilon,\fz}^{-2}v_{\varepsilon,\fz}\|_3\notag\\&\le \delta\|\theta_{\varepsilon,\fz}\nabla u\|^2+C_\delta\|\nabla z\|^2\|\theta_{\varepsilon,\fz}^{-2}v_{\varepsilon,\fz}\|^2_3+C\|\nabla z\|^2\|\theta_{\varepsilon,\fz}^{-2}v_{\varepsilon,\fz}\|_3,
        \end{align}
        where, by applying \eqref{DE16}, one has
        \begin{align}
            \label{GWP9}\int_0^T\left(\|\nabla z\|^2\|\theta_{\varepsilon,\fz}^{-2}v_{\varepsilon,\fz}\|^2_3+\|\nabla z\|^2\|\theta_{\varepsilon,\fz}^{-2}v_{\varepsilon,\fz}\|_3\right)\ddd t\lesssim_\nu \left(\varepsilon^2\MMM^2_T+\varepsilon\MMM_T\right)\int_0^T\|\nabla z\|^2\ddd t
        \end{align}
        with $\MMM_T:=\|\theta_{\varepsilon,\fz} u\|_{C([0,T];L^2)}$.
        Now, by combining \eqref{GWP7}--\eqref{GWP9} and repeating the derivation \eqref{DE22} with sufficiently small $\varepsilon>0$, one gets
        \begin{align}
            \label{GWP10}
            \|\theta_{\varepsilon,\fz}u(T)\|^2&+(\nu-C_\nu\varepsilon\MMM_T)\int_0^T\|\theta_{\varepsilon,\fz}\nabla u\|^2\ddd t\notag\\&\le C_\nu\varepsilon\MMM_T \|\theta_{\varepsilon,\fz} u_0\|+\|\theta_{\varepsilon,\fz} u_0\|^2+C_\nu \left(\varepsilon^2\MMM^2_T+\varepsilon\MMM_T\right)\int_0^T\|\nabla z\|^2\ddd t\notag\\&\quad+C_\nu\int_0^T\|\nabla z\|_4^{\frac{4}{3}}\|\theta_{\varepsilon,\fz}u\|^{2}\ddd t+C_\nu\int_0^T\|z\|_4^4\ddd t.
        \end{align}
        Notice that \eqref{GWP3} ensures
        \[\int_0^T\left(\|\nabla z\|_4^{\frac{4}{3}}+\|z\|^4_4+\|\nabla z\|^2\right)\ddd t<\infty.\]
        Consequently, by shrinking $\varepsilon>0$ if necessary, repeating the bootstrap argument in Step I\!I in the proof of Proposition \ref{propDE4}, and using \eqref{DE4}, it follows
        \begin{align}
            \label{GWP11}
            \sup_{0\le t\le T}\|u(t)&\|_{L^2_b}^{2}+\nu\sup_{\fz\in\mathbb R}\int_{0}^{T}\|\nabla u(t)\|_{L^{2}(D_{\fz})}^{2}\ddd t\notag\le C_{\nu}\exp\left(C_\nu\int_0^T\|\nabla z\|_4^{\frac{4}{3}}\ddd t\right)\\&\qquad\times\left(1+\|u_{0}\|_{L^2_b}^{2}+\int_{0}^{T}\|z(t)\|_4^{4} dt+\left(\int_{0}^{T}\|\nabla z(t)\|^{2} dt\right)^{2}\right)^{2}.
        \end{align}
        
        \noindent\textbf{Step I\!I.} Let $u(t),u'(t)$ be the solutions of \eqref{GWP2} corresponding to the same trajectory $z(\cdot)$, which start from $u_0$ and $u_0'$, respectively. The goal of this step is to establish an estimate for the difference $\tg(t):=u(t)-u'(t)$, which satisfies 
        \begin{align}
            \label{GWP12}
            \begin{cases}
                \partial_t \tg-\nu\Delta\tg+((u'+z)\cdot\nabla) \tg+(\tg\cdot\nabla)(u+z)+\nabla \tilde\pp=0\\
                \dive \tg=0,\\
                \tg|_{t=0}=u_0-u_0',\qquad \tg|_{\partial D}=0.
            \end{cases}
        \raisetag{30pt}
        \end{align}
        To this end, for $\fz\in\R$, we define a cut-off function
        \begin{align}
            \label{GWP12-1}\psi_\fz:=\psi(x-\fz),\qquad\psi\in C^\infty_0((-1,2)),\qquad  0\le \psi\le 1,\qquad \psi\equiv 1\mbox{ on $[0,1]$}.
        \end{align}
        Similar to Step I in the proof of Proposition \ref{propDE4}, we introduce an auxiliary function $v_{\fz}$ based on the weight $\psi_{\fz}$ through the backward Stokes equation
        \begin{align}
       \label{GWP13}\begin{cases}
        -\partial_t v_{\fz}+\nabla \pp_{\fz}=\nu\Delta v_{\fz},\\
        \dive v_{\fz}=2\psi_{\fz}\partial_x\psi_{\fz}\tg_1,\\
        v_{\fz}|_{t=T}=\left(0,2\psi_{\fz}\partial_x\psi_{\fz}\JJ \tg(T)\right),\qquad v_{\fz}|_{\partial D}=0,
    \end{cases}
    \end{align}
    where $\JJ \tg:=\int_{-1}^y \tg_1(x,y')\ddd y'$. By applying Proposition 2.2 and Theorem 3.2 in \cite{AZ14} with $\phi:=\theta^{-2}_{\varepsilon,\fz}$, the Sobolev embedding $H^{\frac{1}{3}}\hookrightarrow L^p$ with $p=2,3$, and the estimate 
    \begin{align}
        \label{GWP14}\|\theta^{-2}_{\varepsilon,\fz}\psi_{\fz}\partial_x\psi_{\fz}\JJ \tg\|_{H^1}\lesssim \|\theta_{\varepsilon,\fz} \tg\|,
    \end{align}
    one has
    \begin{align}
        \label{GWP15}
        \sup_{t\in[0,T]}\|\theta^{-2}_{\varepsilon,\fz}v_\fz(t)\|+ \sup_{t\in[0,T]}\|\theta^{-2}_{\varepsilon,\fz}v_\fz(t)\|_3\lesssim_\nu \sup_{t\in[0,T]}\|\theta_{\varepsilon,\fz} \tg(t)\|.
    \end{align}
    Testing \eqref{GWP12} against $\psi_\fz^2\tg-v_\fz$ and using \eqref{GWP13}$_1$ yield
    \begin{align*}
    \frac{\ddd}{\ddd t}\left(\frac{1}{2}\|\psi_{\fz}\tg\|^2-\langle \tg,v_{\fz}\rangle\right)+\nu&\langle \nabla \tg,\nabla (\psi_{\fz}^2 \tg)\rangle\\&=-\langle ((u'+z)\cdot\nabla) \tg+(\tg\cdot\nabla)(u+z), \psi_{\fz}^2 \tg-v_{\fz}\rangle.
    \end{align*}
    Comparing with \eqref{DE17}, since the terminal data $v_{\fz}(T)$ is of order $O(1)$, there is a troublesome term $\langle\tg(T),v_{\fz}(T)\rangle$ arising from integrating both sides over $[0,T]$. To overcome this difficulty, we modify the auxiliary function $v_{\fz}$ by extracting a divergence-free part so that the resulting auxiliary function $\hv_{\fz}$ has terminal data of the gradient type. As a result, the terminal pairing $\langle\tg(T),v_{\fz}(T)\rangle$ vanishes.

    To see this, we define $\tv_{\fz}(t):=S(T-t)\Pi v_{\fz}(T)$, where $\Pi$ is the Leray projection and $S(t):=e^{\nu t L}$ is the Stokes semigroup. Let $\hv_\fz:=v_\fz-\tv_\fz$, so $\hv_{\fz}$ satisfies
    \begin{align}
       \label{GWP16}\begin{cases}
        -\partial_t \hv_{\fz}+\nabla \hat\pp_{\fz}=\nu\Delta \hv_{\fz},\\
        \dive \hv_{\fz}=2\psi_{\fz}\partial_x\psi_{\fz}\tg_1,\\
        \hv_{\fz}|_{t=T}=(I-\Pi)\left(0,2\psi_{\fz}\partial_x\psi_{\fz}\JJ \tg(T)\right),\qquad \hat v_{\fz}|_{\partial D}=0.
    \end{cases}
    \end{align}
    By applying Corollary \ref{coroSTOKES1}, the Sobolev embedding $H^1\hookrightarrow L^p$ with $p=2,3$, and estimate \eqref{GWP14}, one has
    \begin{align*}
        \|\theta^{-2}_{\varepsilon,\fz}\tv_\fz(t)\|+\|\theta^{-2}_{\varepsilon,\fz}\tv_\fz(t)\|_3&\lesssim_\nu \|\theta^{-2}_{\varepsilon,\fz}v_{\fz}(T)\|+\|\theta^{-2}_{\varepsilon,\fz}v_{\fz}(T)\|_3\lesssim_\nu \|\theta_{\varepsilon,\fz}\tg (T)\|\\&\lesssim_\nu\sup_{t\in[0,T]}\|\theta_{\varepsilon,\fz} \tg(t)\|.
    \end{align*}
    Consequently, estimate \eqref{GWP15} holds for the new auxiliary function $\hv_{\fz}$, if $\varepsilon$ is sufficiently small. Furthermore, based on this auxiliary function, we have the following energy equality
    \begin{align}\label{GWP17}\notag
    \frac{\ddd}{\ddd t}\left(\frac{1}{2}\|\psi_{\fz}\tg\|^2-\langle \tg,\hv_{\fz}\rangle\right)+\nu&\langle \nabla \tg,\nabla (\psi_{\fz}^2 \tg)\rangle\\&=-\langle ((u'+z)\cdot\nabla) \tg+(\tg\cdot\nabla)(u+z), \psi_{\fz}^2 \tg-\hv_{\fz}\rangle,
    \end{align}
    where the terminal pairing $\langle\tg(T),\hv_{\fz}(T)\rangle$ disappears, since $\tg(T)$ is divergence-free and $\hv_{\fz}(T)$ is of the gradient type as indicated in \eqref{GWP16}$_3$.

    We turn to the estimate of each term in \eqref{GWP17}. Throughout the rest of the proof, we fix a sufficiently small $\varepsilon$ such that \eqref{GWP15} holds for $\hat v_{\fz}$. Since $\psi_{\fz}$ has compact support, the dissipation term satisfies
    \begin{align}
        \label{GWP18}
        \nu\langle \nabla \tg,\nabla (\psi_{\fz}^2 \tg)\rangle\ge \frac{\nu}{2}\|\nabla(\psi_\fz  \tg)\|^2-C_\nu\|\partial_x\psi_{\fz}\tg\|^2\ge  \frac{\nu}{2}\|\nabla(\psi_\fz  \tg)\|^2-C_{\nu}\|\theta_{\varepsilon,\fz}\tg\|^2.
    \end{align}
    Let 
    \begin{align}
        \label{GWP18-1}D^*_{\fz}:=(\fz-1,\fz+2)\times (-1,1).
    \end{align}
    Applying the H\"older and Ladyzhenskaya inequalities combined with the support property of $\psi_{\fz}$, one has
    \begin{align}
    \label{GWP19}
    |\langle(\tg\cdot\nabla)(u+z),\psi_{\fz}^2\tg\rangle|&\le C\|\psi_{\fz}\tg\|\|\nabla(\psi_{\fz}\tg)\|\|\nabla(u+z)\|_{L^2(D^*_{\fz})}\notag\\&\le \frac{\nu}{8}\|\nabla(\psi_{\fz}\tg)\|^2+C_\nu\|\psi_{\fz}\tg\|^2\|\nabla(u+z)\|_{L^2(D^*_{\fz})}^2.
    \end{align}
    Let $\FM_T\ge 1$ be a constant depending on $\nu,T,\|u_0\|_{L^2_b},\|u'_0\|_{L^2_b},\|z\|_{C([0,T];V)},\|z\|_{L^2(0,T;H^2)}$ such that
    \[\sup_{0\le t\le T}\left(\|u(t)\|_{L^2_b}^{2}+\|u'(t)\|_{L^2_b}^{2}\right)+\nu\sup_{\fz\in\mathbb R}\int_{0}^{T}\left(\|\nabla u(t)\|_{L^{2}(D_{\fz})}^{2}+\|\nabla u'(t)\|_{L^{2}(D_{\fz})}^{2}\right)\ddd t\notag\le \FM_T,\]
    and 
    \[\|z\|^2_{C([0,T];V)}+\|z\|^2_{L^2(0,T;H^2)}\le \FM_T.\]
    By applying the above bounds and the Ladyzhenskaya inequality, integrating by parts, and using the fact $\dive (u'+z)=0$, it follows
    \begin{align}
    \label{GWP20}
    \notag|\langle(u'+z)\cdot\nabla&\tg,\psi_{\fz}^2\tg\rangle|=|\langle (u'+z)_1\psi_{\fz}\partial_x\psi_{\fz},|\tg|^2\rangle\le \|u'+z\|_{L^4(D^*_\fz)}\|\psi_{\fz}\tg\|_4\|\partial_x\psi_\fz \tg\|\\&\le C\FM_T^{\frac{1}{4}}\|\nabla (u'+z)\|^{\frac{1}{2}}_{L^2(D^*_\fz)}\|\psi_{\fz}\tg\|^{\frac{1}{2}}\|\nabla(\psi_{\fz}\tg)\|^{\frac{1}{2}}\|\theta_{\varepsilon,\fz}\tg\|\notag\\&\le \frac{\nu}{8}\|\nabla(\psi_{\fz}\tg)\|^2+\FM_T\|\nabla (u'+z)\|^{2}_{L^2(D^*_\fz)}\|\psi_{\fz}\tg\|^{2}+C_{\nu}\|\theta_{\varepsilon,\fz}\tg\|^2.
    \end{align}
    We turn to the terms involving $\hv_{\fz}$. By \eqref{GWP15} with $v_\fz$ replaced by $\hv_\fz$, the Gagliardo--Nirenberg inequality, and estimate \eqref{DE4}, one has
    \begin{align}
        \label{GWP21}|\langle(\tg\cdot\nabla)(&u+z),\hv_\fz\rangle|\le \|\theta^{-2}_{\varepsilon,\fz}\hv_\fz\|_3\|\theta_{\varepsilon,\fz}\tg\|_6\|\theta_{\varepsilon,\fz}\nabla(u+z)\|\notag\\&\le C_\nu \sup_{t\in[0,T]}\|\theta_{\varepsilon,\fz}\tg(t)\|\|\theta_{\varepsilon,\fz}\tg\|^{\frac{1}{3}}\|\theta_{\varepsilon,\fz}\nabla \tg\|^{\frac{2}{3}}\|\theta_{\varepsilon,\fz}\nabla(u+z)\|\notag\\&\le \delta\|\theta_{\varepsilon,\fz}\nabla \tg\|^{2}+\delta\sup_{t\in[0,T]}\|\theta_{\varepsilon,\fz}\tg(t)\|^2\|\theta_{\varepsilon,\fz}\nabla(u+z)\|^2+C_{\nu,\delta}\|\theta_{\varepsilon,\fz}\tg\|^2,
    \end{align}
    and 
    \begin{align}
        \label{GWP22}|\langle(u'+z)\cdot\nabla\tg,\hv_\fz\rangle|&\le \|\theta^{-2}_{\varepsilon,\fz}\hv_\fz\|_3\|\theta_{\varepsilon,\fz}(u'+z)\|^{\frac{1}{3}}\|\theta_{\varepsilon,\fz}\nabla (u'+z)\|^{\frac{2}{3}}\|\theta_{\varepsilon,\fz}\nabla\tg\|\notag\\&\le C_{\nu}\FM_T^{\frac{1}{6}}\sup_{t\in[0,T]}\|\theta_{\varepsilon,\fz}\tg(t)\|\|\theta_{\varepsilon,\fz}\nabla (u'+z)\|^{\frac{2}{3}}\|\theta_{\varepsilon,\fz}\nabla\tg\|\notag\\& \le \delta\|\theta_{\varepsilon,\fz}\nabla \tg\|^{2}+\delta\sup_{t\in[0,T]}\|\theta_{\varepsilon,\fz}\tg(t)\|^2\|\theta_{\varepsilon,\fz}\nabla (u'+z)\|^{2}\notag\\&\quad+ C_{\nu,\delta}\FM_T\sup_{t\in[0,T]}\|\theta_{\varepsilon,\fz}\tg(t)\|^2.
    \end{align}
    As for the initial pairing, by \eqref{GWP15} with $v_\fz$ replaced by $\hv_\fz$, we have 
    \begin{align}
        \label{GWP23}
        |\langle\tg(0),\hv_\fz(0)\rangle|&\le C_\nu \sup_{t\in[0,T]}\|\theta_{\varepsilon,\fz} \tg(t)\|\|\theta_{\varepsilon,\fz}(u_0-u_0')\|\notag\\&\le C_{\nu,\delta}\|\theta_{\varepsilon,\fz}(u_0-u_0')\|^2+\delta\sup_{t\in[0,T]}\|\theta_{\varepsilon,\fz} \tg(t)\|^2.
    \end{align}
    Now, by combining \eqref{GWP17}--\eqref{GWP23} and using the fact that $\psi_\fz\equiv 1$ on $D_\fz$, it follows
    \begin{align*}
        \frac{1}{2}\|\psi_\fz \tg(T)\|^2 &+ \frac{\nu}{4}\int_0^T\|\nabla \tg\|^2_{L^2(D_\fz)}\ddd t\le C_{\nu,\delta}\|\theta_{\varepsilon,\fz}\tg(0)\|^2\notag+2\delta\int_0^T\|\theta_{\varepsilon,\fz}\nabla \tg\|^{2}\ddd t\\&\quad+\left(\delta C_{\nu}\FM_T+C_{\nu,\delta}T\FM_T\right)\sup_{t\in[0,T]}\|\theta_{\varepsilon,\fz} \tg(t)\|^2 \notag\\&\quad+(C_\nu+\FM_T)\int_0^T \|\psi_{\fz}\tg\|^2\left(\|\nabla(u'+z)\|_{L^2(D^*_{\fz})}^2+\|\nabla(u+z)\|_{L^2(D^*_{\fz})}^2\right)\ddd t,
    \end{align*}
    which, together with taking $T\le 1$ and applying the Gronwall inequality, yields
    \begin{align*}
        \|\psi_\fz \tg(T)\|^2 + \nu\int_0^T\|\nabla \tg\|^2_{L^2(D_\fz)}\ddd t&\le C_{\nu,\delta,\FM_1}\|\theta_{\varepsilon,\fz}\tg(0)\|^2\notag+\delta C_{\nu,\FM_1}\int_0^T\|\theta_{\varepsilon,\fz}\nabla \tg\|^{2}\ddd t\\&\quad+\left(\delta C_{\nu,\FM_1}+C_{\nu,\delta,\FM_1}T\right)\sup_{t\in[0,T]}\|\theta_{\varepsilon,\fz} \tg(t)\|^2.
    \end{align*}
    To close the estimate, we use a convolution-type argument. Multiplying on both sides with $\theta^2_{\varepsilon,\tilde\fz}(
    \fz)=(1+\varepsilon^2(\fz-\tilde\fz)^2)^{-1}$, integrating the resulting estimate with respect to $\fz\in \R$, and using Proposition 2.2 in \cite{AZ14} combined with the fact
    \begin{align}
        \label{GWP24}\int_\R\frac{1}{(1+\varepsilon^2(x'-x)^2)(1+\varepsilon^2 x^2)}\ddd x=\frac{2\pi}{\varepsilon(4+\varepsilon^2x'^2)},
    \end{align} 
    one gets
    \begin{align*}
        \|\theta_{\varepsilon,\tilde\fz} \tg(T)\|^2 + \nu\int_0^T\|&\theta_{\varepsilon,\tilde\fz}\nabla \tg\|^2\ddd t\le C_{\nu,\delta,\FM_1}\|\theta_{\varepsilon,\tilde\fz}\tg(0)\|^2\notag+\delta C_{\nu,\FM_1}\int_0^T\|\theta_{\varepsilon,\tilde\fz}\nabla \tg\|^{2}\ddd t\\&\quad+\left(\delta C_{\nu,\FM_1}+C_{\nu,\delta,\FM_1}T\right)\int_\R\theta^2_{\varepsilon,\tilde\fz}(
    \fz)\sup_{t\in[0,T]}\|\theta_{\varepsilon,\fz} \tg(t)\|^2\ddd \fz.
    \end{align*}
    Choose $\delta,T$ sufficiently small such that 
    \[\delta C_{\nu,\FM_1}\le \frac{\nu}{2},\qquad\delta C_{\nu,\FM_1}+C_{\nu,\delta,\FM_1}T\le \frac{\varepsilon}{16\pi},\]
    so one has
    \begin{align*}
        \sup_{t\in[0,T]}\|\theta_{\varepsilon,\tilde\fz} \tg(t)\|^2\le C_{\nu,\delta,\FM_1}\|\theta_{\varepsilon,\tilde\fz}\tg(0)\|^2+\frac{\varepsilon}{16\pi}\int_\R\theta^2_{\varepsilon,\tilde\fz}(
    \fz)\sup_{t\in[0,T]}\|\theta_{\varepsilon,\fz} \tg(t)\|^2\ddd \fz.
    \end{align*}
    Multiplying again on both sides of the above inequality with $\theta^2_{\varepsilon,\hat\fz}(
    \tilde\fz)$, integrating the resulting estimate with respect to $\tilde \fz\in \R$, and using \eqref{GWP24} yield
    \begin{align*}
        \int_\R\theta^2_{\varepsilon,\hat\fz}(
    \fz)\sup_{t\in[0,T]}\|\theta_{\varepsilon,\fz} \tg(t)\|^2\ddd \fz\le C_{\nu,\delta,\FM_1}\|\theta_{\varepsilon,\hat\fz}\tg(0)\|^2+\frac{1}{8}\int_\R\theta^2_{\varepsilon,\hat\fz}(
    \fz)\sup_{t\in[0,T]}\|\theta_{\varepsilon,\fz} \tg(t)\|^2\ddd \fz.
    \end{align*}
    Consequently, since $\hat \fz$ is arbitrary, one has 
    \begin{align}\label{GWP25-0}\sup_{t\in[0,T]}\|\theta_{\varepsilon,\fz} \tg(t)\|^2+\nu\int_0^T\|\theta_{\varepsilon,\fz}\nabla \tg\|^2\ddd t\lesssim_{\nu,\FM_1}\|\theta_{\varepsilon,\fz}\tg(0)\|^2.
    \end{align}
    Since $T$ does not depend on the initial data $\tg(0)$, by iterating the above estimate and using \eqref{DE4}, it follows 
    \begin{align}
        \label{GWP25}\sup_{t\in[0,T]}\|u(t)-u'(t)\|_{L^2_b}^2+\nu\sup_{\fz\in\R}\int_0^T\|\nabla (u-u')\|_{L^2(D_\fz)}^2\ddd t\le C_T\|u_0-u'_0\|_{L^2_b}^2,
    \end{align}
    where $C_T$ depends on $\nu,T,u_0,u_0'$ and the trajectory $z(\cdot)$.

    \noindent\textbf{Step I\!I\!I.} The uniqueness now follows from \eqref{GWP25}. The existence part requires additional treatment, since in general one can only approximate $u\in H_b$ by elements in $H$ under the weak metric $d_{\loc}$ defined in \eqref{INTRO5}. 
    
    Let $u_0,u_0'\in B_{H_b}(R)$. Notice that for any integer $k\ge 1$ and $\fz\in \R$, by applying \eqref{GWP25-0}, one has
    \begin{align*}
    \sup_{t\in[0,T]}\|u(t)-u'(t)\|_{L^2(D_\fz)}^2&\le C_T\|u_0-u_0'\|^2_{L^2(I_{k^2,\fz})}+C_T\int_{I_{k^2,\fz}^c}\frac{|u_0(x)-u'_0(x)|^2}{1+\varepsilon^2(x-\fz)^2}\ddd x\\&\lesssim C_T\|u_0-u_0'\|^2_{L^2(I_{k^2,\fz})} +C_T\|u_0-u_0'\|_{L^2_b}^2 \varepsilon^{-2} k^{-2},
    \end{align*}
    where $I_{k^2,\fz}:=(\fz-k^2,\fz+k^2)\times (-1,1)$. This implies 
    \begin{align*}
    \|u(t)-u'(t)\|_{L^2(I_k)}\lesssim C_T\sqrt{k}\|u_0-u_0'\|_{L^2(I_{2k^2})}+C_TR\varepsilon^{-1}k^{-\frac{1}{2}}
    \end{align*}
    with $I_k:=(-k,k)\times (-1,1)$. Consequently, for any integer $k\ge 1$ and $t\in[0,T]$, one has
    \begin{align}\label{GWP26}
        &d_{\loc}(u(t),u'(t))\le \left(\ssum_{l\le k}2^{-l}\right)\|u(t)-u'(t)\|_{L^2(I_{k})}+\ssum_{l> k}\frac{1\wedge\|u(t)-u'(t)\|_{L^2(I_l)}}{2^l}\notag\\&\lesssim C_{T}(\sqrt{k}\|u_0-u_0'\|_{L^2(I_{2k^2})}+R\varepsilon^{-1}k^{-\frac{1}{2}})+C_T\ssum_{l>k}\frac{\sqrt{l}(1\wedge\|u_0-u_0'\|_{L^2(I_{2l^2})})}{2^l}.
    \end{align}
    Now for any given $u_0\in H_b$, choose $\{u_{0,n}\}_{n\ge 1}\subset H$ such that
    \[\sup_{n\ge 1}\|u_{0,n}\|_{L^2_b}<\infty,\qquad \lim_{n\to\infty}d_{\loc}(u_0,u_{0,n})=0.\]
    Denote by $u_n(\cdot)$ the finite-energy solution issued from $u_{0,n}$. By choosing $k$ sufficiently large in \eqref{GWP26} and then sending $n,m\to\infty$, one obtains
    \[\lim_{n,m\to\infty}\sup_{t\in[0,T]} d_{\loc}(u_n(t),u_m(t))=0.\]
    This shows that $\{u_n(\cdot)\}_{n\ge 1}$ is a Cauchy sequence in $C([0,T];L^2_{\loc})$, and thus converges to some function $u$. Moreover, the a priori estimate \eqref{GWP11} shows that
    \[\sup_{n\ge 1}\sup_{0\le t\le T}\|u_n(t)\|_{L^2_b}^{2}\le C_T\]
    for some positive constant $C_T$. Since every closed $H_b$-ball is complete under $d_{\loc}$, we conclude that $u\in C([0,T];X)$.
    
    Finally, notice that the convergence under the metric of $C([0,T];L^2_{\loc})$ is sufficient to pass to the limit for the nonlinear term. This, together with the weak convergence in $L^2(0,T;L^2_{\loc})$ for the gradient term, ensures the existence of an infinite-energy solution of \eqref{GWP2}. 
    \end{proof}
    Next, we study the smoothing effect of the stochastic NS equation \eqref{INTRO1} in the infinite-energy setting. 
    \begin{prop}
        Let $\{u(t)\}_{\ge0}$ be the unique global infinite-energy solution of \eqref{INTRO1}. Then, for any $\delta,T>0$, $\sup_{t\in[\delta,T]}\|u(t)\|_{\infty}<\infty$ holds $\PP$-almost surely.
    \end{prop}
    \begin{proof}The proof is based on Lemma \ref{lemmaMIE1}, whose proof is postponed to the end of this section.
    
    \noindent\textbf{Step I.} To begin with, we show that
    \begin{align}\label{GWP26-1}
    \PP^X_u\Big\{\sup_{t\in[\delta,1]}\|u_t\|_\infty<\infty\Big\}=1
    \end{align}
     for every $u\in H_b$ and $\delta\in(0,1)$. Let $R:=\|u\|_{L^2_b}$. By applying \eqref{MIE13}, it follows that 
     \[\sup_{t\in[\delta,1]}\|u(t)\|_\infty\le \CCC_5 \delta^{-\frac{1}{2}}\]
     on the random event $\A_{1,\LLLL}$, which implies 
     \[\PP_u^X\Big\{\sup_{t\in[\delta,1]}\|u_t\|_\infty=\infty\Big\}\le \PP_u^X\Big\{\sup_{t\in[\delta,1]}\|u_t\|_\infty>\CCC_5\delta^{-\frac{1}{2}}\Big\}\le \PP(\A^c_{1,\LLLL})\to 0,\]
     as $\LLLL\to\infty$. Therefore, \eqref{GWP26-1} holds.  

     \noindent\textbf{Step I\!I.} Without loss of generality, we assume $0<\delta<\frac{1}{2}$. Introduce the following covering
    \[[\delta,T]\subset\cup_{j=0}^N[j/2+\delta,j/2+1].\]
    By the Markov property, it follows
    \[\PP_u^X\Big\{\sup_{t\in[\delta,1]}\|u_{j/2+t}\|_\infty<\infty\Big|\FF_{j/2}\Big\}=\PP^X_{u_{j/2}}\Big\{\sup_{t\in[\delta,1]}\|u(t)\|_\infty<\infty\Big\}=1.\]
    Consequently, by taking the intersection of these events, we obtain
    \[\sup_{t\in[\delta,T]}\|u(t)\|_\infty<\infty\]
    almost surely.
    \end{proof}

    We turn to the verification of the decomposition stated in Theorem \ref{MT1}.
    \begin{prop}
        Let $\{u(t)\}_{\ge0}$ be the unique global infinite-energy solution of \eqref{INTRO1} issued from $u_0\in H_b\cap L^\infty$. Then, 
        \[u(t)=w(t)+u_{\rem}(t),\]
        where $w(t)$ is an infinite-energy solution of the homogeneous NS equation \eqref{INTRO8} and $u_{\rem}$ satisfies the path regularity \eqref{INTRO8-0}.
    \end{prop}
    \begin{proof}
        It suffices to prove \eqref{INTRO8-0}. To this end, we test the remainder equation \eqref{SP1} against $u_{\rem}$ and apply the It\^o formula to derive 
        \begin{align}\label{GWP27}
            \notag\ddd \|u_{\rem}(t)\|^2+2\nu\|\nabla u_{\rem}\|^2\ddd t&=2\langle h-(u_{\rem}\cdot\nabla)w,u_{\rem}\rangle\ddd t\\&\quad+2\ssum_{j\ge 1}b_j\langle u_{\rem},e_j\rangle\ddd \beta_j+\FB_0\ddd t,
        \end{align}
        where, by the cancellation property of the convection term, one has
        \begin{align*}
            |\langle(u_{\rem}\cdot\nabla)w,u_{\rem}\rangle|\le \|u_{\rem}\|\|\nabla u_{\rem}\|\|w\|_\infty.
        \end{align*}
        As a result, by taking expectation on both sides of \eqref{GWP27}, it follows
        \begin{align}
            \label{GWP28}
            \frac{\ddd}{\ddd t}\E \|u_{\rem}(t)\|^2+\nu\E\|\nabla u_{\rem}\|^2\le \FB_0+C_\nu\|h\|^2+C_\nu\|w\|^2_\infty\E\|u_{\rem}\|^2
        \end{align}
        On the other hand, by Proposition \ref{propDE0}, one has
        \[\int_0^\infty\|w(t)\|^2_\infty\ddd t\le C_{\nu,u_0}.\]
        This, together with \eqref{GWP28} combined with the Gronwall inequality, yields
        \begin{align*}
            \E\|u_{\rem}(T)\|^2\le T(\FB_0+C_\nu\|h\|^2)\exp\left(C_\nu\int_0^\infty\|w(t)\|^2_\infty\ddd t\right)\le C_{\nu,u_0}T(\FB_0+\|h\|^2).
        \end{align*}
        Plugging this into \eqref{GWP28} and integrating the resulting estimate over $[0,T]$ imply
        \begin{align}
            \label{GWP29}
            \E\|u_{\rem}(T)\|^2+\nu\E\int_0^T\|\nabla u_{\rem}\|^2\ddd t\le  C_{\nu,u_0}T(\FB_0+\|h\|^2).
        \end{align}
        Finally, by the Burkholder--Davis--Gundy and Poincar\'e inequalities combined with \eqref{GWP29}, it follows
        \begin{align*}
            \E\sup_{t\in[0,T]}&\left|\ssum_{j\ge 1}b_j\int_0^t\langle u_{\rem},e_j\rangle\ddd\beta_j\right|\le C \E\left(\ssum_{j\ge 1}b_j^2\int_0^T\langle u_{\rem},e_j\rangle^2\ddd t \right)^{\frac{1}{2}}\\&\qquad\qquad\le C\FB_0^{\frac{1}{2}}\left(\E\int_0^T\|\nabla u_{\rem}\|^2\ddd t \right)^{\frac{1}{2}}\le C_{\nu,u_0}\FB_0^{\frac{1}{2}} T^{\frac{1}{2}}(\FB_0+\|h\|^2)^{\frac{1}{2}}.
        \end{align*}
        Together with integrating both sides in $t$ of \eqref{GWP27} and taking first supremum in $t\in[0,T]$ and then expectation, this implies 
        \begin{align*}
            \E\sup_{t\in[0,T]}\|u_{\rem}(t)\|^2\le C_{\nu,h,\FB_0,T,u_0}<\infty.
        \end{align*}
        Combining this and \eqref{GWP29}, we obtain the path regularity \eqref{INTRO8-0}.
    \end{proof}
    Now, we prove the last statement in Theorem \ref{MT1}. 
    \begin{prop}
        The infinite-energy solutions of \eqref{INTRO1} define a Markov family $(u_t,\PP^X_u)$.
    \end{prop}
    \begin{proof} As presented in Example 1.3.6 in \cite{KS12}, it suffices to show that \eqref{INTRO1} admits a solution map, which is Borel measurable in the initial data and noise trajectory. Indeed, since $\FB_1<\infty$, there is a measurable map
    \[\bfPsi_1: C([0,\infty);V)\to  C([0,\infty);V)\cap L_{\loc}^2(0,\infty;V\cap H^2)\]
    such that 
    \[\bfPsi_1 W=z,\qquad\PP\mbox{-a.s.,}\]
    where $W(t):=\int_0^t\eta(s)\ddd s$ is the Wiener process and $\{z(t)\}_{t\ge0 }$ is the solution of the stochastic Stokes equation \eqref{GWP1}. Introduce
    \[\bfPsi_2: (u,z)\mapsto u_{\det}(\cdot),\qquad X\times \left(C([0,\infty);V)\cap L^2_{\loc}(0,\infty;V\cap H^2)\right)\to C([0,\infty);X)\]
    as the solution map of \eqref{GWP2}. Define
    \[\bfPhi: (u,W)\mapsto \bfPsi_2(u,\bfPsi_1 W)+\bfPsi_1 W,\qquad X\times C([0,\infty);V)\to C([0,\infty);X).\]
    By the unique solvability of the stochastic NS equation \eqref{INTRO1} in the infinite-energy setting, one has 
    \[\bfPhi(u,W)=\bfPsi_2(u,\bfPsi_1 W)+\bfPsi_1 W=u(\cdot),\qquad\PP\mbox{-a.s.,}\]
    where $\{u(t)\}_{t\ge 0}$ is the infinite-energy solution of \eqref{INTRO1}. Moreover, the derivation of \eqref{GWP25} and \eqref{GWP26} can be extended to two solutions with different noise trajectories, which ensures the Borel measurability of $\bfPsi_2$ and thus the Borel measurability of the solution map $\bfPhi$.
    \end{proof}
    Finally, we present the proof of Lemma \ref{lemmaMIE1}.
    \begin{proof}[Proof of Lemma \ref{lemmaMIE1}]Fix arbitrary $t_0\in(0,1]$. Introduce the decomposition 
    \[u(t)=u_{\det}(t)+z(t),\]
    where $\{z(t)\}_{t\ge0}$ satisfies the stochastic Stokes equation \eqref{GWP1} and $\{u_{\det}(t)\}_{t\ge0}$ is the solution of \eqref{GWP2}. Throughout the proof, the analysis is carried out on the random event $\A_{t_0,\LLLL}$. Notice that the Stokes semigroup $S(t):=e^{t\nu L}$ satisfies
    \[\left\|\int_0^\cdot S(\cdot-s)h \ddd s\right\|_{L^\infty\left(0,t;H^{3/2}\right)}\lesssim_\nu t^{\frac{1}{4}}\|h\|,\]
    which, together with the Sobolev embeddings $H^{\frac{3}{2}}\hookrightarrow L^\infty$ and $H^{\frac{3}{2}}\hookrightarrow W^{1,4}$, yield
    \begin{align}\label{SE1}
    \sup_{s\in[0,t_0]}\left(\|z(s)\|_\infty+\|z(s)\|_{H^1}+\|z(s)\|_{W^{1,4}}\right)\le C_{\nu,h,\LLLL}.
    \end{align}
    Therefore, it remains to bound $u_{\det}$. As before, to simplify the notation, we shall drop the subscript in $u_{\det}$ and omit the dependence of the constants on $h,\LLLL,R$ in what follows.

    \noindent\textbf{Step I.} Let $\theta_{\varepsilon,\fz}, \psi_{\fz}$ be the weight functions introduced in \eqref{DE3} and \eqref{GWP12-1}, respectively. Denote by $D^*_\fz$ the rectangle introduced in \eqref{GWP18-1}. By applying \eqref{GWP11} combined with \eqref{SE1}, one has
    \begin{align}
        \label{SE2}
        \sup_{0\le t\le t_0}\|u(t)&\|_{L^2_b}^{2}+\nu\sup_{\fz\in\mathbb R}\int_{0}^{t_0}\|\nabla u(t)\|_{L^{2}(D_{\fz})}^{2}\ddd t\le C_\nu.
    \end{align}
    Testing the equation \eqref{GWP2} against $-\dive (\psi_{\fz}^2\nabla u)$ gives
    \begin{align}\label{SE3}
    \frac{1}{2}\frac{\ddd }{\ddd t}\|\psi_{\fz}\nabla u\|^2+\nu\langle\Pi\Delta u,\dive (\psi_{\fz}^2\nabla u)\rangle= \langle \Pi((u+z)\cdot\nabla)(u+z),\dive (\psi_{\fz}^2\nabla u)\rangle.
    \end{align}
    Each term in the above energy equality is estimated in the following way. We begin with the dissipation term. Since
    \[\dive (\psi_{\fz}^2\nabla u)=\psi_{\fz}^2\Delta u+2\psi_{\fz}\psi_{\fz}'\partial_xu,\]
    it follows
    \begin{align}\label{SE4}
        \langle\Pi\Delta u,\dive (\psi_{\fz}^2\nabla u)\rangle&\notag=\langle\Pi\Delta u,\psi_{\fz}^2\Delta u\rangle+\langle\Pi\Delta u,2\psi_{\fz}\psi_{\fz}'\partial_xu\rangle\\&=\|\psi_{\fz}\Pi\Delta u\|^2+\langle \Pi\Delta u,[\Pi,\psi_{\fz}^2]\Delta u\rangle +\langle\Pi\Delta u,2\psi_{\fz}\psi_{\fz}'\partial_xu\rangle,
    \end{align}
    where $[\Pi,\psi_{\fz}^2]:=\Pi\psi_{\fz}^2-\psi_{\fz}^2\Pi$. The second term on the right-hand side of \eqref{SE4} is estimated by using Lemma 6.3\footnote{This lemma does not rely on the smallness assumption of $(\psi^2)''$, which is not assumed in \eqref{GWP12-1}.} in \cite{AZ14}, Lemma \ref{lemmaSTOKES1}, Corollary \ref{coroSTOKES1}, the trace and interpolation inequalities combined with integrating by parts, which gives
    \begin{align*}|\langle \Pi\Delta u,[\Pi,\psi_{\fz}^2]\Delta u\rangle|&\le |\langle \nabla u,\nabla\Pi[\Pi,\psi_{\fz}^2]\Delta u\rangle|+|\langle \partial_nu,\Pi[\Pi,\psi_{\fz}^2]\Delta u\rangle _{\partial D}|\notag\\& \le C\|\theta_{\varepsilon,\fz}\nabla u\|\|\theta^{-2}_{\varepsilon,\fz}[\Pi,\psi_{\fz}^2]\Delta u\|_{H^1}+\|\theta_{\varepsilon,\fz}\partial_y u\|_{H^\frac{3}{4}}\|\theta^{-2}_{\varepsilon,\fz}[\Pi,\psi_{\fz}^2]\Delta u\|_{H^1}\notag \\&\le C\|\theta_{\varepsilon,\fz}u\|_{H^2}^{\frac{3}{2}}\|\theta_{\varepsilon,\fz}u\|^{\frac{1}{2}}+\|\theta_{\varepsilon,\fz}u\|_{H^2}^{\frac{15}{8}}\|\theta_{\varepsilon,\fz}u\|^{\frac{1}{8}}\notag\\&\le \delta\nu^{-1}\|\theta_{\varepsilon,\fz}u\|_{H^2}^{2}+C_{\delta,\nu}\|\theta_{\varepsilon,\fz}u\|^{2}.
    \end{align*}
    Together with \eqref{SE4}, Lemma 6.3 in \cite{AZ14}, and
    \[2|\langle\Pi\Delta u,\psi_{\fz}\psi_{\fz}'\partial_xu\rangle|\le\frac{1}{4}\|\psi_{\fz}\Pi\Delta u\|^2+C\|\nabla u\|_{L^2(D^*_\fz)}^2,\]
    this implies 
    \begin{align}\label{SE5}
        \nu\langle\Pi\Delta u,\dive (\psi_{\fz}^2\nabla u)\rangle&\ge \frac{3\nu}{4}\|\psi_{\fz}\Pi\Delta u\|^2-\delta\|\theta_{\varepsilon,\fz}u\|_{H^2}^{2}-C_{\delta,\nu}\|\theta_{\varepsilon,\fz}u\|^{2}-C_\nu\|\nabla u\|_{L^2(D^*_\fz)}^2\notag\\&\ge \frac{3\nu}{8}\left(\|\psi_{\fz}\Pi\Delta u\|^2+c\|\psi_\fz u\|^2_{H^2}\right)\notag\\&\qquad\qquad-\delta\|\theta_{\varepsilon,\fz}u\|_{H^2}^{2}-C_{\delta,\nu}\|\theta_{\varepsilon,\fz}u\|^{2}-C_\nu\|\nabla u\|_{L^2(D^*_\fz)}^2.
    \end{align}
    We turn to the nonlinear terms. To this end, notice that
    \begin{align*}
    \langle\Pi f,\dive (\psi_{\fz}^2\nabla u)\rangle&=\langle\psi_{\fz}f,\psi_{\fz}\Pi\Delta u\rangle+\langle f,[\Pi,\psi_{\fz}^2]\Delta u\rangle+\langle f,\Pi(2\psi_{\fz}\psi_{\fz}'\partial_xu)\rangle,
    \end{align*}
    where, by Lemma 6.3 in \cite{AZ14} and Lemma \ref{lemmaSTOKES1}, one has
    \begin{align}
    \|\theta_{\varepsilon,\fz}^{-2}\nabla[\Pi,\psi_{\fz}^2]\Delta u\|+\|\theta_{\varepsilon,\fz}^{-2}\nabla\Pi(2\psi_{\fz}\psi_{\fz}'\partial_xu)\|\lesssim\|\theta_{\varepsilon,\fz}u\|_{H^2}.\label{SE6}
    \end{align}
    We begin with 
    \begin{align}
        \label{SE7}
        \II_3:&=\langle\Pi (u\cdot \nabla) u,\dive (\psi_{\fz}^2\nabla u)\rangle\notag\\&=\langle (u\cdot\nabla)(\psi_{\fz}u),\psi_{\fz}\Pi\Delta u\rangle-\langle u_1\psi_{\fz}'u,\psi_{\fz}\Pi\Delta u\rangle\notag\\&\qquad\qquad+\langle(u\cdot\nabla)u,[\Pi,\psi_{\fz}^2]\Delta u\rangle +\langle (u\cdot\nabla)u,\Pi(2\psi_{\fz}\psi_{\fz}'\partial_xu)\rangle\notag\\&=:\II_{3,1}+\II_{3,2}+\II_{3,3}+\II_{3,4}.
    \end{align}
    For $\II_{3,1}$ and $\II_{3,2}$, by the Ladyzhenskaya and Young inequalities combined with the support property of $\psi_\fz$, it follows
    \begin{align*}
    |\II_{3,1}|&\le C\|u\|_{L^4(D_{\fz}^*)}\|\nabla(\psi_{\fz}u)\|_4\|\psi_{\fz}\Pi\Delta u\|\\&\le C\|u\|_{L^4(D_{\fz}^*)}\|\nabla(\psi_{\fz}u)\|^{\frac{1}{2}}\|\psi_{\fz}u\|_{H^2}^{\frac{1}{2}}\|\psi_{\fz}\Pi\Delta u\|\\&\le \fa\left(\|\psi_{\fz}\Pi\Delta u\|^2+\|\psi_{\fz}u\|_{H^2}^2\right)+C_\fa\|u\|_{L^4(D_{\fz}^*)}^4\|\nabla(\psi_{\fz}u)\|^2\\&\le\fa\left(\|\psi_{\fz}\Pi\Delta u\|^2+\|\psi_{\fz}u\|_{H^2}^2\right)+C_\fa\|u\|_{L^2(D_{\fz}^*)}^2\|u\|_{H^1(D_{\fz}^*)}^2\left(\|\psi_{\fz}\nabla u\|^2+\|u\|_{L^2(D_{\fz}^*)}^2\right),
    \end{align*}   
    and 
    \begin{align*}
    |\II_{3,2}|&\le C\|u\|_{L^4(D_{\fz}^*)}^2\|\psi_{\fz}\Pi\Delta u\|\le\fa\|\psi_{\fz}\Pi\Delta u\|^2+C_\fa\|u\|_{L^4(D_{\fz}^*)}^4\\&\le\fa\|\psi_{\fz}\Pi\Delta u\|^2+C_\fa\|u\|_{L^2(D_{\fz}^*)}^2\|u\|_{H^1(D_{\fz}^*)}^2.\end{align*}
    The terms $\II_{3,3},\II_{3,4}$ are bounded by integrating by parts and using \eqref{SE6} and the Gagliardo--Nirenberg inequality
    \[\|f\|_4\lesssim \|f\|^{\frac{3}{4}}\|f\|_{H^2}^{\frac{1}{4}},\]
    which yields
    \begin{align*}
        |\II_{3,3}|&=|\langle u\otimes u,\nabla [\Pi,\psi_{\fz}^2]\Delta u\rangle|\le C\|\theta_{\varepsilon,\fz} u\|^2_4\|\theta_{\varepsilon,\fz}  u\|_{H^2}\\& \le C\|\theta_{\varepsilon,\fz} u\|^{\frac{3}{2}}\|\theta_{\varepsilon,\fz}  u\|^{\frac{3}{2}}_{H^2}\le \delta\|\theta_{\varepsilon,\fz}  u\|^{2}_{H^2} + C_\delta \|\theta_{\varepsilon,\fz} u\|^6,
    \end{align*}
    and 
    \begin{align*}
        |\II_{3,4}|&=\langle u\otimes u,\nabla \Pi(2\psi_{\fz}\psi_{\fz}'\partial_xu)\rangle\le C\|\theta_{\varepsilon,\fz} u\|^2_4\|\theta_{\varepsilon,\fz}  u\|_{H^2}\le \delta\|\theta_{\varepsilon,\fz}u\|^{2}_{H^2} + C_\delta \|\theta_{\varepsilon,\fz} u\|^6.
    \end{align*}
    Plugging the above estimates into \eqref{SE7} gives
    \begin{align}
        \label{SE8}
        |\II_3|&\le \fa\left(\|\psi_{\fz}\Pi\Delta u\|^2+\|\psi_{\fz}u\|_{H^2}^2\right)+\delta\|\theta_{\varepsilon,\fz}  u\|^{2}_{H^2}\notag\\&\qquad+C_\fa\|u\|_{L^2(D_{\fz}^*)}^2\|u\|_{H^1(D_{\fz}^*)}^2\left(1+\|\psi_{\fz}\nabla u\|^2+\|u\|_{L^2(D_{\fz}^*)}^2\right)+C_\delta \|\theta_{\varepsilon,\fz} u\|^6.
    \end{align}
    Next, we estimate 
    \begin{align*}\II_4:&=\langle\Pi (z\cdot \nabla) u,\dive (\psi_{\fz}^2\nabla u)\rangle\notag\\&=\langle \psi_{\fz}(z\cdot\nabla)u,\psi_{\fz}\Pi\Delta u\rangle+\langle(z\cdot\nabla)u,[\Pi,\psi_{\fz}^2]\Delta u\rangle +\langle (z\cdot\nabla)u,\Pi(2\psi_{\fz}\psi_{\fz}'\partial_xu)\rangle\notag\\&=:\II_{4,1}+\II_{4,2}+\II_{4,3},
    \end{align*}
    where, by proceeding similarly as in the derivation of \eqref{SE8}, one has 
    \begin{align*}
        |\II_{4,1}|\le\|z\|_\infty\|\psi_{\fz}\nabla u\|\|\psi_{\fz}\Pi\Delta u\|\le\fa\|\psi_{\fz}\Pi\Delta u\|^2+C_\fa\|z\|_\infty^2\|\psi_{\fz}\nabla u\|^2,
    \end{align*}
    and 
    \begin{align*}
        |\II_{4,2}|&=\left|\langle u\otimes z,\nabla[\Pi,\psi_{\fz}^2]\Delta u\rangle \right|\le C\|z\|_\infty\|\theta_{\varepsilon,\fz}u\|\|\theta_{\varepsilon,\fz}u\|_{H^2}\\&\le\delta\|\theta_{\varepsilon,\fz}u\|_{H^2}^2+C_\delta\|z\|_\infty^2\|\theta_{\varepsilon,\fz}u\|^2,
    \end{align*}
    and 
    \begin{align*}
       |\II_{4,3}|&=\left|\langle u\otimes z,\nabla\Pi(2\psi_{\fz}\psi_{\fz}'\partial_xu)\rangle \right|\le C\|z\|_\infty\|\theta_{\varepsilon,\fz}u\|\|\theta_{\varepsilon,\fz}u\|_{H^2}\\&\le\delta\|\theta_{\varepsilon,\fz}u\|_{H^2}^2+C_\delta\|z\|_\infty^2\|\theta_{\varepsilon,\fz}u\|^2.
    \end{align*}
    Consequently, it follows
    \begin{align}
        \label{SE9}
        |\II_4|\le \delta\|\theta_{\varepsilon,\fz}u\|_{H^2}^2+\fa\|\psi_{\fz}\Pi\Delta u\|^2+C_\delta\|z\|_\infty^2\|\theta_{\varepsilon,\fz}u\|^2+C_\fa\|z\|_\infty^2\|\psi_{\fz}\nabla u\|^2.
    \end{align}
    The term
    \begin{align*}
        \II_5:&=\langle\Pi (u\cdot \nabla) z,\dive (\psi_{\fz}^2\nabla u)\rangle\notag\\&=\langle \psi_{\fz}(u\cdot\nabla)z,\psi_{\fz}\Pi\Delta u\rangle+\langle(u\cdot\nabla)z,[\Pi,\psi_{\fz}^2]\Delta u\rangle +\langle (u\cdot\nabla)z,\Pi(2\psi_{\fz}\psi_{\fz}'\partial_xu)\rangle\notag\\&=:\II_{5,1}+\II_{5,2}+\II_{5,3}
    \end{align*}
    is estimated in a similar way, which gives
    \begin{align*}
        |\II_{5,1}|&\le \|\nabla z\|_4\|\psi_{\fz}u\|_4\|\psi_{\fz}\Pi\Delta u\|\le\fa\|\psi_{\fz}\Pi\Delta u\|^2+C_\fa\|\nabla z\|_4^2\|\psi_{\fz}u\|_4^2\\&\le\fa\|\psi_{\fz}\Pi\Delta u\|^2+C_\fa\|\nabla z\|_4^2\|\psi_{\fz}u\|\left(\|\psi_{\fz}\nabla u\|+\|u\|_{L^2(D_{\fz}^*)}\right),
    \end{align*}
    and 
    \begin{align*}
        |\II_{5,2}|&=\left|\langle z\otimes u,\nabla[\Pi,\psi_{\fz}^2]\Delta u\rangle \right|\le C\|z\|_\infty\|\theta_{\varepsilon,\fz}u\|\|\theta_{\varepsilon,\fz}u\|_{H^2}\\&\le\delta\|\theta_{\varepsilon,\fz}u\|_{H^2}^2+C_\delta\|z\|_\infty^2\|\theta_{\varepsilon,\fz}u\|^2,
    \end{align*}
    and 
    \begin{align*}
        |\II_{5,3}|&=\left|\langle z\otimes u,\nabla\Pi(2\psi_{\fz}\psi_{\fz}'\partial_xu)\rangle\right|\le C\|z\|_\infty\|\theta_{\varepsilon,\fz}u\|\|\theta_{\varepsilon,\fz}u\|_{H^2}\\&\le\delta\|\theta_{\varepsilon,\fz}u\|_{H^2}^2+C_\delta\|z\|_\infty^2\|\theta_{\varepsilon,\fz}u\|^2.
    \end{align*}
    Therefore, one has
    \begin{align}
        \label{SE10}
        |\II_5|\notag\le  \delta&\|\theta_{\varepsilon,\fz}u\|_{H^2}^2+\fa\|\psi_{\fz}\Pi\Delta u\|^2+C_\delta\|z\|_\infty^2\|\theta_{\varepsilon,\fz}u\|^2\\&\qquad\qquad\qquad+C_\fa\|\nabla z\|_4^2\|\psi_{\fz}u\|\left(\|\psi_{\fz}\nabla u\|+\|u\|_{L^2(D_{\fz}^*)}\right).
    \end{align}
    The last term 
    \begin{align*}
        \II_6:&=\langle\Pi (z\cdot \nabla) z,\dive (\psi_{\fz}^2\nabla u)\rangle\notag\\&=\langle \psi_{\fz}(z\cdot\nabla)z,\psi_{\fz}\Pi\Delta u\rangle+\langle(z\cdot\nabla)z,[\Pi,\psi_{\fz}^2]\Delta u\rangle +\langle (z\cdot\nabla)z,\Pi(2\psi_{\fz}\psi_{\fz}'\partial_xu)\rangle\notag\\&=:\II_{6,1}+\II_{6,2}+\II_{6,3}
    \end{align*}
    is bounded in the same way, which gives
    \begin{align*}
        |\II_{6,1}|&\le \|\psi_{\fz}\nabla z\|\|z\|_\infty\|\psi_{\fz}\Pi\Delta u\|\le\fa\|\psi_{\fz}\Pi\Delta u\|^2+C_\fa\|\psi_{\fz}\nabla z\|^2\|z\|^2_\infty,
    \end{align*}
    and 
    \begin{align*}
        |\II_{6,2}|&=\left|\langle z\otimes z,\nabla[\Pi,\psi_{\fz}^2]\Delta u\rangle \right|\le C\|\theta_{\varepsilon,\fz}z\|_4^2\|\theta_{\varepsilon,\fz}u\|_{H^2}\le\delta\|\theta_{\varepsilon,\fz}u\|_{H^2}^2+C_\delta\|\theta_{\varepsilon,\fz}z\|_4^4,
    \end{align*}
    and 
    \begin{align*}
        |\II_{6,3}|&=\left|\langle z\otimes z,\nabla\Pi(2\psi_{\fz}\psi_{\fz}'\partial_xu)\rangle\right|\le C\|\theta_{\varepsilon,\fz}z\|_4^2\|\theta_{\varepsilon,\fz}u\|_{H^2}\le\delta\|\theta_{\varepsilon,\fz}u\|_{H^2}^2+C_\delta\|\theta_{\varepsilon,\fz}z\|_4^4.
    \end{align*}
    As a result, it follows
    \begin{align}
    \label{SE11}
    |\II_6|\le\delta \|\theta_{\varepsilon,\fz}u\|_{H^2}^2+\fa\|\psi_{\fz}\Pi\Delta u\|^2+C_\delta\|\theta_{\varepsilon,\fz}z\|_4^4+C_\fa\|\psi_{\fz}\nabla z\|^2\|z\|^2_\infty.
    \end{align}
    Summing \eqref{SE8}--\eqref{SE11} and applying \eqref{SE1} and \eqref{SE2}, one gets
    \begin{align}\label{SE12}
    \ssum_{i=3}^6|\II_i|&\le \fa\left(\|\psi_{\fz}\Pi\Delta u\|^2+\|\psi_{\fz}u\|_{H^2}^2\right)+\delta\|\theta_{\varepsilon,\fz}  u\|^{2}_{H^2}+C_{\fa,\delta}\notag\\&\qquad\qquad+C_\fa\|u\|_{H^1(D_{\fz}^*)}^2\left(1+\|\psi_{\fz}\nabla u\|^2\right)+C_\fa\left(1+\|\psi_{\fz}\nabla u\|^2\right).
    \end{align}
    Plugging \eqref{SE5} and \eqref{SE12} into \eqref{SE3} and choosing $\fa$ sufficiently small imply
    \begin{align*}\frac{\ddd }{\ddd t}\|\psi_{\fz}\nabla u\|^2+\frac{\nu}{4}\left(\|\psi_{\fz}\Pi\Delta u\|^2+c\|\psi_\fz u\|^2_{H^2}\right)&\le \delta\|\theta_{\varepsilon,\fz}  u\|^{2}_{H^2}+C_{\nu,\delta}\left(1+\| u\|^2_{H^1(D^*_\fz)}\right)\\&\qquad +C_\nu\left(1+\| u\|^2_{H^1(D^*_\fz)}\right)\|\psi_{\fz}\nabla u\|^2,
    \end{align*}
    which gives
    \begin{align*}
        \frac{\ddd }{\ddd t}\left(t\|\psi_{\fz}\nabla u\|^2\right)+\frac{c\nu t}{4}\|\psi_\fz u\|^2_{H^2}&\le \delta t\|\theta_{\varepsilon,\fz}  u\|^{2}_{H^2}+C_{\nu,\delta}\left(1+\| u\|^2_{H^1(D^*_\fz)}\right)\\&\qquad \qquad+C_\nu\left(1+\| u\|^2_{H^1(D^*_\fz)}\right)t\|\psi_{\fz}\nabla u\|^2
    \end{align*}
    for any $t\in(0,t_0]$. An application of the Gronwall inequality then yields
    \begin{align}
         t\|\psi_{\fz}\nabla u(t)\|^2 +\frac{c\nu}{4}\int_0^t s\|\psi_\fz u\|^2_{H^2}\ddd s\le C_1\delta \int_0^t s\|\theta_{\varepsilon,\fz} u\|^2_{H^2}\ddd s+C_2,
    \end{align}
    where $C_1$ depends only on $\nu$ and the bound in \eqref{SE2}, and $C_2$ depends additionally on $\delta$. Together with repeating the convolution-type argument used in the derivation of \eqref{GWP25} and choosing $\delta$ sufficiently small, this gives
    \begin{align}
        \label{SE13}\sup_{s\in(0,t_0]}s\|u(s)\|_{H^1_b}^2+\sup_{\fz\in\R}\int_0^{t_0}s\|u(s)\|_{H^2(D_{\fz})}^2 \ddd s\le C.
    \end{align}

    \noindent\textbf{Step I\!I.} Next, we establish $L^\infty$-bounds. Due to the Gagliardo--Nirenberg inequality
    \begin{align}
        \label{SE13-0}\|f\|_\infty\lesssim \|f\|_{L^2_b}^{\frac{1}{5}}\|f\|_{W^{3/4,4}_b}^{\frac{4}{5}},
    \end{align}
    it suffices to bound the $W^{3/4,4}_b$-norm. To this end, we apply \eqref{STOKES1} and \eqref{STOKES2} combined with the interpolation inequality
    \begin{align}
        \label{SE13-1}\|f\|_{W^{3/4,4}_b}\lesssim \|f\|_{L^4_b}^{\frac14}\|f\|_{W^{1,4}_b}^{\frac34},
    \end{align}
    which yields
    \begin{align}\label{SE14}
    \|S(t)u\|_{W^{3/4,4}_b}\lesssim_\nu t^{-\frac38}\|u\|_{L^4_b},\qquad\|S(t)\Pi\dive f\|_{W^{3/4,4}_b}\lesssim_\nu t^{-\frac78}\|f\|_{L^4_b}.
    \end{align}
    Moreover, by \eqref{SE2}, \eqref{SE13}, and the Gagliardo--Nirenberg inequalities
    \[\|f\|_{L^4_b}\lesssim \|f\|_{L^2_b}^{\frac{1}{2}}\|f\|_{H^{1}_b}^{\frac{1}{2}},\qquad\|f\|_{L^8_b}\lesssim \|f\|_{L^2_b}^{\frac{1}{4}}\|f\|_{H^{1}_b}^{\frac{3}{4}},\]
    it follows
    \begin{align}\label{SE15}
    \|u(t)\|_{L^4_b}\lesssim_\nu t^{-\frac{1}{4}},\qquad\|u(t)\|_{L^8_b}\lesssim_\nu t^{-\frac{3}{8}},
    \end{align}
    which, together with \eqref{SE1}, yields
    \begin{align}\label{SE16}
    \|(u(t)+z(t))\otimes(u(t)+z(t))\|_{L^4_b}\lesssim_\nu t^{-\frac34}.
    \end{align}
    Now applying the estimates \eqref{SE14}--\eqref{SE16} to the Duhamel formula
    \[u(t)=S(t/2)u(t/2)-\int_{t/2}^t S(t-s)\Pi\dive\big((u(s)+z(s))\otimes(u(s)+z(s))\big)\ddd s\]
    yields
    \begin{align}\label{SE17}
    \|u(t)\|_{W^{3/4,4}_b}&\lesssim t^{-\frac38}\|u(t/2)\|_{L^4_b}+t^{-\frac34}\int_{t/2}^t(t-s)^{-\frac78}\ddd s\lesssim t^{-\frac58}.
    \end{align}
    Together with \eqref{SE2} and \eqref{SE13-0}, this implies 
    \[\|u(t)\|_\infty\lesssim t^{-\frac{1}{2}},\qquad\forall t\in(0,t_0].\]
    Combining this with \eqref{SE1} and \eqref{SE2}, we obtain \eqref{MIE13}.

    \noindent\textbf{Step I\!I\!I.} It remains to prove \eqref{MIE14}. To this end, we first derive some probability estimates for the Wiener process $W(t):=\int_0^t\eta(s)\ddd s$. Since $\FB_1<\infty$, for any $s\le t$, $W(t)-W(s)$ is a Gaussian random variable valued in $V$. This yields
    \[\PP \left\{\|W(t)-W(s)\|_V>\rho\right\}\lesssim \exp\left(-\frac{\rho^2}{4\FB_1|t-s|}\right).\]
    Applying this estimate to dyadic increments gives
    \begin{align}\label{SE20}
    \PP\left\{[W]_{C^{3/8}([0,1];V)}>\rho\right\}\lesssim \exp\left(-\frac{C\rho^2}{\FB_1}\right).
    \end{align}

    We turn to the proof of \eqref{MIE14}. By the stochastic integration by parts formula, see (2.102) in \cite{KS12}, one has
    \[\ZZ_t=e^{\nu tL}W(t)-\nu\int_0^t Le^{\nu(t-s)L}(W(t)-W(s)) \ddd s,\]
    which, together with the classical Stokes semigroup estimate
    \[\|L^\gamma e^{\nu t L}\|_{\LL(H)}\le C_{\gamma}t^{-\gamma},\]
    and the embedding $D(L^{3/4})\hookrightarrow H^{\frac32}$, leads to
    \begin{align*}
    \notag\|\ZZ_t\|_{H^{\frac32}}&\lesssim_{\nu} \|L^{3/4}e^{\nu tL}W(t)\|+\int_0^t\| L^{7/4}e^{\nu(t-s)L}(W(t)-W(s))\| \ddd s\\\notag
    &\lesssim_\nu t^{-\frac{1}{4}}\|W(t)\|_V+\int_0^t (t-s)^{-\frac54}\|W(t)-W(s)\|_V \ddd s
    \\&\lesssim_\nu  [W]_{C^{3/8}([0,t];V)}\left(t^{\frac18}+\int_0^t(t-s)^{-\frac78}\ddd s\right)\lesssim_\nu t^{\frac18}[W]_{C^{3/8}([0,t];V)}.
    \end{align*}
    Consequently,
    \begin{align}
        \label{SE21}\sup_{0\le s\le t}\|\ZZ_s\|_{H^{\frac32}}\lesssim_\nu t^{\frac18}[W]_{C^{3/8}([0,t];V)}.
    \end{align}
    By rescaling, $[W]_{C^{3/8}([0,t];V)}$ has the same law as $t^{1/8}[W]_{C^{3/8}([0,1];V)}.$
    Combining this identity with \eqref{SE20} and \eqref{SE21}, one gets \eqref{MIE14}.
    \end{proof}
    \vspace{.2in}
    \bibliographystyle{alpha}
    \bibliography{reference}

@article{KS17,
  author  = {Kuksin, Sergei and Shirikyan, Armen},
  title   = {Rigorous results in space-periodic two-dimensional turbulence},
  journal = {Physics of Fluids},
  volume  = {29},
  number  = {12},
  pages   = {125106},
  year    = {2017},
  doi     = {10.1063/1.4996545}
}

@article {NZ24,
    AUTHOR = {Nersesyan, V. and Zhao, M.},
     TITLE = {Exponential mixing for the white-Forced complex
              {G}inzburg--{L}andau equation in the whole space},
   JOURNAL = {SIAM J. Math. Anal.},
    VOLUME = {56},
      YEAR = {2024},
    NUMBER = {3},
     PAGES = {3646--3678}
}

@article {FM95,
    AUTHOR = {Flandoli, F. and Maslowski, B.},
     TITLE = {Ergodicity of the {$2$}-{D} {N}avier--{S}tokes equation under
              random perturbations},
   JOURNAL = {Comm. Math. Phys.},
    VOLUME = {172},
      YEAR = {1995},
    NUMBER = {1},
     PAGES = {119--141}
}

@article {KS02,
    AUTHOR = {Kuksin, S. and Shirikyan, A.},
     TITLE = {Coupling approach to white-forced nonlinear {PDE}s},
   JOURNAL = {J. Math. Pures Appl. (9)},
    VOLUME = {81},
      YEAR = {2002},
    NUMBER = {6},
     PAGES = {567--602}
}

@article{HM08,
  author  = {Hairer, M. and Mattingly, J. C.},
  title   = {Spectral gaps in {Wasserstein} distances and the {2D}
             stochastic {Navier--Stokes} equations},
  journal = {Ann. Probab.},
  volume  = {36},
  number  = {6},
  pages   = {2050--2091},
  year    = {2008},
  doi     = {10.1214/08-AOP392}
}

@article{BMO17,
  author  = {Brze\'{z}niak, Z. and Motyl, E. and Ondrejat, M.},
  title   = {Invariant measure for the stochastic {Navier--Stokes}
             equations in unbounded {2D} domains},
  journal = {Ann. Probab.},
  volume  = {45},
  number  = {5},
  pages   = {3145--3201},
  year    = {2017},
  doi     = {10.1214/16-AOP1133}
}

@article{EMS01,
    AUTHOR = {E, W. and Mattingly, J. C. and Sinai, Ya.},
     TITLE = {Gibbsian dynamics and ergodicity for the stochastically forced
              {N}avier--{S}tokes equation},
   JOURNAL = {Comm. Math. Phys.},
     VOLUME = {224},
      YEAR = {2001},
    NUMBER = {1},
     PAGES = {83--106}
}

@article {HM06,
    AUTHOR = {Hairer, M. and Mattingly, J. C.},
     TITLE = {Ergodicity of the 2{D} {N}avier--{S}tokes equations with
              degenerate stochastic forcing},
   JOURNAL = {Ann. of Math. (2)},
     VOLUME = {164},
      YEAR = {2006},
    NUMBER = {3},
     PAGES = {993--1032}
}

@book {KS12,
    AUTHOR = {Kuksin, S. and Shirikyan, A.},
     TITLE = {Mathematics of two-dimensional turbulence},
    SERIES = {Cambridge Tracts in Mathematics},
    VOLUME = {194},
 PUBLISHER = {Cambridge University Press, Cambridge},
      YEAR = {2012},
     PAGES = {xvi+320}
}

@incollection {D13,
    AUTHOR = {Debussche, A.},
     TITLE = {Ergodicity results for the stochastic {N}avier--{S}tokes
              equations: an introduction},
 BOOKTITLE = {Topics in mathematical fluid mechanics},
    SERIES = {Lecture Notes in Math.},
    VOLUME = {2073},
     PAGES = {23--108},
 PUBLISHER = {Springer, Heidelberg},
      YEAR = {2013}
}

@book {KS91,
    AUTHOR = {Karatzas, I. and Shreve, S. E.},
     TITLE = {Brownian motion and stochastic calculus},
    SERIES = {Graduate Texts in Mathematics, vol. 113},
 PUBLISHER = {Springer-Verlag, New York, Edition 2},
      YEAR = {2014},
}

@incollection {Shi08,
    AUTHOR = {Shirikyan, A.},
     TITLE = {Exponential mixing for randomly forced partial differential equations: method of coupling},
    SERIES = {Int. Math. Ser.},
     BOOKTITLE = {Instability in models connected with fluid flows {II}},
    VOLUME = {7},
     PAGES = {155--188},
 PUBLISHER = {Springer, New York},
      YEAR = {2008},
}

@article {BL19,
    AUTHOR = {Bakhtin, Y. and Li, L.},
     TITLE = {Thermodynamic limit for directed polymers and stationary solutions of the {B}urgers equation},
   JOURNAL = {Comm. Pure Appl. Math.},
    VOLUME = {72},
      YEAR = {2019},
    NUMBER = {3},
     PAGES = {536--619},
}

@article {BCK14,
    AUTHOR = {Bakhtin, Y. and Cator, E. and Khanin, K.},
     TITLE = {Space-time stationary solutions for the {B}urgers equation},
   JOURNAL = {J. Amer. Math. Soc.},
    VOLUME = {27},
      YEAR = {2014},
    NUMBER = {1},
     PAGES = {193--238},
}

@article {DGR21,
    AUTHOR = {Dunlap, A. and Graham, C. and Ryzhik, L.},
     TITLE = {Stationary solutions to the stochastic {B}urgers equation on
              the line},
   JOURNAL = {Comm. Math. Phys.},
    VOLUME = {382},
      YEAR = {2021},
    NUMBER = {2},
     PAGES = {875--949},
}

@article {N22,
    AUTHOR = {Nersesyan, V.},
     TITLE = {Ergodicity for the randomly forced {N}avier--{S}tokes system in
              a two-dimensional unbounded domain},
   JOURNAL = {Ann. Henri Poincar\'{e}},
     VOLUME = {23},
      YEAR = {2022},
    NUMBER = {6},
     PAGES = {2277--2294},
}

@article {BL06,
    AUTHOR = {Brze\'{z}niak, Z. and Li, Y.},
     TITLE = {Asymptotic compactness and absorbing sets for 2{D} stochastic
              {N}avier--{S}tokes equations on some unbounded domains},
   JOURNAL = {Trans. Amer. Math. Soc.},
    VOLUME = {358},
      YEAR = {2006},
    NUMBER = {12},
     PAGES = {5587--5629},
}

@article {KV14,
    AUTHOR = {Kukavica, I. and Vicol, V.},
     TITLE = {Moments for strong solutions of the 2{D} stochastic
              {N}avier--{S}tokes equations in a bounded domain},
   JOURNAL = {Asymptot. Anal.},
     VOLUME = {90},
      YEAR = {2014},
    NUMBER = {3-4},
     PAGES = {189--206},
}

@article{Zel07,
  author  = {Zelik, Sergey},
  title   = {Spatially nondecaying solutions of the {2D} {Navier--Stokes} equation in a strip},
  journal = {Glasgow Mathematical Journal},
  volume  = {49},
  number  = {3},
  pages   = {525--588},
  year    = {2007},
  doi     = {10.1017/S0017089507003849}
}

@article{AZ14,
  author  = {Anthony, Peter and Zelik, Sergey},
  title   = {Infinite-energy solutions for the {Navier--Stokes} equations in a strip revisited},
  journal = {Communications on Pure and Applied Analysis},
  volume  = {13},
  number  = {4},
  pages   = {1361--1393},
  year    = {2014},
  doi     = {10.3934/cpaa.2014.13.1361}
}

@article {FG95,
    AUTHOR = {Flandoli, F. and G\k{a}tarek, D.},
     TITLE = {Martingale and stationary solutions for stochastic
              {N}avier-{S}tokes equations},
   JOURNAL = {Probab. Theory Related Fields},
  FJOURNAL = {Probability Theory and Related Fields},
    VOLUME = {102},
      YEAR = {1995},
    NUMBER = {3},
     PAGES = {367--391},
}

@book {Bog07,
    AUTHOR = {Bogachev, V. I.},
     TITLE = {Measure theory. {V}ol. {I}, {II}},
 PUBLISHER = {Springer-Verlag, Berlin},
      YEAR = {2007},
}

@article{Jak98,
  author  = {Jakubowski, Adam},
  title   = {{The almost sure Skorokhod representation for subsequences in nonmetric spaces}},
  journal = {Theory of Probability and Its Applications},
  volume  = {42},
  number  = {1},
  pages   = {167--174},
  year    = {1998},
  doi     = {10.1137/S0040585X97976052}
}

@book{Str10,
  author    = {Stroock, Daniel W.},
  title     = {Probability Theory: An Analytic View},
  edition   = {2},
  publisher = {Cambridge University Press},
  address   = {Cambridge},
  year      = {2010},
  doi       = {10.1017/CBO9780511974243}
}

@misc{NZ24-1,
  author        = {Nersesyan, Vahagn and Zhao, Meng},
  title         = {Polynomial mixing for the white-forced
                   {Navier--Stokes} system in the whole space},
  year          = {2024},
  eprint        = {2410.15727},
  archivePrefix = {arXiv},
  primaryClass  = {math.AP},
  note          = {arXiv:2410.15727},
  url           = {https://arxiv.org/abs/2410.15727}
}

@article{AW05,
  author  = {Abels, Helmut and Wiegner, Michael},
  title   = {Resolvent estimates for the {Stokes} operator on an infinite layer},
  journal = {Differential and Integral Equations},
  volume  = {18},
  number  = {10},
  pages   = {1081--1110},
  year    = {2005},
  doi     = {10.57262/die/1356060107}
}

@article{Zel03,
  author  = {Zelik, Sergey V.},
  title   = {Attractors of Reaction-Diffusion Systems in Unbounded Domains and Their Spatial Complexity},
  journal = {Communications on Pure and Applied Mathematics},
  volume  = {56},
  number  = {5},
  pages   = {584--637},
  year    = {2003},
  doi     = {10.1002/cpa.10068}
}

@article{DPS87,
  author  = {Da Prato, Giuseppe and Sinestrari, Eugenio},
  title   = {Differential operators with non dense domain},
  journal = {Ann. Scuola Norm. Sup. Pisa Cl. Sci. (4)},
  year    = {1987},
  volume  = {14},
  number  = {2},
  pages   = {285--344},
  url     = {https://www.numdam.org/item/ASNSP_1987_4_14_2_285_0/}
}

@incollection{Sch04,
  author    = {Schnaubelt, Roland},
  title     = {Asymptotic Behaviour of Parabolic Nonautonomous Evolution Equations},
  booktitle = {Functional Analytic Methods for Evolution Equations},
  editor    = {Iannelli, Mimmo and Nagel, Rainer and Piazzera, Susanna},
  series    = {Lecture Notes in Mathematics},
  volume    = {1855},
  pages     = {401--472},
  publisher = {Springer},
  address   = {Berlin, Heidelberg},
  year      = {2004},
  doi       = {10.1007/978-3-540-44653-8_5}
}

@article{Ros98,
  author  = {Rosa, Ricardo},
  title   = {The global attractor for the {2D} {Navier--Stokes} flow on some unbounded domains},
  journal = {Nonlinear Analysis: Theory, Methods \& Applications},
  volume  = {32},
  number  = {1},
  pages   = {71--85},
  year    = {1998},
  doi     = {10.1016/S0362-546X(97)00453-7}
}

@article{BCLLLR13,
  author  = {Brze{\'z}niak, Z. and Caraballo, T. and Langa, J. A.
             and Li, Y. and {\L}ukaszewicz, G. and Real, J.},
  title   = {Random attractors for stochastic {2D}-{Navier--Stokes}
             equations in some unbounded domains},
  journal = {Journal of Differential Equations},
  volume  = {255},
  number  = {11},
  pages   = {3897--3919},
  year    = {2013},
  doi     = {10.1016/j.jde.2013.07.043}
}

@misc{CXZ26,
  author        = {Chen, Yuxuan and Xiang, Shengquan and Zhang, Zhifei},
  title         = {Exponential mixing for nonlinear {Schr{\"o}dinger}
                   equations perturbed by bounded degenerate noise},
  year          = {2026},
  eprint        = {2604.05911},
  archivePrefix = {arXiv},
  primaryClass  = {math.AP},
  note          = {arXiv:2604.05911},
  doi           = {10.48550/arXiv.2604.05911}
}

@article{CXZZ26,
  author  = {Chen, Yuxuan and Xiang, Shengquan and Zhang, Zhifei
             and Zhao, Jia-Cheng},
  title   = {Exponential mixing for the randomly forced {NLS} equation},
  journal = {Communications in Mathematical Physics},
  volume  = {407},
  pages   = {211},
  year    = {2026},
  doi     = {10.1007/s00220-026-05732-z}
}

@misc{LWXZZ24,
  author        = {Liu, Ziyu and Wei, Dongyi and Xiang, Shengquan
                   and Zhang, Zhifei and Zhao, Jia-Cheng},
  title         = {Exponential mixing for random nonlinear wave equations:
                   weak dissipation and localized control},
  year          = {2024},
  eprint        = {2407.15058},
  archivePrefix = {arXiv},
  primaryClass  = {math.AP},
  note          = {arXiv:2407.15058},
  doi           = {10.48550/arXiv.2407.15058}
}

@article{AM05,
  author  = {Afendikov, Andrei L. and Mielke, Alexander},
  title   = {Dynamical properties of spatially non-decaying {2D}
             {Navier--Stokes} flows with {Kolmogorov} forcing
             in an infinite strip},
  journal = {J. Math. Fluid Mech.},
  volume  = {7},
  number  = {suppl. 1},
  pages   = {S51--S67},
  year    = {2005},
  doi     = {10.1007/s00021-004-0131-9}
}

@article{Rou02,
  author  = {Rougemont, Jacques},
  title   = {Space-time invariant measures, entropy, and dimension for
             stochastic {Ginzburg--Landau} equations},
  journal = {Communications in Mathematical Physics},
  volume  = {225},
  number  = {2},
  pages   = {423--448},
  year    = {2002},
  doi     = {10.1007/s002200100586}
}

@article{EH01,
  author  = {Eckmann, Jean-Pierre and Hairer, Martin},
  title   = {Invariant measures for stochastic {PDE}'s in unbounded domains},
  journal = {Nonlinearity},
  volume  = {14},
  number  = {1},
  pages   = {133--151},
  year    = {2001},
  doi     = {10.1088/0951-7715/14/1/308}
}

@article{DDGLR23,
  author  = {Drivas, Theodore D. and Dunlap, Alexander and Graham, Cole
             and La, Joonhyun and Ryzhik, Lenya},
  title   = {Invariant measures for stochastic conservation laws on the line},
  journal = {Nonlinearity},
  volume  = {36},
  number  = {9},
  pages   = {4553--4584},
  year    = {2023},
  doi     = {10.1088/1361-6544/acdb3a}
}

@article{Zel13,
  author  = {Zelik, Sergey},
  title   = {Infinite Energy Solutions for Damped {Navier--Stokes}
             Equations in {$\mathbb{R}^2$}},
  journal = {Journal of Mathematical Fluid Mechanics},
  volume  = {15},
  number  = {4},
  pages   = {717--745},
  year    = {2013},
  doi     = {10.1007/s00021-013-0144-3}
}
\end{document}